\documentclass[12pt,reqno]{amsart}
\usepackage{subfiles}
\usepackage[a4paper,top=3cm,bottom=3cm,inner=4cm,outer=3cm]{geometry}
\usepackage{mathrsfs,amsmath,amsfonts,amsthm,amssymb,graphicx,pdfpages}
\usepackage{setspace,fancyhdr,calc,verbatim,bbm,enumerate,color,umoline}
\usepackage[all,2cell]{xy}\UseAllTwocells\SilentMatrices
\usepackage{tikz}
\usepackage{comment}
\definecolor{darkgreen}{rgb}{0,0.45,0}
\usepackage[colorlinks,citecolor=darkgreen,urlcolor=gray,final,hyperindex,linktoc=page]{hyperref}
\usepackage[capitalize]{cleveref}
\crefname{equation}{}{}
\crefname{item}{}{}
\usetikzlibrary{decorations.markings,arrows.meta,calc,fit,quotes,cd,math}
  
\tikzset{tick/.style={postaction={decorate,decoration={markings,
mark=at position 0.5 with {\draw[-] (0,.4ex) -- (0,-.4ex);}}}}}
\tikzset{tickx/.style={postaction={decorate,decoration={markings,mark=at position 0.5 with
{\fill circle [radius=.28ex];}}}}}

\newtheorem*{thm*}{Theorem}
\theoremstyle{remark}
\newtheorem*{rmk*}{Remark}
\newtheorem*{lem*}{Lemma}
\theoremstyle{definition}
\newtheorem*{defi*}{Definition}
\newtheorem*{cor*}{Corollary}
\theoremstyle{definition}
\newtheorem*{examples*}{Examples}
\newtheorem{prop*}{Proposition}

\theoremstyle{plain}
\newtheorem{thm}{Theorem}[section]
\theoremstyle{plain}
\newtheorem{prop}[thm]{Proposition}
\theoremstyle{remark}
\newtheorem{rmk}[thm]{Remark}
\theoremstyle{plain}
\newtheorem{lem}[thm]{Lemma}
\theoremstyle{plain}
\newtheorem{cor}[thm]{Corollary}
\theoremstyle{definition}
\newtheorem{defi}[thm]{Definition}
\theoremstyle{definition}

\DeclareFontFamily{U}{mathx}{\hyphenchar\font45}
\DeclareFontShape{U}{mathx}{m}{n}{
      <5> <6> <7> <8> <9> <10>
      <10.95> <12> <14.4> <17.28> <20.74> <24.88>
      mathx10
      }{}
\DeclareSymbolFont{mathx}{U}{mathx}{m}{n}
\DeclareFontSubstitution{U}{mathx}{m}{n}
\DeclareMathAccent{\widecheck}{0}{mathx}{"71}
\newcommand{\Mnd}{\B{Mnd}}
\newcommand{\Cmd}{\B{Cmd}}
\newcommand{\wc}{\widecheck}
\newcommand{\wh}{\widehat}
\newcommand{\VMat}{\ca{V}\textrm{-}\B{Mat}}
\newcommand{\VMMat}{\ca{V}\textrm{-}\MMat}
\newcommand{\ca}{\mathcal}
\newcommand{\caa}{\mathbb} 
\newcommand{\Hom}{\ensuremath{\mathrm{Hom}}}
\newcommand{\HOM}{\textrm{\scshape{Hom}}}
\newcommand{\Comod}{\ensuremath{\mathbf{Comod}}}
\newcommand{\Mod}{\ensuremath{\mathbf{Mod}}}
\newcommand{\Cat}{\B{Cat}}

\newcommand{\Alg}{\ensuremath{\mathbf{Alg}}}
\newcommand{\Coalg}{\ensuremath{\mathbf{Coalg}}}
\newcommand{\Mon}{\ensuremath{\mathbf{Mon}}}
\newcommand{\Comon}{\ensuremath{\mathbf{Comon}}}

\newcommand{\ob}{\ensuremath{\mathrm{ob}}}

\newcommand{\cod}{\mathrm{cod}}
\newcommand{\opl}{\mathrm{opl}}
\newcommand{\lax}{\mathrm{lax}}
\newcommand{\Cart}{\ensuremath{\mathrm{Cart}}}
\newcommand{\Cocart}{\ensuremath{\mathrm{Cocart}}}
\newcommand{\ps}{\mathscr}
\newcommand{\B}{\mathbf}
\newcommand{\Gr}{\mathfrak}
\newcommand{\Mat}{\mathbf{Mat}}

\DeclareMathOperator*\colim{colim}
\newcommand{\op}{\mathrm{op}}

\newcommand{\MMat}{\mathbf{\caa{M}at}}
\newcommand{\id}{\mathrm{id}}
\newcommand{\HD}{\ca{H}(\caa{D})}
\newcommand{\sbul}{\scriptstyle\bullet}
\newcommand{\tick}{\object@{|}}
\newcommand{\feta}{\dot{\eta}}
\newcommand{\fepsilon}{\dot{\varepsilon}}
\newcommand{\Dendo}{\caa{D}_1^\bullet}
\newcommand{\Eendo}{\caa{E}_1^\bullet}
\newcommand{\VCocat}{\ca{V}\textrm{-}\B{Cocat}}
\newcommand{\VCat}{\ca{V}\textrm{-}\B{Cat}}
\newcommand{\VMod}{\ca{V}\textrm{-}\B{Mod}}
\newcommand{\VComod}{\ca{V}\textrm{-}\B{Comod}}
\newcommand{\VGrph}{\ca{V}\textrm{-}\B{Grph}}
\newcommand{\Fendo}{F_1^\bullet}
\newcommand{\Dar}{\caa{D}_1}
\newcommand{\Set}{\B{Set}}
\newcommand{\matr}[3]{\SelectTips{eu}{10}\xymatrix@C=.2in{#1\colon #2\ar[r]|-{\object@{|}} & #3}}
\newcommand{\proar}[3]{\SelectTips{eu}{10}\xymatrix@C=.2in{#1\colon #2\ar[r]|-{\sbul} & #3}}
\newcommand{\simrightarrow}{\xrightarrow{\raisebox{-4pt}[0pt][0pt]{\ensuremath{\sim}}}}

\begin{document}
\title[Enriched duality in double categories]{Enriched duality in double categories: $\ca{V}$-categories and $\ca{V}$-cocategories*}

\author{Christina Vasilakopoulou} 
\thanks{*This is a provisory draft of a paper whose definitive version is due to be published in the
``Journal of Pure and Applied Algebra".}
\address{Departement of Mathematics, University of California, Riverside, 900 University Avenue, 92521, USA}
\email{vasilak@ucr.edu}

\begin{abstract}
In this work, we explore a double categorical framework for categories of enriched graphs, categories
and the newly introduced notion of cocategories. A fundamental goal is to establish an enrichment of $\ca{V}$-categories
in $\ca{V}$-cocategories, which generalizes the so-called Sweedler theory relatively to an enrichment of algebras
in coalgebras. The language employed is that of $\ca{V}$-matrices, and an interplay between the double categorical
and bicategorical perspective provides a high-level flexibility for demonstrating essential features of these dual structures.
Furthermore, we investigate an enriched fibration structure involving categories of monads and comonads
in fibrant double categories a.k.a. proarrow equipments,
which leads to natural applications for their materialization as categories and cocategories in the enriched matrices setting.
\end{abstract}

\maketitle

\setcounter{tocdepth}{1}
\tableofcontents

\section{Introduction}

In \cite{Measuringcomonoid}, an enrichment of the category of monoids in comonoids is established in a braided monoidal closed category
$\ca{V}$ which is moreover locally presentable, induced by a generalization of Sweedler's \emph{universal measuring coalgebra}
$P(A,B)$ for algebras $A$, $B$ \cite{Sweedler}. This constitutes an abstract framework for the so-called \emph{Sweedler theory} of algebras
and coalgebras \cite{AnelJoyal} in differential graded vector spaces, leading to an efficient formalism for the bar-cobar
adjunction in a broader effort to conceptually clarify the Koszul duality for (co)algebras \cite{AlgebraicOperads}.

The present work generalizes this result to its many-object setting; introducing the notion of a $\ca{V}$-\emph{enriched cocategory}
which reduces to a comonoid in $\ca{V}$ when its set of objects is singleton, we establish an enrichment of the category of $\ca{V}$-categories
in $\ca{V}$-cocategories. This enrichment can be realized under the same assumptions on $\ca{V}$, and shares all fundamental
characteristics with Sweedler theory for (co)monoids. In particular, the braided monoidal closed $\VCocat$ acts on
the monoidal $\VCat$ via a convolution-like functor which exhibits its adjoint,
the \emph{generalized Sweedler hom} $T\colon\VCat^\op\times\VCat\to\VCocat$ as the enriched hom-functor. Moreover,
the enrichment in question is tensored and cotensored, via the \emph{generalized Sweedler product}
$\triangleright\colon\VCocat\times\VCat\to\VCat$ and the action respectively.

The framework that brings all these structures together is that of a double category of $\ca{V}$-matrices, $\VMMat$. Its horizontal bicategory
$\VMat$ is well-studied even in the more abstract case of matrices enriched in bicategories \cite{VarThrEnr}, and our approach follows
its one-object case as well as \cite{KellyLack} in viewing $\ca{V}$-categories as monads therein and establishing important
categorical properties. It turns out that working on the double categorical context offers a much clearer perspective
for the categories of interest and their interrelations; for example, $\ca{V}$-functors are precisely \emph{monad morphisms}
(vertical monad maps in \cite{Monadsindoublecats}) in $\VMMat$ whereas only a special case of monad maps in $\VMat$.

For that purpose, we present a detailed framework for monads and comonads in arbitrary double categories, explore their (op)fibrational
structure in the fibrant case \cite{Framedbicats} and push the enrichment objective as far as possible. This way,
demanding calculations involving enriched graphs, categories and cocategories (some of them found in \cite[\S 7]{PhDChristina}) are deduced
from natural properties of monads and comonads in monoidal fibrant double categories, which are furthermore
\emph{locally closed monoidal}. By introducing this concept which endows the vertical and horizontal categories with a
monoidal closed structure, we obtain an action of the category of monads on comonads which under certain assumptions induces
the desired enrichment.

Finally, we are also interested in combining such an enrichment with the natural (op)fibred structure of monads and comonads. As a result,
we first recall some general fibred adjunction results as well as some basic \emph{enriched fibration} machinery from \cite{Enrichedfibration}
and subsequently apply it to the double categorical setting for the monoidal (op)fibrations of (co)monads, leading to respective results
for $\ca{V}$-categories and $\ca{V}$-cocategories in $\VMMat$. It is expected that this abstract picture
shall also allow for applications for other (co)monads in double categories of similar flavor,
indicatively for colored operads and cooperads.

A next step to this development will be to consider categories of \emph{modules} and \emph{comodules} for monads
and comonads in double categories and look for similar enrichments, this time for the fibration $\VMod\to\VCat$
over an appropriately defined $\VComod\to\VCocat$ in $\VMMat$. This would again provide a many-object generalization of the enrichment
of $\Mod$ in $\Comod$, i.e. global categories of (co)modules for (co)monoids, established in \cite{Measuringcomodule}.

A brief outline of the paper is as follows. \cref{sec:Background} assembles all the necessary background material in order to make
this paper as self-contained as possible. This includes selected facts about bicategories, (co)monoids in monoidal categories
and local presentability aspects, the theory of actions inducing enrichments, the universal measuring comonoid
and the theory of fibrations and enriched fibrations. In \cref{doublecats}, after giving some basic double
categorical definitions, we explore the framework for monads and comonads.
Considering monoidal and fibrant double categories also with a locally closed monoidal structure, we furthermore combine it
with the enriched fibration theory. Finally, \cref{sec:Enrichedmatrices} applies all the previous results to the double category
$\VMMat$ of $\ca{V}$-matrices. By establishing necessary properties of monads ($\ca{V}$-categories) and comonads ($\ca{V}$-cocategories)
therein, an enrichment between them is exhibited also on the (op)fibration level as the ultimum result. In the process, a detailed
exposition of the structures involved and their dual relations gathers known along with newly established facts about these fundamental categories,
also generalizing classical properties for (co)monoids in monoidal categories.

\section{Preliminaries}\label{sec:Background}

In this section, we gather all the necessary material for what follows. This includes some basic
bicategorical notions, elements of the theory of monoidal categories with focus on the categories of monoids and comonoids
and local presentability aspects, as well as parts of the theory of actions of monoidal categories inducing enrichment relations.
Finally, we recall some results from related work concerning an enrichment of monoids in comonoids, as well as the recently introduced
enriched fibration structure; the goal is to later fit those in a double categorical context which serves as the common framework
for our objects of interest.

In order to restrain the length of the paper, we provide appropriate references
whenever definitions and constructions are only sketched. 
The choice of how detailed certain material review is, solely relies on what is specifically used later in the paper,
with the purpose of making this work as self-sufficient as possible. This is why, for example, bicategories and functors between
them are spelled out, as opposed to monoidal categories and functors.

\subsection{Bicategories}\label{sec:bicategories}

The original definition of a bicategory and a lax functor between bicategories can be found in B{\'e}nabou's 
\cite{Benabou}. Other references, including the definitions of transformations and modifications are 
\cite{Categoricalstructures,Handbook1}. 
Categories of (co)monads in bicategories are carefully recalled;
regarding 2-category theory, indicative references are \cite{Review,2-catcompanion},
whereas \cite{FormalTheoryMonadsI} presents the formal theory of monads in 2-categories.
Due to coherence for bicategories, we are often able to use 2-categori\-cal machinery like pasting and mates 
correspondence directly in the weaker context. 

\begin{defi}\label{def:bicategory}
A \emph{bicategory} $\ca{K}$ is specified by objects $A,B,...$ called \emph{0-cells},
and for each pair of objects a category $\ca{K}(A,B)$, whose objects are called \emph{1-cells} and whose arrows
are called \emph{2-cells}; vertical composition of $2$-cells is denoted
\begin{displaymath}
\xymatrix @C=.8in
{A\ar @/^4ex/[r]^-f \ar[r]|-g \ar @/_4ex/[r]_-h \rtwocell<\omit>{<-2>\alpha} \rtwocell<\omit>{<2>\;\alpha'} &
B}=\xymatrix @!=.5in {A\rtwocell^f_h{\;\;\;\alpha'\cdot\alpha} & B.}
\end{displaymath}
and the identity 2-cell is $1_f\colon f\Rightarrow f\colon A\to B$.
Moreover, for each triple of objects there is the \emph{horizontal composition} functor $\circ:\ca{K}(B,C)\times\ca{K}(A,B)\to\ca{K}(A,C)$
which maps a pair of 1-cells $(g,f)$ to $g\circ f=gf$ and a pair of 2-cells
\begin{displaymath}
\xymatrix @!=.5in {A\rtwocell^f_u{\alpha} & 
B\rtwocell^g_v{\beta} & C}=
\xymatrix @!=.5in {A\rtwocell<\omit>{\;\;\;\beta*\alpha}
\ar @/^3ex/[r]^-{gf}
\ar @/_3ex/[r]_{vu} & C.}
\end{displaymath}
Finally, for each object we have the \emph{identity 1-cell} $1$-cell $1_A:A\to A$.

The associativity and identity constraints are expressed via the \emph{associator} with components invertible 2-cells
$\alpha_{h,g,f}:(h\circ g)\circ f\simrightarrow h\circ(g\circ f)$ and the \emph{unitors} by $\lambda_f:1_B\circ f\simrightarrow f,$
$\rho_f:f\circ 1_A\simrightarrow f.$
The above are subject to coherence conditions: for
$\SelectTips{10}{eu}\xymatrix@C=.3in{A\ar[r]|-{\;f\;} & B\ar[r]|-{\;g\;} & C\ar[r]|-{\;h\;} & D\ar[r]|-{\;k\;} & E}$, the following commute
\begin{equation}\label{bicataxiom1}
\xymatrix @R=.2in@C=.01in
{((k{\circ}  h){\circ} g){\circ} f\ar[d]_-{\alpha_{kh,g,f}}\ar[rr]^-{\alpha_{k,h,g}*1_f} && (k{\circ}(h{\circ} g)){\circ} f,\ar[d]^-{\alpha_{k,hg,f}} \\
(k{\circ} h){\circ}(g{\circ} f)\ar[dr]_-{\alpha_{k,h,gf}} && k{\circ}((h{\circ} g){\circ} f) \ar[dl]^-{1_k*\alpha_{h,g,f}}\\
& k{\circ}(h{\circ}(g{\circ} f)) &}\quad
\xymatrix@R=.6in@C=.2in
{(g{\circ} 1_B){\circ} f\ar[rr]^-{\alpha_{g,1_B,f}}\ar[dr]_-{\rho_g*1_f} && g{\circ}(1_B{\circ} f)\ar[dl]^-{1_g*\lambda_f} \\
& g{\circ} f &}
\end{equation}
\end{defi}

By functoriality of the horizontal composition we have $1_g\circ 1_f=1_{g\circ f}$ and
$(\beta'\cdot\beta)*(\alpha'\cdot\alpha)=(\beta'*\alpha')\cdot(\beta*\alpha)$, the latter known as the \emph{interchange law}. 

Given a bicategory $\ca{K}$, we may reverse only the $1$-cells and form the bicategory $\ca{K}^\mathrm{op}$,
with $\ca{K}^\mathrm{op}(A,B)=\ca{K}(B,A)$. We may also reverse only the $2$-cells and form the bicategory $\ca{K}^\mathrm{co}$
with $\ca{K}^\mathrm{co}(A,B)=\ca{K}(A,B)^\mathrm{op}$. Reversing both $1$-cells and $2$-cells yields 
a bicategory $(\ca{K}^\mathrm{co})^\mathrm{op}=(\ca{K}^\mathrm{op})^\mathrm{co}$.

There are numerous examples of well-known bicategories. Indicatively,
\begin{itemize}
 \item $\B{Span}(\ca{C})$ for any category $\ca{C}$ with pullbacks has objects the ones in $\ca{C}$, 1-cells spans $A\leftarrow M\rightarrow B$
 and 2-cells span morphisms; 
 \item $\B{Rel}(\ca{C})$ for any regular category $\ca{C}$ is defined as $\B{Span}(\ca{C})$ but with 1-cells relations
 $R\rightarrowtail X\times Y$, and composition is given by first taking the pullback and then performing epi-mono factorization;
 \item $\B{BMod}$ has objects rings, 1-cells bimodules and 2-cells bimodule maps;
 \item $\ca{V}$-$\B{Prof}$ has objects $\ca{V}$-categories, 1-cells $\ca{V}$-profunctors ($\ca{V}$-bimodules)
 $F\colon\ca{B}^\op\times\ca{A}\to\ca{V}$ and 2-cells appropriate $\ca{V}$-natural transformations. 
\end{itemize}

\begin{defi}\label{laxfunctor}
Given bicategories $\ca{K}$ and $\ca{L}$, a  \emph{lax functor} $\ps{F}:\ca{K}\to\ca{L}$ consists of a mapping
on objects $A\mapsto\ps{F}A$, a functor $\ps{F}_{A,B}:\ca{K}(A,B)\to\ca{L}(\ps{F}A,\ps{F}B)$
for every $A,B\in\ca{K}$, a natural transformation with components $\delta_{g,f}:(\ps{F}g)\circ(\ps{F}f)\to
\ps{F}(g\circ f)$ for any composable 1-cells, and a natural transformation with components
$\gamma_A:1_{\ps{F}A}\to\ps{F}(1_A)$ for every $A\in\ca{K}$.

The natural transformations $\gamma$ and $\delta$ have to satisfy the following coherence axioms: for 1-cells
$\SelectTips{10}{eu}\xymatrix@C=.3in{A\ar[r]|-{\;f\;} & B\ar[r]|-{\;g\;} & C\ar[r]|-{\;h\;} & D}$, the following diagrams commute:
\begin{equation}\label{laxcond1}
\xymatrix @C=.6in @R=.25in
{(\ps{F}h\circ\ps{F}g)\circ\ps{F}f\ar[r]^-{\delta_{h,g}*1}\ar[d]_-{\alpha} & \ps{F}(h\circ g)\circ\ps{F}f\ar[d]^-{\delta_{hg,f}} \\
\ps{F}h\circ(\ps{F}g\circ\ps{F}f) \ar[d]_-{1*\delta_{g,f}} & \ps{F}((h\circ g)\circ f)\ar[d]^-{\ps{F}\alpha} \\
\ps{F}h\circ\ps{F}(g\circ f)\ar[r]_-{\delta_{h,gf}} & \ps{F}(h\circ(g\circ f))}
\end{equation}

\begin{equation}\label{laxcond2}
\xymatrix @C=.5in @R=.25in
{1_{\ps{F}B}\circ\ps{F}f\ar[r]^-{\gamma_B*1} \ar[d]_-\lambda & \ps{F}(1_B)\circ\ps{F}f \ar[d]^-{\delta_{1_B,f}} \\
\ps{F}f & \ps{F}(1_B\circ f)\ar[l]^-{\ps{F}\lambda}}\quad
\xymatrix @C=.5in @R=.25in
{\ps{F}f\circ 1_{\ps{F}A}\ar[r]^-{1*\gamma_A}\ar[d]_-\rho & \ps{F}f\circ\ps{F}(1_A)\ar[d]^-{\delta_{f,1_A}} \\
\ps{F}f & \ps{F}(f\circ 1_A)\ar[l]^-{\ps{F}\rho}}
\end{equation}
\end{defi}

If $\gamma$ and $\delta$ are natural isomorphisms (respectively identities), then $\ps{F}$ is called a \emph{pseudofunctor}
or \emph{homomorphism} (respectively \emph{strict functor}) of bicategories. Similarly, we can define a \emph{colax 
functor} of bicategories by reversing the direction of $\gamma$ and $\delta$, sometimes also called \emph{oplax}.
We obtain categories $\B{Bicat}_l$, $\B{Bicat}_c$, $\B{Bicat}_{ps}$, $\B{Bicat}_s$ with objects bicategories
and arrows lax, colax, pseudo and strict functors respectively.

A \emph{monoidal bicategory} $\ca{K}$ is a bicategory equipped with a pseudofunctor $\otimes\colon\ca{K}\times\ca{K}\to\ca{K}$
which is coherently associative and has an identity object $I$. The explicit definition with all
appropriate diagrams can be found in any of the standard references, e.g. \cite[\S 2.1]{Carmody}. In our examples,
we will choose to establish monoidal structure of bicategories via the more general structure of a monoidal double category,
as explained in \cref{doublecats}; arguably, that technique is easier to apply given certain assumptions.

\begin{defi}\label{laxnattrans}
If $\ps{F},\ps{G}:\ca{K}\to\ca{L}$ are two lax functors, a \emph{lax natural transformation} $\tau:\ps{F}\Rightarrow\ps{G}$ consists of
morphisms $\tau_A:\ps{F}A\to\ps{G}A$ in $\ca{L}$, along with natural transformations
\begin{equation}\label{nattranslax}
\xymatrix @C=.4in @R=.4in
{\ca{K}(A,B)\ar[r]^-{\ps{F}_{A,B}}\ar[d]_-{\ps{G}_{A,B}} &
\ca{L}(\ps{F}A,\ps{F}B)\ar[d]^-{\ca{L}(1,\tau_B)} \\
\ca{L}(\ps{G}A,\ps{G}B)\ar[r]_-{\ca{L}(\tau_A,1)} &
\ca{L}(\ps{F}A,\ps{G}B)\ultwocell<\omit>{\tau}}
\end{equation}
with components 2-cells $\tau_f\colon\tau_B\circ\ps{F}f\Rightarrow \ps{G}f\circ\tau_A$.
This data is subject to standard axioms expressing the compatibility of $\tau$ with composition and units, using
$\delta$ and $\gamma$ of the lax functors.
\end{defi}

A transformation $\tau$ is \emph{pseudonatural} (respectively
\emph{strict}) when all the components $\tau_f$ of \cref{nattranslax} are isomorphisms (respectively identities).
Also, an \emph{oplax} natural transformation is equipped with a natural transformation in the opposite direction
of \cref{nattranslax}. Note that between either lax or oplax functors of bicategories,
we can consider both lax and oplax natural transformations.

Along with \emph{modifications} between transformations (see \cite{Handbook1}) we can form four different functor bicategories of
combinations of (op)lax functors, (op)lax natural transformations and modifications, e.g. $\B{Bicat}_{l,l}(\ca{K},\ca{L})$, which contain
$\B{Bicat}_{ps,l}(\ca{K},\ca{L})$, $\B{Bicat}_{ps,opl}(\ca{K},\ca{L})$ 
and $\B{Hom}(\ca{K},\ca{L})$ of pseudofunctors and lax/oplax/pseudo natural transformations as sub-bicategories.

Now a (strict) \emph{2-category} is a bicategory in which all constraints are identities, \emph{i.e.} $\alpha,\rho,\lambda=1$.
In this case, the horizontal composition is strictly associative and unitary and the axioms \cref{bicataxiom1} hold automatically. Consequently, 
the collection of 0-cells and 1-cells form a category on its own. Note that
when $\ca{L}$ is a 2-category, all the above functor bicategories are also 2-categories.

\begin{examples*}\hfill
\begin{enumerate}
 \item $\B{Cat}$ of (small) categories, functors and natural transformations;
\item $\Mon\B{Cat}$ of monoidal categories, (strong) monoidal functors and mo\-noi\-dal natural transformations;
\item $\ca{V}$-$\B{Cat}$ of $\ca{V}$-enriched categories, $\ca{V}$-functors and $\ca{V}$-na\-tu\-ral tra\-nsfo\-rma\-tions
for a monoidal category $\ca{V}$;
\item $\B{Fib}(\caa{X})$ and $\B{OpFib}(\caa{X})$ of fibrations and opfibrations over $\caa{X}$, 
(op)fibred functors and (op)fibred natural transformations (see \cref{fibrations});
\item $\B{Cat}(\caa{E})$ of categories internal to $\caa{E}$, for a finitely complete category. Instances of this are ordinary categories
($\caa{E}=\B{Set}$), double categories ($\caa{E}=\B{Cat}$) and crossed modules ($\caa{E}=\B{Grp}$).
\end{enumerate}
\end{examples*}



We now turn to notions of monads and comonads in bicategories.
\begin{defi}\label{monadbicat}
A \emph{monad} in a bicategory $\ca{K}$ consists of an object $B$ together with an endomorphism
$t:B\to B$ and 2-cells $\eta:1_B\Rightarrow t$, $m:t\circ t\Rightarrow t$ called the \emph{unit} and \emph{multiplication},
such that the following diagrams commute:
\begin{displaymath}
\xymatrix @R=.25in
{(t\circ t)\circ t\ar[rr]^-{\alpha_{t,t,t}}
\ar[d]_-{m\circ1} && t\circ(t\circ t)\ar[d]^-{1\circ m} \\
t\circ t\ar[dr]_-m && t\circ t\ar[ld]^-m \\
& t &}\qquad\mathrm{and}\qquad
\xymatrix @R=.65in @C=.5in
{1_B\circ t\ar[r]^-{\eta\circ 1}
\ar[dr]_-{\lambda_t} & t\circ t
\ar[d]^-{m} & t\circ1_B \ar[l]_-{1\circ\eta}
\ar[ld]^-{\rho_t} \\
& t &}
\end{displaymath}
\end{defi}

Equivalently, a monad in a bicategory $\ca{K}$ is a lax functor $\ps{F}:\B{1}\to\ca{K}$, where $\B{1}$ is the terminal bicategory with 
a unique 0-cell $\star$. This amounts to an object $\ps{F}(\star)=B\in\ca{K}$ and a functor $\ps{F}_{\star,\star}:\B{1}(\star,\star)\to\ca{K}(B,B)$
which picks up an endoarrow $t:B\to B$. The natural transformations $\delta$ and $\gamma$ of the lax functor give
the multiplication and the unit of $t$
\begin{displaymath}
m\equiv\delta_{1_\star,1_\star}:t\circ t\to t
\quad\textrm{and}\quad
\eta\equiv\gamma_{\star}:1_B\to t
\end{displaymath}
and the axioms for $\ps{F}$ give the monad axioms for $(t,m,\eta)$.

\begin{rmk}\label{laxfunctorspreservemonads}
If $\ps{G}:\ca{K}\to\ca{L}$ is a lax functor between bicategories, the composite
\begin{displaymath}
\B{1}\xrightarrow{\;\ps{F}\;}\ca{K}\xrightarrow{\;\ps{G}\;}\ca{L}
\end{displaymath}
is itself a lax functor from $\B{1}$ to $\ca{L}$, hence defines a monad. In other words, if $t:B\to B$ is a monad in the bicategory
$\ca{K}$, then $\ps{G}t:\ps{G}B\to\ps{G}B$ is a monad in the bicategory $\ca{L}$, \emph{i.e.} lax functors preserve monads.
\end{rmk}

\begin{defi}\label{monadfunctor}
A \emph{(lax) monad functor} between two monads $t:B\to B$ and $s:C\to C$ in a bicategory
consists of an 1-cell $f:B\to C$ between the 0-cells of the monads together with a 2-cell
\begin{displaymath}
 \xymatrix
{B\ar[r]^-f \ar[d]_-t \drtwocell<\omit>{\psi} & C\ar[d]^-s \\
B\ar[r]_-f & C}
\end{displaymath}
satisfying compatibility conditions with multiplications and units.
\end{defi}
If the 2-cell $\psi$ is in the opposite direction, and the diagrams are accordingly modified, we have a \emph{colax} monad functor
(or monad \emph{opfunctor}) between two monads. Along with appropriate notions of monad natural transformations (see \cite{FormalTheoryMonadsI}),
we obtain a bicategory $\B{Mnd}(\ca{K})\equiv[\B{1},\ca{K}]_l$.

Dually to the above, and for future reference, we have the following.
\begin{defi}\label{comonadbicat}
A \emph{comonad} in a bicategory $\ca{K}$ consists of an object $A$ together with an endoarrow $u:A\to A$ and 2-cells
$\Delta:u\Rightarrow u\circ u$, $\varepsilon:u\Rightarrow 1_A$ called the \emph{comultiplication} and \emph{counit} respectively,
such that the following commute 
\begin{displaymath}
\xymatrix @R=.25in
{& u\ar[dr]^-\Delta \ar[dl]_-\Delta & \\
u\circ u \ar[d]_-{\Delta\circ1} && u\circ u\ar[d]^-{1\circ\Delta} \\
(u\circ u)\circ u\ar[rr]_-{\alpha_{u,u,u}} && u\circ(u\circ u)}\qquad
\xymatrix @R=.65in @C=.5in
{1_A\circ u \ar[dr]_-{\lambda_u} & u\circ u \ar[l]_-{\varepsilon\circ 1}\ar[r]^-{1\circ\varepsilon} & u\circ1_A\ar[ld]^-{\rho_u} \\
& u\ar[u]_-{\Delta} &}
\end{displaymath}
\end{defi}
Notice that a comonad in the bicategory $\ca{K}$ is precisely a monad in the bicategory $\ca{K}^\mathrm{co}$;
along with colax comonad functors and comonad natural transformation, we have the bicategory $\Cmd(\ca{K})=[\B{1},\ca{K}]_c$.

\subsection{Monoids and comonoids in monoidal categories}\label{Monoidalcats}

Standard references for the theory of monoidal categories, as well as monoids and comonoids, are for example \cite{BraidedTensorCats,Quantum};
here we recall only a few things necessary for later constructions. In any case, monoidal categories and 
(co)lax/strong/strict functors between them are just the one-object cases of Definitions \ref{def:bicategory} and \ref{laxfunctor}.

Suppose $(\ca{V},\otimes,I)$ is a monoidal category. A \emph{monoid} is an object $A$ equipped with
a multiplication and unit $m:A\otimes A\rightarrow A\leftarrow I:\eta$ that satisfy usual associativity
and unit laws; along with monoid morphisms that preserve the structure in that $f\circ(m\otimes m)=m'\circ f$
and $f\circ\eta=\eta'$, they form a category $\Mon(\ca{V})$.
Dually, we have \emph{comonoids} $(C,\Delta\colon C\to C\otimes C,\epsilon\colon C\to I)$ whose category
is denoted by $\Comon(\ca{V})$. Both these categories are monoidal
only if $\ca{V}$ is braided, and they also inherit the braiding or symmetry from $\ca{V}$.

\begin{rmk}\label{monadsaremonoids}
For any object $B$ in a bicategory $\ca{K}$, the hom-category $\ca{K}(B,B)$ is equipped with a monoidal structure
induced by the horizontal composition of the bicategory, namely $f\otimes g=g\circ f$ and $I=1_B$.
Then, a monoid in $(\ca{K}(B,B),\circ,1_B)$ is precisely a monad in $\ca{K}$ (\cref{monadbicat}) and dually, a comonad
$u:A\to A$ in a bicategory $\ca{K}$ is a comonoid in the monoidal $\ca{K}(A,A)$. 
\end{rmk}

It is well-known that lax monoidal functors between monoidal categories induce functors between their category of monoids,
as below.

\begin{prop}\label{monf}
If $F\colon\ca{V}\to\ca{W}$ is a lax monoidal functor, with structure maps $\phi_{A,B}\colon FA\otimes FB\to F(A\otimes B)$
and $\phi_0\colon I\to F(I)$, it induces a map between their categories of monoids
$$\Mon F\colon\Mon(\ca{V})\to\Mon(\ca{W})$$
by $(A,m,\eta)\mapsto(FA,Fm\circ\phi_{A,A},F\eta\circ\phi_0)$. Dually, colax functors induce maps between the categories of comonoids.
\end{prop}

It is also well-known that (due to the \emph{doctrinal adjunction}) colax monoidal structures on left adjoints correspond bijectively to lax
monoidal structures on right adjoints between monoidal categories; this generalizes to parametrized adjunctions, e.g.
\cite[3.2.3]{PhDChristina} or for higher dimension in \cite[Prop.~2]{Monoidalbicatshopfalgebroids}.

When $\ca{V}$ is braided monoidal closed, the tensor product functor has
a strong monoidal structure via $A\otimes B\otimes A'\otimes B'\simrightarrow A\otimes A'\otimes B\otimes B'$,
$I\simrightarrow I\otimes I$. Therefore the internal hom functor $[-,-]\colon\ca{V}^\op\times\ca{V}\to\ca{V}$
obtains a lax monoidal structure as its parametrized adjoint. The induced functor between the monoids is denoted
\begin{equation}\label{defMon[]}
\Mon[-,-]\colon\Comon(\ca{V})^\op\times\Mon(\ca{V})\to\Mon(\ca{V}); 
\end{equation}
for $C$ a comonoid and $A$ a monoid, $[C,A]$ has the \emph{convolution} monoid structure.

Turning to other properties of the categories of monoids and comonoids, for any $\ca{V}$
there exist forgetful $S\colon\Mon(\ca{V})\to\ca{V}$, $U\colon\Comon(\ca{V})\to\ca{V}$;
when these have a left or right adjoint respectively, they are called \emph{free monoid} and \emph{cofree comonoid}
functors. Evidently, the free monoid one is quite frequent.

\begin{prop}\label{freemonoidprop}
Suppose that $\ca{V}$ is a monoidal category with countable coproducts which are preserved by $\otimes$ on either side.
The forgetful $\Mon(\ca{V})\to\ca{V}$ has a left adjoint $L$, and the free monoid on an object $X$ is given by
\begin{displaymath}
LX=\coprod_{n\in\mathbb{N}}{X^{\otimes n}}.
\end{displaymath}
\end{prop}

On the other hand, the existence of the cofree comonoid is more problematic, and has been studied from various authors
mainly in the context of vector spaces or modules over a commutative ring. We are interested in Porst's approach
\cite{FundConstrCoalgCorComod,MonComonBimon} which in particular focuses on local presentability properties
inherited from $\ca{V}$.

Recall that an \emph{accessible} category $\ca{C}$ is one with a small set 
of $\kappa$-presentable objects $C$ (i.e. $\ca{C}(C,-)$ preserves $\kappa$-filtered colimits) such that every object
in $\ca{C}$ is the $\kappa$-filtered colimit of presentable objects, for some regular cardinal $\kappa$. A functor
between accessible functors is \emph{accessible} if it preserves $\kappa$-filtered colimits.
A \emph{locally presentable} category is an accessible category which is cocomplete. 
More on the theory of locally presentable categories can be found in the standard \cite{LocallyPresentable}.
An important fact is that any cocontinuous functor from a locally presentable category has a right adjoint;
this can be seen as a corollary to the following adjoint functor theorem, since presentable objects
form a small dense subcategory of $\ca{C}$.

\begin{thm}\cite[5.33]{Kelly}\label{Kellyadj}
If the cocomplete $\ca{C}$ has a small dense subcategory, every
cocontinuous $S:\ca{C}\to\ca{B}$ has a right adjoint.
\end{thm}

Going back to monoids and comonoids, the following result establishes their local presentability under certain assumptions.
We then briefly sketch parts of the proof because it will be later generalized, \cref{VCocatlocpresent}.

\begin{prop}~\cite[2.6-2.7]{MonComonBimon}\label{moncomonadm}
Suppose $\ca{V}$ is a locally presentable mo\-noi\-dal category, such that $\otimes$ preserves filtered colimits
in both variables.

$(1)$ $\Mon(\ca{V})$ is finitary monadic over $\ca{V}$ and locally presentable.

$(2)$ $\Comon(\ca{V})$ is a locally presentable category and comonadic over $\ca{V}$.
\end{prop}

The proof uses categories of \emph{functor algebras} and \emph{coalgebras} $\Alg F$ and $\Coalg F$ for any endofunctor
$F\colon\ca{C}\to\ca{C}$, namely objects $A$ equipped with plain arrows $FA\to A$ or $A\to FA$ in $\ca{C}$
and morphisms that commute with them. Their most important properties are the following.

\begin{lem}\label{functoralgebrasprops}
For any endofunctor $F\colon\ca{C}\to\ca{C}$,
\begin{enumerate}
 \item\label{one} $\Alg F\to\ca{C}$ creates all limits and those colimits preserved by $F$;
 \item\label{two} $\Coalg F\to\ca{C}$ creates all colimits and those limits preserved by $F$;
 \item\label{three} if $\ca{C}$ is locally presentable and $F$ preserves filtered colimits, $\Alg F$ and $\Coalg F$ are locally presentable.
\end{enumerate}
\end{lem}
Notably, these categories can be expressed as specific \emph{inserters} $\Alg F=\B{Ins}(F,\id_\ca{C})$ and $\Coalg F=\B{Ins}(\id_\ca{C},F)$
and \cref{three} then follows from the more general `Limit Theorem' \cite[5.1.6]{MakkaiPare}.

For the endofunctors with mappings $T_+(C)=(C\otimes C)+I$ and $T_\times(C)=(C\otimes C)\times I$,
$\Mon(\ca{V})$ is a complete full subcategory of the locally presentable and finitary monadic over $\ca{V}$ category $\Alg T_+$,
and $\Comon(\ca{V})$ is a cocomplete full subcategory of the locally presentable and comonadic over $\ca{V}$ $\Coalg T_\times$.

Specifically for comonoids, local presentability is deduced by expressing it as an \emph{equifier} of a triple of natural transformations
between accessible functors. Then comonadicity follows: in the commutative triangle
\begin{equation}\label{diagforComoncomonadicity}
\xymatrix @C=.5in @R=.2in
{\Comon(\ca{V})\ar@{-->}[dr]_U\ar@{^(->}[r] & \Coalg T_\times \ar[d] \\ & \ca{V}}
\end{equation}
where all categories are locally presentable, both forgetful functors to $\ca{V}$ have a right adjoint 
by Theorem \ref{Kellyadj}, since they are cocontinuous. Moreover, the right leg is comonadic by \cref{two},
and the inclusion preserves and reflects all limits. Therefore it creates equalizers of split pairs and so does $U$,
which then satisfies the conditions of Precise Monadicity Theorem.
In particular, the existence of the \emph{cofree comonoid functor} $R:\ca{V}\to\Comon(\ca{V})$ is established.

Another piece of structure inherited from $\ca{V}$ to $\Comon(\ca{V})$ in the locally presentable context is monoidal closedness,
again obtained from \cref{Kellyadj} for an adjoint of $-\otimes C\colon\Comon(\ca{V})\to\Comon(\ca{V})$.

\begin{prop}\cite[3.2]{MonComonBimon}\label{Comonmonclosed}
If $\ca{V}$ is a locally presentable braided monoidal closed category, $\Comon(\ca{V})$ is also monoidal closed.
\end{prop}

\subsection{Universal measuring comonoid}\label{sec:actionenrich}

One of the basic goals of \cite{Measuringcomonoid} was to estabilsh an enrichment
of the category of monoids in the category of comonoids, under certain assumptions on $\ca{V}$.
Below we summarize certain results; details can be found in Sections 4 and 5 therein.

Recall \cite{AnoteonActions} that an \emph{action} of a monoidal category on an ordinary one is given by a functor
$*\colon\ca{V}\times\ca{D}\to\ca{D}$ expressing that $\ca{D}$ is a pseudomodule for the pseudomonoid $\ca{V}$ in
the monoidal 2-category $(\B{Cat},\times,\B{1})$. In more detail, we have natural isomorphisms with components
\begin{equation}\label{actionmaps}
\chi_{X,Y,D}\colon(X\otimes Y)*D\simrightarrow X*(Y*D)\;\textrm{ and }\;\nu_{D}\colon I*D\simrightarrow D
\end{equation}
satisfying compatibility conditions. If $*$ is an action, then $*^\op$ is an action too.

As a central example for our purposes, we have the action of the opposite monoidal category on itself via the internal hom,
see \cite[3.7\&5.1]{Measuringcomonoid}.

\begin{lem}\label{inthomaction}
Suppose $\ca{V}$ is a braided monoidal closed category. The internal hom $[-,-]:\ca{V}^\mathrm{op}\times\ca{V}\to\ca{V}$
constitutes an action of $\ca{V}^\mathrm{op}$ on $\ca{V}$. Moreover, the induced $\Mon[-,-]:\Comon(\ca{V})^\mathrm{op}\times
\Mon(\ca{V})\to\Mon(\ca{V})$ \cref{defMon[]} is an action of the monoidal $\Comon(\ca{V})^\mathrm{op}$
on $\Mon(\ca{V})$. Similarly for their opposite functors.
\end{lem}

A very important fact is that given a category $\ca{D}$ with an action from a monoidal category $\ca{V}$ with a parametrized
adjoint, we obtain a $\ca{V}$-enriched category. This follows from a much stronger result of \cite{enrthrvar}
for categories enriched in bicategories; details can be found in \cite{AnoteonActions} and \cite[\S 4.3]{PhDChristina}.
For the explicit definitions of (co)tensored enriched categories, see \cite[3.7]{Kelly}.

\begin{thm}\label{actionenrich}
Suppose that $\ca{V}$ is a monoidal category which acts on a category $\ca{D}$ via a functor 
$*:\ca{V}\times\ca{D}\to\ca{D}$, such that $-*D$ has a right adjoint $F(D,-)$ for every $D\in\ca{D}$ with a natural isomorphism
\begin{displaymath}
\ca{D}(X*D,E)\cong\ca{V}(X,F(D,E)).
\end{displaymath}
Then we can enrich $\ca{D}$ in $\ca{V}$, in the sense that there is a $\ca{V}$-category $\underline{\ca{D}}$
with hom-objects $\underline{\ca{D}}(A,B)=F(A,B)$ and underlying category $\ca{D}$.

Moreover, if $\ca{V}$ is monoidal closed, the enrichment is \emph{tensored}, with $X*D$ the tensor of $X{\in}\ca{V}$ and $D{\in}\ca{D}$.
If $\ca{V}$ is moreover braided, the enrichment is \emph{cotensored} if $X*-$ has a right adjoint;
finally, we can also enrich $\ca{D}^\op$ in $\ca{V}$.
\end{thm}

By \cref{inthomaction}, the internal hom induces specific actions; their parametri\-zed adjoints
will induce the desired enrichment of monoids in comonoids. The following follows from \cref{Kellyadj}
applied to the locally presentable $\Comon(\ca{V})$ by \cref{moncomonadm}.

\begin{thm}\cite[Thm~4.1]{Measuringcomonoid}\label{measuringcomonoidprop}
If $\ca{V}$ is a locally presentable braided monoidal closed category, the functor
$\Mon[-,B]^\op\colon\Comon(\ca{V})\to\Mon(\ca{V})^\op$ has a right adjoint $P(-,B)$, i.e. there is a natural isomorphism
\begin{displaymath}
\Mon(\ca{V})(A,[C,B])\cong\Comon(\ca{V})(C,P(A,B)).
\end{displaymath}
\end{thm}

The parametrized adjoint of the functor $\Mon[-,-]^\op$, namely
\begin{equation}\label{Sweedlerhom}
P\colon\Mon(\ca{V})^\op\times\Mon(\ca{V})\to\Comon(\ca{V})
\end{equation}
is called the \emph{Sweedler hom}, and $P(A,B)$ is called the \emph{universal measuring comonoid},
generalizing Sweedler's measuring coalgebras of \cite[\S VII]{Sweedler} as well as the \emph{finite
dual} $P(A,I)=A^o$ of an algebra $A$.

Moreover, each $\Mon[C,-]^\op$ turns out to also have a right adjoint $(C\triangleright-)^\op$
\cite[\S 6.2]{PhDChristina}, and the induced functor of two variables
\begin{equation}\label{Sweedlerprod}
\triangleright\colon\Comon(\ca{V})\times\Mon(\ca{V})\to\Mon(\ca{V})
\end{equation}
is called the \emph{Sweedler product} in \cite{AnelJoyal}.

Applying \cref{actionenrich} to the braided monoidal closed
$\Comon(\ca{V})$ (\cref{Comonmonclosed})
and $\ca{D}=\Mon(\ca{V})^\op$, we obtain the following \cite[Thm.~5.2]{Measuringcomonoid}.

\begin{thm}\label{monoidenrichment}
Suppose $\ca{V}$ is locally presentable and braided monoidal closed.
\begin{enumerate}
 \item The category $\Mon(\ca{V})^\op$ is a monoidal $\Comon(\ca{V})$-category, tensored and cotensored, with hom-objects
 $\underline{\Mon(\ca{V})^\op}(A,B)=P(B,A)$.
 \item The category $\Mon(\ca{V})$ is a monoidal $\Comon(\ca{V})$-category, tensored and cotensored,
 with $\underline{\Mon(\ca{V})}(A,B)=P(A,B)$, cotensor $[C,B]$ and tensor $C\triangleright B$ for any comonoid $C$ and monoid $B$.
\end{enumerate}
\end{thm}

\subsection{Fibrations}\label{fibrations}

We recall some basic facts and constructions from the theory of fibrations and opfibrations. 
A few indicative references for the general theory are \cite{FibredAdjunctions,Handbook2,Jacobs},
and for our specific context \cite[\S 5]{PhDChristina}.

A functor $P\colon\ca{A}\to\caa{X}$ is a \emph{fibration} when for every arrow $f\colon X\to Y$
in the \emph{base} category $\caa{X}$ and every object $B$ in the \emph{total}
category $\ca{A}$ above $Y$, i.e. $P(B)=Y$, there exists a \emph{cartesian
morphism} with codomain $B$ above $f$. If we denote it $\phi\colon A\to B$, this means that for any
$g\colon X\to X'$ and $A'\to B$ as in the diagram below, there exists a unique factorization through the domain of the cartesian morphism
over $g$:
\begin{displaymath}
\xymatrix @R=.1in @C=.6in
{A'\ar [drr]^-{\theta}\ar @{-->}[dr]_-{\exists!\psi} 
\ar @{.>}@/_/[dd] &&& \\
& A\ar[r]_-{\phi} \ar @{.>}@/_/[dd] & 
B \ar @{.>}@/_/[dd] & \textrm{in }\ca{A}\\
X'\ar [drr]^-{f\circ g=P\theta}\ar[dr]_-g &&&\\
& X\ar[r]_-{f=P\phi} & Y & \textrm{in }\caa{X}}
\end{displaymath}
For each $X$ in the base, its \emph{fibre} category $\ca{A}_X\subset\ca{A}$ consists of objects above $X$
and morphisms above the identity $\id_X$, called $P$-\emph{vertical}. We call $\phi$ a \emph{cartesian lifting} of $B$ along $f$.
Assuming the axiom of choice, cartesian liftings can be selected (up to vertical isomorphism) for each morphism
$f$ in the base and object $B\in\ca{A}_{\cod f}$, henceforth denoted $\Cart(f,B)\colon f^*B\to B$.

Dually, a functor $U\colon\ca{C}\to\caa{X}$ is an \emph{opfibration} when its opposite functor $U^\op$ is a fibration:
for every $g\colon X\to Y$ in the base and $C\in\ca{C}_X$ above the domain, there is a cocartesian morphism
from $C$ above $g$, the \emph{cocartesian lifting} of $C$ along $g$ denoted $\Cocart(g,C)\colon C\to g_!C$.

Any arrow in the total category of an (op)fibration factorizes uniquely into
a vertical morphism followed by a (co)cartesian one:
\begin{equation}\label{factor}
\xymatrix @C=.4in @R=.2in
{A\ar[rr]^\theta\ar @{-->}[d]_-{\psi} && B \ar @{.>}[dd] &\\
f^*B\ar[urr]_-{\;\Cart(f,B)} \ar @{.>}[d] &&& \textrm{in }\ca{A} \\
X\ar[rr]_-{f} && Y & \textrm{in }\caa{X},}\qquad
\xymatrix @C=.4in @R=.2in
{C \ar @{.>}[dd]\ar[rr]^\gamma \ar[drr]_-{\Cocart(g,C)} && D &\\
&& f_!C \ar @{-->}[u]_-{\delta} \ar @{.>}[d] & \textrm{in }\ca{C} \\
X\ar[rr]_-{g} && Y & \textrm{in }\caa{X}.}
\end{equation}
The choice of (co)cartesian liftings in an (op)fibration induces
a so-called \emph{reindexing functor} between the fibre categories
\begin{displaymath}
f^*:\ca{A}_Y\to\ca{A}_X\quad\textrm{ and }\quad g_!\colon\ca{C}_X\to\ca{C}_Y
\end{displaymath}
respectively, for each morphism $f$ or $g\colon X\to Y$ in the base category, mapping each object to the (co)domain of its lifting.

\begin{rmk}\label{rmkforadjointintexingbifr}
Due to the unique factorization of arrows in a (op)fibration through (co)cartesian liftings, we can deduce that a fibration
$P:\ca{A}\to\caa{X}$ is also an opfibration (consequently a \emph{bifibration}) if and only if, for every $f:X\to Y$ the reindexing
$f^*:\ca{A}_Y\to\ca{A}_X$ has a left adjoint, namely $f_!:\ca{A}_X\to\ca{A}_Y$ (e.g. \cite[Proposition 1.2.7]{hermidaphd}). 
\end{rmk}

An \emph{oplax fibred 1-cell} $(S,F)$ between $P:\ca{A}\to\caa{X}$ and $Q:\ca{B}\to\caa{Y}$ is given by a commutative square of
categories and functors
\begin{displaymath}
\xymatrix @C=.4in @R=.4in
{\ca{A}\ar[r]^-S \ar[d]_-P & \ca{B}\ar[d]^-Q \\
\caa{X}\ar[r]_-F & \caa{Y}}
\end{displaymath}
called an \emph{oplax morphism of fibrations} in \cite[Def.~3.5]{Framedbicats}; this name is justified by the correspondence
of \cref{Grothendieckcorrespondence}.
By \cref{factor}, we always have a comparison vertical morphism
\begin{displaymath}
\xymatrix @C=.6in @R=.2in
{Sf^*B\ar[rr]^-{S\Cart(f,B)}\ar @{-->}[d]_-{\psi} && SB \ar @{.>}[dd] &\\
(Ff)^*SB\ar[urr]_-{\;\Cart(Ff,SB)} \ar @{.>}[d] &&& \textrm{in }\ca{B} \\
FX\ar[rr]_-{Ff} && FY & \textrm{in }\caa{Y},}
\end{displaymath}
to the chosen $Q$-cartesian lifting of $SB$ along $Ff$.
If moreover $S$ preserves cartesian arrows, meaning that if $\phi$ is $P$-cartesian then $S\phi$ is $Q$-cartesian
or equivalently the above comparison map is an isomorphism, the pair $(S,F)$ is called a \emph{fibred 1-cell} or
\emph{strong morphism of fibrations}.

In particular, when $P$ and $Q$ are fibrations over the same base $\caa{X}$, we may consider
oplax morphisms or fibred 1-cells of the form $(S,1_{\caa{X}})$ displayed as
\begin{displaymath}
\xymatrix @C=.2in
{\ca{A}\ar[rr]^-S \ar[dr]_-P
 && \ca{B}\ar[dl]^-Q\\
 & \caa{X} &}
\end{displaymath}
when $S$ is called \emph{(oplax) fibred functor}.
Dually, we have the notion of an \emph{lax opfibred 1-cell} $(K,F)$, \emph{opfibred 1-cell} when $K$ is cocartesian,
and \emph{(lax) opfibred functor} $(K,1_\caa{X})$.
Notice that any oplax fibred 1-cell $(S,F)$ determines a collection of functors
\begin{displaymath}
S_X\colon\ca{A}_X\longrightarrow\ca{B}_{FX}
\end{displaymath}
between the fibres, as the restriction of $S$ to the corresponding subcategories.

A \emph{fibred 2-cell} between oplax fibred 1-cells $(S,F)$ and $(T,G)$ is a pair of natural transformations 
($\alpha:S\Rightarrow T,\beta:F\Rightarrow G$) with $\alpha$ above $\beta$, \emph{i.e.} $Q(\alpha_A)
=\beta_{PA}$ for all $A\in\ca{A}$, displayed
\begin{displaymath}
\xymatrix @C=.8in @R=.5in
{\ca{A}\rtwocell^S_T{\alpha}\ar[d]_-P
& \ca{B}\ar[d]^-Q \\
\caa{X}\rtwocell^F_G{\beta} & \caa{Y}.}
\end{displaymath}
A \emph{fibred natural transformation} is of the form $(\alpha,1_{1_{\caa{X}}}):(S,1_{\caa{X}})\Rightarrow(T,1_\caa{X})$
which ends up having vertical components, $Q(\alpha_A)=1_{PA}$.
Notice that if the 1-cells are strong, the definition of a 2-cell between them remains the same.
Dually, we have the notion of an \emph{opfibred 2-cell} and \emph{opfibred natural transformation} between
lax opfibred 1-cells and functors respectively.

We obtain 2-categories $\B{Fib}_\opl$ and $\B{Fib}$ of fibrations over arbitrary base categories, (oplax) fibred 1-cells and
fibred 2-cells. In particular, there are 2-categories $\B{Fib}_\opl(\caa{X})$ and $\B{Fib}(\caa{X})$
of fibrations over a fixed base category $\caa{X}$, (oplax) fibred functors and fibred natural transformations.
Dually, we have the 2-categories $\B{OpFib}_{(\lax)}$ and $\B{OpFib}_{(\lax)}(\caa{X})$.

The fundamental \emph{Grothendieck construction} \cite{Grothendieckcategoriesfibrees} establishes a standard
equivalence between fibrations and pseudofunctors, \cref{laxfunctor}.
Starting with a pseudofunctor $\ps{M}\colon\caa{X}^\op\to\Cat$, we can form the Grothendieck category
$\int\ps{M}$ with objects pairs $(A,X)\in\ps{M}X\times\caa{X}$ and morphisms
$(A,X)\to(B,Y)$ pairs $(\phi\colon A\to(\ps{M}f)B,f\colon X\to Y)\in\ps{M}X\times\caa{X}$. This is fibred
over $\caa{X}$, with fibres $(\int\ps{M})_X=\ps{M}X$, reindexing functors $\ps{M}f$ and chosen cartesian liftings
\begin{equation}\label{canonicalcartesianlift}
\xymatrix@C=.3in @R=.3in
{((\ps{M}f)B,X)\ar[rr]^-{(1_{(\ps{M}f)B},f)}\ar@{.>}[d] && (B,Y)\ar@{.>}[d] & \textrm{in }\int\ps{M} \\
X\ar[rr]^-f && Y & \textrm{in }\caa{X}.}
\end{equation}
Using similar machinery \cite[Prop. 3.6]{Framedbicats}, we obtain respective correspondences for oplax fibred 1-cells and oplax natural
transformations.

\begin{thm}\label{Grothendieckcorrespondence}
There are equivalences of 2-categories
\begin{gather*}
\B{Fib}_\opl(\caa{X})\simeq[\caa{X}^\op,\B{Cat}]_\mathrm{opl} \\
\B{Fib}(\caa{X})\simeq[\caa{X}^\mathrm{op},\B{Cat}]
\end{gather*}
between the 2-categories of fibrations with fixed base and pseudofunctors with oplax or pseudonatural transformations and modifications.
\end{thm}
There is also a 2-equivalence $\B{ICat}\simeq\B{Fib}$ between fibrations over arbitrary bases and an appropriately defined 2-category of
pseudofunctors with arbitrary domain; for more details, see \cite{hermidaphd}. Along with the dual versions for opfibrations,
namely $\B{OpFib}(\caa{X})\simeq[\caa{X},\B{Cat}]$,
these equivalences allow us to freely change our perspective between (op)fibrations and pseudofunctors.

Moving on to notions of adjunctions between fibrations, we obtain the following definitions as adjunctions
in the respective 2-categories of (op)fibrations.

\begin{defi}\label{generalfibredadjunction}
Given fibrations $P:\ca{A}\to\caa{X}$ and $Q:\ca{B}\to\caa{Y}$,
a \emph{general (oplax) fibred adjunction} is given by a pair of (oplax) fibred 1-cells $(L,F):P\to Q$ and 
$(R,G):Q\to P$ together with fibred 2-cells $(\zeta,\eta):(1_\ca{A},1_\caa{X})\Rightarrow
(RL,GF)$ and  $(\xi,\varepsilon):(LR,FG)\Rightarrow(1_\ca{B},1_\caa{Y})$
such that $L\dashv R$ via $\zeta,\xi$ and $F\dashv G$ via $\eta,\varepsilon$. This
is displayed as
\begin{displaymath}
\xymatrix @C=.7in @R=.4in
{\ca{A} \ar[d]_-P 
\ar @<+.8ex>[r]^-L\ar@{}[r]|-\bot
& \ca{B} \ar @<+.8ex>[l]^-{R} \ar[d]^-Q \\
\caa{X} \ar @<+.8ex>[r]^-F\ar@{}[r]|-\bot
& \caa{Y} \ar @<+.8ex>[l]^-G}
\end{displaymath} 
and we write $(L,F)\dashv(R,G):Q\to P$.
\end{defi}

Notice that by definition, $\zeta$ is above $\eta$ and $\xi$ is above $\varepsilon$,
hence $(P,Q)$ is in particular an ordinary map between adjunctions. Dually, we have the notions of \emph{general (lax) opfibred 
adjunction} and \emph{opfibred adjunction} in $\B{OpFib}_\lax$. 

In \cite[\S 3.2]{Measuringcomodule}, conditions under which a fibred 1-cell has an adjoint are investigated in detail.
Below we only recall the case of general lax opfibred adjunction, due to the applications that follow.

\begin{thm}\label{totaladjointthm}
Suppose $(K,F):U\to V$ is an opfibred 1-cell and $F\dashv G$ is an adjunction with counit $\varepsilon$ between 
the bases of the opfibrations, as in
\begin{displaymath}
\xymatrix @C=.6in
{\ca{C}\ar[r]^-K\ar[d]_-U & \ca{D}\ar[d]^-V \\
\caa{X}\ar @<+.8ex>[r]^-F
\ar@{}[r]|-\bot
& \caa{Y}. \ar @<+.8ex>[l]^-G}
\end{displaymath}
If, for each $Y\in\caa{Y}$, the composite functor between the fibres
\begin{equation}\label{specialfunctor}
\ca{C}_{GY}\xrightarrow{K_{GY}}\ca{D}_{FGY}
\xrightarrow{(\varepsilon_Y)_!}\ca{D}_Y
\end{equation}
has a right adjoint for each $Y\in\caa{Y}$, then $K$ has a right adjoint $R$ between the total categories
and $(K,F)\dashv(R,G)$ is a general lax opfibred adjunction. 
\end{thm}

Finally, in \cite{Enrichedfibration} the notion of an enriched fibration is discussed in length;
we gather the basic definitions in order to employ them in the setting
of double categories later.
The following generalizes the notion of a (right) parametrized adjunction from $\Cat$ to $\B{Fib}_{\mathrm{opl}}$
or $\B{OpFib}_\lax$.

\begin{defi}\label{generaloplaxparametrized}
For fibrations $H,K$, a \emph{fibred parametrized adjunction} consists of
\begin{displaymath}
\xymatrix @C=.6in @R=.3in
{\ca{A}\times\ca{B}\ar[r]^-F\ar[d]_-{H\times J} & \ca{C}\ar[d]^-K \\ \caa{X}\times\caa{Y}\ar[r]_-G & \caa{Z},}\qquad
\xymatrix @C=.6in @R=.3in
{\ca{B}^\op\times\ca{C}\ar[r]^-R\ar[d]_-{J^\op\times K} & \ca{A}\ar[d]^-H \\
\caa{Y}^\op\times\caa{Z}\ar[r]_-S & \caa{X}}
\end{displaymath}
such the following is a general oplax fibred adjunction
\begin{displaymath}
 \xymatrix @C=.9in @R=.5in
{\ca{A}\ar[d]_-H\ar@<+.8ex>[r]^-{F(-,B)}\ar@{}[r]|-{\bot} & \ca{C}\ar[d]^-K\ar@<+.8ex>[l]^-{R(B,-)} \\
\caa{X}\ar@<+.8ex>[r]^-{G(-,JB)}\ar@{}[r]|-{\bot} & \caa{Z}.\ar@<+.8ex>[l]^-{S(JB,-)}}
\end{displaymath}
Dually, an \emph{opfibred parametrized adjunction} consists of 1-cells as above, inducing a general lax opfibred
adjunction. 
\end{defi}

Note that by general arguments, $R(B,-)$ automatically preserves cartesian morphisms and dually,
$F(-,B)$ preserves cocartesian morphisms.

Identifying the pseudomonoids in the cartesian monoidal 2-category $\B{Fib}$, we obtain the following definition,
also \cite[12.1]{Framedbicats}.

\begin{defi}\label{monoidalfibration}
A fibration $T\colon\ca{V}\to\caa{W}$ is \emph{monoidal} when $\ca{V}$, $\caa{W}$ are monoidal categories,
$T$ is a strict monoidal functor and the tensor product $\otimes_\ca{V}$ preserves cartesian arrows.
\end{defi}

A fibration is \emph{symmetric monoidal} when it is a strict braided monoidal functor
between symmetric monoidal categories. Next, expressing a pseudomodule in $(\B{Fib},\times,1_\ca{I})$ gives the following,
\cite[Def.~3.3]{Enrichedfibration}.

\begin{defi}\label{Trepresentation}
A monoidal fibration $T:\ca{V}\to\caa{W}$ \emph{acts} on the fibration
$P:\ca{A}\to\caa{X}$ when there exists a fibred 1-cell
\begin{equation}\label{eq:fibred1cellaction}
 \xymatrix @C=.6in@R=.3in
{\ca{V}\times\ca{A}\ar[r]^-{*}\ar[d]_-{T\times P} & 
\ca{A}\ar[d]^-P \\
\caa{W}\times\caa{X}\ar[r]_-{\diamond} & \caa{X}.}
\end{equation}
where the functors $*\colon\ca{V}\times\ca{A}\to\ca{A}$ and $\diamond\colon\caa{W}\times\caa{X}\to\caa{X}$
are ordinary actions of the monoidal $\ca{V}$, $\caa{W}$ on $\ca{A}$ and $\caa{X}$ respectively,
such that the action constraints are compatible
in the sense that
\begin{equation}\label{actionsabove}
P\chi^\ca{A}_{XYA}=\chi^\caa{X}_{(TX)(TY)(PA)},\quad  P\nu^\ca{A}_A=\nu^\caa{X}_{PA}
\end{equation}
for all $X,Y\in\ca{V}$ and $A\in\ca{A}$, following the notation from \cref{actionmaps}.
\end{defi}

With the purpose of generalizing the action-induced enrichment from $\Cat$ in \cref{actionenrich}
to $\B{Fib}_\mathrm{opl}$, we conclude to \cite[Def.~3.8]{Enrichedfibration}.

\begin{defi}\label{enrichedfibration}
If $T\colon\ca{V}\to\caa{W}$ is a monoidal fibration, we say an ordinary fibration $P\colon\ca{A}\to\caa{X}$ is \emph{enriched} in $T$ when
\begin{itemize}
 \item $\ca{A}$ is enriched in $\ca{V}$, $\caa{X}$ is enriched in $\caa{W}$ and the following commutes:
 \begin{displaymath}
\xymatrix @C=.8in @R=.3in
{\ca{A}^\op\times\ca{A}\ar[r]^-{\ca{A}(-,-)}
\ar[d]_-{P^\op\times P} & \ca{V}\ar[d]^-T \\
\caa{X}^\op\times\caa{X}\ar[r]_-{\caa{X}(-,-)} &
\caa{W}}
\end{displaymath}
\item composition and identities of the enrichments are compatible, in that
\begin{displaymath}
TM^{\ca{A}}_{A,B,C}=M^{\caa{X}}_{PA,PB,PC}\;\textrm{ and }\; Tj^\ca{A}_A=j^{\caa{X}}_{PA}.
\end{displaymath}
\end{itemize}
\end{defi}

Dually, we have the notion of an opfibration enriched in a monoidal opfibration. Moreover, we say
that a fibration $P$ is enriched in a monoidal opfibration $T$ if and only if the opfibration $P^\op$ is $T$-enriched.
Finally, \cite[Thm.~3.11]{Enrichedfibration} gives a direct way of obtaining an enriched fibration.

\begin{thm}\label{thmactionenrichedfibration}
Suppose that $T:\ca{V}\to\caa{W}$ is a monoidal fibration, which acts on an (ordinary) fibration $P:\ca{A}\to\caa{X}$
via the fibred 1-cell \cref{eq:fibred1cellaction}.
If this action has an oplax fibred parametrized adjoint $(R,S):P^\op\times P\to T$, then we can enrich the fibration 
$P$ in the monoidal fibration $T$.
\end{thm}

We will later use the dual version, for which an action of a monoidal opfibration (a pseudomonoid in $\B{OpFib}$)
induces an enrichment via an opfibred parametrized adjunction.

\section{Double categories}\label{doublecats}

The setting of double categories, and more specifically fibrant monoidal double categories, is crucial
for this work's development but also for further applications.
A few references for the theory of double categories and results relevant here are
\cite{Limitsindoublecats,Adjointfordoublecats,Framedbicats};
the original concept of a double category as a category internal in $\Cat$
goes back to \cite{Ehresmanndouble}.
In order to provide a
passage from double categories to bicategories, we largely follow the notation and approach of
\cite{ConstrSymMonBicats} where a method for constructing monoidal bicategories
from monoidal double categories is described.

We then proceed to the study of monads (see also \cite{Monadsindoublecats}) and comonads in double categories,
as well a natural (op)fibrational picture they form over the vertical category
induced by fibrancy conditions. In the monoidal case, these are in fact monoidal (op)fibrations in the sense of the previous section.

Finally, by introducing the notion of a \emph{locally closed monoidal} double category, capturing a closed structure
for both vertical and horizontal categories, we are able to explore enrichment relations between the category of monads
and comonads, sometimes forming an enriched fibration.

\subsection{Background on fibrant double categories}

We recall the central definitions and fix our notation.

\begin{defi}\label{def:doublecats}
 A \emph{(pseudo) double category} $\caa{D}$
consists of a category of objects $\caa{D}_0$ and
a category of arrows $\caa{D}_1$, with (identity, source and target, composition) structure
functors 
\begin{displaymath}
 \B{1}:\caa{D}_0\to\caa{D}_1,\quad 
\Gr{s},\Gr{t}:\caa{D}_1\rightrightarrows\caa{D}_0,\quad
\odot:\caa{D}_1{\times_{\caa{D}_0}}\caa{D}_1\to\caa{D}_1
\end{displaymath}
such that
$\Gr{s}(1_A)$=$\Gr{t}(1_A)$=$A,\;\Gr{s}(M\odot N)$=$\Gr{s}(N),\;
\Gr{t}(M\odot N)$=$\Gr{t}(M)$
for all $A\in\ob\caa{D}_0$ and $M,N\in\ob\caa{D}_1$,
equipped with natural isomorphisms
\begin{gather*}
\alpha:(M\odot N)\odot P\xrightarrow{\sim}M\odot(N\odot P) \\
\lambda:1_{\Gr{s}(M)}\odot M\xrightarrow{\;\sim\;}M \quad
\rho:M\odot1_{\Gr{t}(M)}\xrightarrow{\;\sim\;}M
\end{gather*}
in $\caa{D}_1$ such that 
$\Gr{t}(\alpha),\Gr{s}(\alpha),\Gr{t}(\lambda),\Gr{s}(\lambda),
\Gr{t}(\rho),\Gr{s}(\rho)$ are all
identities, and satisfying the usual coherence conditions
(as for a bicategory \cref{bicataxiom1}).
\end{defi}
The objects of $\caa{D}_0$ are called \emph{0-cells} and the morphisms $f:A\to B$ of $\caa{D}_0$ are called 
\emph{vertical 1-cells}. The objects of $\caa{D}_1$ are the \emph{horizontal 1-cells}
or \emph{proarrows},$\SelectTips{eu}{10}
\xymatrix@C=.2in{M:A\ar[r]
\ar@{}[r]|-{\scriptstyle{\bullet}} & B}$.
The morphisms of $\caa{D}_1$ are the 
\emph{2-morphisms}
\begin{equation}\label{2morphism}
\xymatrix @C=.3in @R=.3in
{A\ar[r]^-M\ar@{}[r]|-{\scriptstyle{\bullet}} 
\rtwocell<\omit>{<4>\alpha} \ar[d]_-f & B\ar[d]^-g \\
C\ar[r]_-N\ar@{}[r]|-{\scriptstyle{\bullet}} & D}
\end{equation}
or $^f\alpha^g:M\Rightarrow N$, 
where $\Gr{s}(\alpha)=f$ and $\Gr{t}(\alpha)=g$.
The composition of vertical 1-cells  
and the vertical composition of 2-morphisms
are strictly associative, whereas 
horizontal composition of horizontal
1-cells and 2-morphisms is associative up to isomorphism, written
\begin{equation}\label{2cellscomp}
 \xymatrix @C=.25in @R=.25in
{A\ar[r]^-M\ar@{}[r]|-{\scriptstyle{\bullet}} \ar[d]_-f
\rtwocell<\omit>{<3>\alpha} & B\ar[d]^-g \\
C\ar[r]^-N\ar@{}[r]|-{\scriptstyle{\bullet}} \ar[d]_-h
\rtwocell<\omit>{<3>\beta} & D\ar[d]^-k \\
E\ar[r]_-P\ar@{}[r]|-{\scriptstyle{\bullet}} & F}
\xymatrix{\hole \\ = \\ \hole}
\xymatrix @C=.25in @R=.25in
{A\ar[r]^-M\ar@{}[r]|-{\scriptstyle{\bullet}} \ar[dd]_-{hf} &
B\ar[dd]^-{kg} \\
\qquad\color{white}{C} \rtwocell<\omit>{\;\beta\alpha} & \\
E\ar[r]_-P\ar@{}[r]|-{\scriptstyle{\bullet}} & F,}\quad
 \xymatrix @C=.08in @R=.08in
{&& \\
A\ar[rr]^-M\ar@{}[rr]|-{\scriptstyle{\bullet}} \ar[dd]_-f
\rrtwocell<\omit>{<3.5>\alpha} && B\ar[rr]^-N
\ar@{}[rr]|-{\scriptstyle{\bullet}} \ar[dd]_-g \rrtwocell
<\omit>{<3.5>\beta} && C\ar[dd]^-h \\
&&& \\
D\ar[rr]_-P\ar@{}[rr]|-{\scriptstyle{\bullet}} && E\ar[rr]_-K
\ar@{}[rr]|-{\scriptstyle{\bullet}}  && F \\
&&}
\xymatrix @C=.08in @R=.08in
{ \\
\hole \\
= \\ }
\xymatrix @C=.08in @R=.08in
{&& \\
A\ar[rrr]^-{N\odot M}\ar@{}[rrr]|-{\scriptstyle{\bullet}}
\ar[dd]_-f & \rtwocell<\omit>{<3.5>{\quad\beta\odot\alpha}} && 
C\ar[dd]^-h \\
&&& \\
D\ar[rrr]_-{K\odot P}\ar@{}[rrr]|-{\scriptstyle{\bullet}} &&& F. \\
&&} 
\end{equation}
Strict (vertical) identities are $\mathrm{id}_A:A\to A$ and $\id_M:M\Rightarrow M$,
and horizontal units are $\SelectTips{eu}{10}\xymatrix@C=.2in
{1_A:A\ar[r]\ar@{}[r]|-{\scriptstyle{\bullet}} & A}$and
$^f1_f^f:1_A\Rightarrow 1_B$.
Functoriality of the horizontal composition results in the relation $1_N\odot1_M=1_{N\odot M}$ and 
the interchange law $(\beta'\beta)\odot(\alpha'\alpha)=(\beta'\odot\alpha')(\beta\odot\alpha).$
A 2-morphism with identity source and target vertical 1-cell, like $a,l,r$ above, is called  \emph{globular}.

The \emph{opposite double category} $\caa{D}^\op$
is the double category with vertical category 
$\caa{D}_0^\op$ and horizontal category $\caa{D}_1^\op$.
There also exist the \emph{horizontally opposite}
double category $\caa{D}^\mathrm{hop}$ and 
\emph{vertically opposite} double category 
$\caa{D}^\mathrm{vop}$ with opposite horizontal and 
vertical categories respectively.

For every double category $\caa{D}$ there is a corresponding bicategory denoted by $\ca{H}(\caa{D})$, called its \emph{horizontal bicategory};
in a sense, it comes from discarding the vertical structure of the double category.
It consists of the objects, horizontal 1-cells and globular 2-morphisms, and the required axioms are satisfied by default.
Many well-known bicategories arise as the horizontal bicategories of
specific double categories: in fact, all the examples of \cref{sec:bicategories} can be seen as such for the following double
categories, see e.g. \cite{Limitsindoublecats,Framedbicats}.
\begin{itemize}
 \item $\caa{S}\B{pan}(\ca{C})$ and $\caa{R}\B{el}(\ca{C})$ with vertical category $\ca{C}$;
 \item $\caa{B}\B{Mod}$ with vertical category $\B{Rng}$ of rings and ring homomorphisms;
 \item ($\ca{V}$-)$\caa{P}\B{rof}$ with vertical category ($\ca{V}$-)$\B{Cat}$.
\end{itemize}
What is evident from the above examples is that the double categorical perspective includes not only the morphisms
which the bicategorical structure is usually named after and describes best, but also the more fundamental, strict morphisms
between the objects: functions, ring maps and functors above.


\begin{defi}\label{defi:doublefunctor}
For $\caa{D}$ and $\caa{E}$ (pseudo) double categories, a \emph{pseudo double functor} $F:\caa{D}\to\caa{E}$
consists of functors $F_0:\caa{D}_0\to\caa{E}_0$ and $F_1:\caa{D}_1\to\caa{E}_1$ between the categories
of objects and arrows, such that $\Gr{s}\circ F_1=F_0\circ\Gr{s}$ 
and $\Gr{t}\circ F_1=F_0\circ\Gr{t}$, and natural transformations $F_{\odot}$, $F_U$ 
with components globular isomorphisms
$F_1M\odot F_1N\simrightarrow F_1(M\odot N)\textrm{ and }1_{F_0A}\simrightarrow F_1(1_A)$
which satisfy the usual coherence axioms \cref{laxcond1,laxcond2} for a pseudofunctor.
\end{defi}

Due to the compatibility of $F_0$, $F_1$ with sources and targets, we can write the mapping
of $F_1$ on 1-cells and 2-morphisms as
\begin{equation}\label{F1mapping}
\xymatrix@R=.35in
{A\ar[r]^-M\ar@{}[r]|-{\scriptstyle{\bullet}} \ar[d]_-{f}
\rtwocell<\omit>{<4>\alpha} & B\ar[d]^-{g} \\
C\ar[r]_-N\ar@{}[r]|-{\scriptstyle{\bullet}} & D}\;
\xymatrix @R=.1in
{\hole \\ \mapsto }\;
\xymatrix @C=.5in @R=.35in
{F_0A\ar[r]^-{F_1M}\ar@{}[r]|-{\scriptstyle{\bullet}} \ar[d]_-{F_0f}
\rtwocell<\omit>{<4>\quad F_1\alpha} & F_0B\ar[d]^-{F_0g} \\
F_0C\ar[r]_-{F_1N}\ar@{}[r]|-{\scriptstyle{\bullet}} & F_0D}
\end{equation}
We also have notions of \emph{lax} and \emph{colax double functors} between pseudo double categories, where the natural
transformations $F_{\odot}$ and $F_U$ have components globular 2-morphisms in one of the two possible directions respectively.
The full definitions can be found in the appendix of \cite{Limitsindoublecats} or \cite{Adjointfordoublecats}.

Any lax/colax/pseudo double functor $F:\caa{D}\to\caa{E}$ naturally induces a lax/colax/pseudo functor, \cref{laxfunctor},
between the respective horizontal bicategories.
It is denoted by $\ca{H}F:\ca{H}(\caa{D})\to\ca{H}(\caa{E})$, where each $A\in\caa{D}_0$ is mapped to $F_0A\in\caa{E}_0$, and
there are ordinary functors $\ca{H}F_{A,B}:\ca{H}(\caa{D})(A,B)\to\ca{H}(\caa{E})(F_0A,F_0B)$ mapping globular 2-cells
to globular 2-cells via \cref{F1mapping}.

With an appropriate notion for transformations between double functors, there is a 2-category $\ca{D}bl$ of double categories.
A monoidal double category then is a pseudomonoid therein, see \cite[2.9]{ConstrSymMonBicats}
for the full definition. Notably, in \cite{Adjointfordoublecats} the tensor product $\otimes$ is required to be a colax double functor
rather than pseudo double. 

\begin{defi}\label{monoidaldoublecategory}
A \emph{monoidal double category} is a double category $\caa{D}$ equipped with pseudo double functors
$\otimes=(\otimes_1,\otimes_1):\caa{D}\times\caa{D}\to\caa{D}$ and $\B{I}:\B{1}\to\caa{D}$
and invertible transformations expressing associativity and unity constraints, subject to axioms.

These amount to $(\caa{D}_0,\otimes_0,I)$ and $(\caa{D}_1,\otimes_1,1_I)$ being monoidal categories with units $I=\B{I}(*)$
and$\SelectTips{eu}{10}\xymatrix@C=.2in{1_I:I\ar[r]|-{\sbul} & I,}$ $\Gr{s},\Gr{t}$ being strict monoidal and preserving
associativity and unit constraints, and the existence of globular isomorphisms
\begin{align}\label{monoidaldoubleiso}
(M\otimes_1 N)\odot(M'\otimes_1 N')&\cong
(M\odot M')\otimes_1(N\odot N') \\
1_{(A\otimes_0 B)}&\cong
1_A\otimes_1 1_B\nonumber
\end{align}
subject to coherence conditions.
\end{defi}

A \emph{braided} or \emph{symmetric} monoidal double category $\caa{D}$ is one for which $\caa{D}_0$, $\caa{D}_1$ are
braided or symmetric, and the source and target functors $\Gr{s},\Gr{t}$ are strict braided monoidal, subject subject to two more axioms.
By definition \cref{F1mapping}, $\otimes_1$ is the following mapping:
\begin{equation}\label{D1monoidal}
\otimes_1:\xymatrix @C=1.5in
{\caa{D}_1\times\caa{D}_1\ar[r] & \caa{D}_1\phantom{ABC}}
\end{equation}\vspace{-0.2in}
\begin{displaymath}
\xymatrix @C=.045in @R=.25in
{(A\ar[rrr]|-\sbul^M\ar[d]_-f &\rtwocell<\omit>{<4>{\alpha}}&& B,\ar[d]^-g & C\ar[rrr]^-N|-\sbul\ar[d]_-h
&\rtwocell<\omit>{<4>\beta}&& D)\ar[d]^-k
\ar@{|.>}[rrrr] &&&& A\otimes_0 C\ar[rrr]^-{M\otimes_1 N}|-\sbul
\ar[d]_-{f\otimes_0 h} &\rtwocell<\omit>{<4>\quad\alpha\otimes_1\beta}
&& B\otimes_0 D\ar[d]^-{g\otimes_0 k} \\
(A'\ar[rrr]_-{M'}|-\sbul &&& B', & C'\ar[rrr]_-{N'}|-\sbul &&& D')
\ar@{|.>}[rrrr] &&&& A'\otimes_0 B'\ar[rrr]_-{M'\otimes_1 N'}|-\sbul &&& B'\otimes_1 D'} 
\end{displaymath}

Passing on to the theory of fibrant double categories, it is the case that in many examples of double categories,
there exists a canonical way of turning vertical 1-cells into horizontal 1-cells.
Such links have been studied in various works, and the terminology used below can be found in
\cite{Adjointfordoublecats,Framedbicats,Thespanconstruction}.

\begin{defi}\label{deficompconj}
 Let $\caa{D}$ be a double category and $f:A\to B$
a vertical 1-cell. A \emph{companion} of $f$
is a horizontal 1-cell$\SelectTips{eu}{10}\xymatrix@C=.2in
{\hat{f}:A\ar[r]\ar@{}[r]|-{\scriptstyle{\bullet}} & B}$together with
2-morphisms
\begin{displaymath}
 \xymatrix
{A\ar[r]^-{\hat{f}}\ar@{}[r]|-{\scriptstyle{\bullet}} 
\rtwocell<\omit>{<4>\;p_1} \ar[d]_-f & B\ar[d]^-{\mathrm{id}_B} \\
B\ar[r]_-{1_B}\ar@{}[r]|-{\scriptstyle{\bullet}} & B} 
\quad\xymatrix@R=.05in{\hole \\
\textrm{and}}\quad
 \xymatrix
{A\ar[r]^-{1_A}\ar@{}[r]|-{\scriptstyle{\bullet}} 
\rtwocell<\omit>{<4>\;p_2} \ar[d]_-{\mathrm{id}_A} & 
A\ar[d]^-{f} \\
A\ar[r]_-{\hat{f}}\ar@{}[r]|-{\scriptstyle{\bullet}} & B}
\end{displaymath}
such that $p_1p_2=1_f$ and $p_1\odot p_2\cong1_{\hat{f}}$.
Dually, a \emph{conjoint} of 
$f$ is a horizontal 1-cell$\SelectTips{eu}{10}\xymatrix@C=.2in
{\check{f}:B\ar[r]\ar@{}[r]|-{\scriptstyle{\bullet}} & A}$together with 2-morphisms
\begin{displaymath}
 \xymatrix
{B\ar[r]^-{\check{f}}\ar@{}[r]|-{\scriptstyle{\bullet}} 
\rtwocell<\omit>{<4>\;q_1} \ar[d]_-{\mathrm{id}_B} & A\ar[d]^-f \\
B\ar[r]_-{1_B}\ar@{}[r]|-{\scriptstyle{\bullet}} & B} \quad\xymatrix@R=.05in{\hole \\
\textrm{and}}\quad
 \xymatrix
{A\ar[r]^-{1_A}\ar@{}[r]|-{\scriptstyle{\bullet}} 
\rtwocell<\omit>{<4>\;q_2} \ar[d]_-{f} & 
A\ar[d]^-{\mathrm{id}_A} \\
B\ar[r]_-{\check{f}}\ar@{}[r]|-{\scriptstyle{\bullet}} & A}
\end{displaymath}
such that $q_1q_2=1_f$ and $q_2\odot q_1\cong 1_{\check{f}}$.
\end{defi}
The ideas which led to the above definitions go
back to \cite{Doublegroupoidsandcrossedmodules},
where a \emph{connection} on a double category
corresponds to a strictly functorial choice
of a companion for each vertical arrow.

\begin{defi}\cite[Definition 3.4]{ConstrSymMonBicats}\label{fibrantdoublecat}
A \emph{fibrant double category} is a double category
for which every vertical 1-morphism has a companion and a conjoint.
\end{defi}

Fibrant double categories are also called \emph{framed bicategories} or equivalently \emph{proarrow equipments} \cite{Wood:1982a,Verityequip}.
The above definition's equivalence with the following can be found for example at \cite[Thm.~4.1]{Framedbicats}.

\begin{defi}\label{Grothfibrant}
A fibrant double category is one where the functor
\begin{displaymath}
(\Gr{s},\Gr{t})\colon\caa{D}_1\longrightarrow\caa{D}_0\times\caa{D}_0
\end{displaymath}
mapping each horizontal 1-cell and 2-morphism to the pair of source and target is a fibration, or equivalently an opfibration.
\end{defi}

In this view, the canonical cartesian lifting of some$\proar{N}{C}{D}$ along
a pair of vertical morphisms $f\colon A\to C$, $g\colon B\to D$
\begin{equation}\label{cartliftframdebicat}
\xymatrix @C=.17in
{\check{g}\odot N\odot\hat{f}\ar[rrr]^-{\Cart((f,g),N)}\ar @{.>}[d] &&& N \ar @{.>}[d] &\caa{D}_1\ar[d]^-{(\Gr{s},\Gr{t})} \\
(A,B)\ar[rrr]_-{(f,g)} &&& (C,D) & \caa{D}_0\times\caa{D}_0}\quad\textrm{ is }\quad
\xymatrix @R=.43in
{A\ar[d]_-f \ar[r]^-{\hat{f}}\ar@{}[r]|-{\scriptstyle{\bullet}}\rtwocell<\omit>{<5>\;p^f_1} &
C\ar[r]^-N\ar@{}[r]|-{\scriptstyle{\bullet}} \ar@{=}[d]\rtwocell<\omit>{<5>\;1_N} & D\ar@{=}[d]\ar[r]^-{\check{g}}
\ar@{}[r]|-{\scriptstyle{\bullet}} \rtwocell<\omit>{<5>\;q^g_1} & C\ar[d]^-g \\
C\ar[r]_-{1_C}\ar@{}[r]|-{\scriptstyle{\bullet}} & C\ar[r]_-N\ar@{}[r]|-{\scriptstyle{\bullet}} & 
D\ar[r]_-{1_D}\ar@{}[r]|-{\scriptstyle{\bullet}} & F}
\end{equation}
Many properties for fibrant double categories can be deduced from the companion and conjoint definitions.
The following lemma gathers the most useful for us; the explicit proofs can be found in
\cite{ConstrSymMonBicats}, or can be easily deduced e.g. via mates and factorization
through lifting \cref{factor} for \cref{five}.

\begin{lem}\label{compconjprops}
Suppose $\caa{D}$ is a fibrant double category.
\begin{enumerate}
 \item Companions and conjoints of a vertical 1-cell are essentially unique (up to unique globular isomorphism).
 \item\label{five} For any vertical 1-cells $f\colon A\to C$, $g\colon B\to D$ and horizontal 1-cells
 $\proar{M}{A}{B,}\proar{N}{C}{D,}$ we have bijections between
 \begin{gather}\label{matesforcompconj}
 \ca{H}(\caa{D})(M,\check{g}\odot N\odot\hat{f})\cong \ca{H}(\caa{D})(\hat{g}\odot M,N\odot\hat{f}) \cong \\
 \ca{H}(\caa{D})(M\odot\check{f},\check{g}\odot N)\cong\ca{H}(\caa{D})(\hat{g}\odot M\odot\check{f},N) \nonumber
 \end{gather}
 and the 2-morphisms $M\Rightarrow N$ with source and target $f$ and $g$ \cref{2morphism}.
 \item\label{compositecompconj} The horizontal composites $\hat{g}\odot\hat{f}$ and $\check{g}\odot\check{f}$ are the
 companion and the conjoint of the vertical composite $gf$, for any composable vertical 1-cells.
 \item The companion and conjoint of the vertical identities are the horizontal identites, $\wh{\id_A}=\wc{\id_A}=1_A$.
 \item For any vertical 1-cell $f\colon A\to B$, we have an adjunction $\hat{f}\dashv\check{f}$ in the horizontal
 bicategory $\ca{H}(\caa{D})$.
 \item If $\caa{D}$ is also monoidal, $\hat{f}\otimes_1\hat{g}$ and $\check{f}\otimes_1\check{g}$
 are the companion and conjoint of $f\otimes_0 g$ for any vertical 1-cells $f$, $g$.
 \end{enumerate}
\end{lem}

For example, $\ca{V}$-$\caa{P}\B{rof}$ is a fibrant double category: the companion and conjoint for a $\ca{V}$-functor
$F\colon\ca{A}\to\ca{B}$ are given by the \emph{representable profunctors}
\begin{gather*}
\matr{\hat{F}=F_*}{\ca{A}}{\ca{B}}\textrm{ by }F_*(B,A)=\ca{B}(B,FA) \\
\matr{\check{F}=F^*}{\ca{B}}{\ca{A}}\textrm{ by }F^*(A,B)=\ca{B}(FA,B)
\end{gather*}
For $\caa{S}\B{pan}(\ca{C})$, the companion and conjoint of a function $f\colon A\to B$ are $\check{f}=(\id_A,f)$ and
$\hat{f}=(f,\id_B)$, whereas for $\B{BMod}$ a ring morphism $f\colon A\to B$ gives rise to
$B$ as a left-$A$ right-$B$ bimodule but also as a left-$B$ right-$A$ bimodule via restriction of scalars.

A fundamental property of a fibrant double category with a monoidal structure is that its horizontal bicategory inherits it.
This process, studied in detail in \cite{ConstrSymMonBicats}, allows us to reduce a lengthy and demanding task of verifying
the coherence conditions of monoidal structure on a bicategory into a much more concise one, essentially involving a pair 
of ordinary monoidal categories.

\begin{thm}\cite[Theorem 5.1]{ConstrSymMonBicats}\label{monoidalhorizontalbicategory}
If $\caa{D}$ is a fibrant monoidal double category,
then $\ca{H}(\caa{D})$ is a monoidal bicategory. If $\caa{D}$
is braided or symmetric, then so is $\ca{H}(\caa{D})$.
\end{thm}

Evidently, the monoidal structure of the bicategory consists of
the induced pseudofunctor of bicategories $\ca{H}(\otimes):\ca{H}(\caa{D})\times\ca{H}(\caa{D})\to
\ca{H}(\caa{D})$ and the monoidal unit $1_I$ of $\caa{D}_1$, for a $(\caa{D},\otimes, I)$ as in
\cref{monoidaldoublecategory}.

\subsection{Monads and comonads in double categories}\label{moncomondouble}


Suppose $\caa{D}$ is a double category. Define the category of \emph{endomorphisms} $\caa{D}_1^\bullet$ to be
the (non-full) subcategory of $\caa{D}_1$ of all horizontal endo-1-cells and 2-morphisms
with the same source and target. Explicitly, objects are of the form$\SelectTips{eu}{10}\xymatrix@C=.2in
{M:A\ar[r]\ar@{}[r]|-{\scriptstyle{\bullet}} & A}$and arrows
\begin{equation}\label{endo2morphism}
\xymatrix
{A\ar[r]^-M\ar@{}[r]|-{\scriptstyle{\bullet}} 
\rtwocell<\omit>{<4>\alpha} \ar[d]_-f & A\ar[d]^-f \\
B\ar[r]_-N\ar@{}[r]|-{\scriptstyle{\bullet}} & B}
\end{equation}
denoted by $\alpha_f:M_A\to N_B$.
In \cite{Monadsindoublecats}, this category
coincides with the vertical 1-category of the 
double category $\caa{E}\B{nd}(\caa{D})$
of (horizontal) endomorphisms,
horizontal endomorphism maps, vertical endomorphism
maps and endomorphism squares
in $\caa{D}$.

\begin{defi}\label{Monadindoublecat}
A \emph{monad} in a double category $\caa{D}$ is a horizontal endo-1-cell$\SelectTips{eu}{10}\xymatrix@C=.2in
{M:A\ar[r]\ar@{}[r]|-{\scriptstyle{\bullet}} & A}$ \emph{i.e.}
an object in $\caa{D}_1^\bullet$, equipped with 
globular 2-morphisms
\begin{displaymath}
\xymatrix @C=.4in @R=.4in
{A\ar[r]^-M\ar@{}[r]|-{\scriptstyle{\bullet}}  
\rrtwocell<\omit>{<4.5>m} \ar[d]_-{\mathrm{id_A}} 
& A\ar[r]^-M\ar@{}[r]|-{\scriptstyle{\bullet}} & A\ar[d]^-{\mathrm{id}_A} \\
A\ar[rr]_-M\ar@{}[rr]|-{\scriptstyle{\bullet}}  && A,}\qquad
\xymatrix @C=.4in @R=.4in
{A\ar[r]^-{1_A}\ar@{}[r]|-{\scriptstyle{\bullet}}  
\rtwocell<\omit>{<4.5>\eta} \ar[d]_-{\mathrm{id_A}} 
& A\ar[d]^-{\mathrm{id}_A} \\
A\ar[r]_-M\ar@{}[r]|-{\scriptstyle{\bullet}}  & A}
\end{displaymath}
satisfying the usual associativity and unit laws; this is the same as a monad in its horizontal bicategory $\ca{H}(\caa{D})$
(\cref{monadbicat}).
A \emph{monad morphism} consists of an arrow $\alpha_f:M_A\to N_B$ in $\caa{D}_1^\bullet$ which respects multiplication and unit:
\begin{equation}\label{monadhom}
\xymatrix@C=.3in @R=.2in
{A\ar[r]^-M\ar@{}[r]|-{\scriptstyle{\bullet}}\ar[d]_-f
\rtwocell<\omit>{<3>\alpha} & 
A\ar[r]^-M\ar@{}[r]|-{\scriptstyle{\bullet}}\ar[d]^-f
\rtwocell<\omit>{<3>\alpha} &
A\ar[d]^-f \\
B\ar[r]_-N\ar@{}[r]|-{\scriptstyle{\bullet}}\rrtwocell<\omit>{<3>m}
\ar@{=}[d] &
B\ar[r]_-N\ar@{}[r]|-{\scriptstyle{\bullet}}
& B \ar@{=}[d] \\
B\ar[rr]_-N\ar@{}[rr]|-{\scriptstyle{\bullet}} && B}
\xymatrix@R=.2in{\hole \\ =}
\xymatrix@C=.3in @R=.2in
{A\ar[r]^-M\ar@{}[r]|-{\scriptstyle{\bullet}}
\rrtwocell<\omit>{<3>m}\ar@{=}[d] & 
A\ar[r]^-M\ar@{}[r]|-{\scriptstyle{\bullet}} 
& A\ar@{=}[d] \\
A\ar[rr]_-M\ar@{}[rr]|-{\scriptstyle{\bullet}}\ar[d]_-f
\rrtwocell<\omit>{<4>\alpha} && A\ar[d]^-f \\
B\ar[rr]_-N\ar@{}[rr]|-{\scriptstyle{\bullet}}
 && B,}
\qquad
\xymatrix@C=.3in @R=.2in
{A\ar[r]^-{1_A}\ar@{}[r]|-{\scriptstyle{\bullet}}\ar@{=}[d]
\rtwocell<\omit>{<3>\eta} & A\ar@{=}[d] \\
A\ar[r]_-M\ar@{}[r]|-{\scriptstyle{\bullet}}\ar[d]_-f
\rtwocell<\omit>{<3.5>\alpha} &
A\ar[d]^-f \\
B\ar[r]_-N\ar@{}[r]|-{\scriptstyle{\bullet}} & B}
\xymatrix@R=.2in{\hole \\ =}
\xymatrix@C=.3in @R=.2in
{A\ar[r]^-{1_A}\ar@{}[r]|-{\scriptstyle{\bullet}}\ar[d]_-f
\rtwocell<\omit>{<3>1_f} & A\ar[d]^-f \\
B\ar[r]_-{1_B}\ar@{}[r]|-{\scriptstyle{\bullet}}\ar@{=}[d]
\rtwocell<\omit>{<3.5>\eta} &
B\ar@{=}[d] \\
B\ar[r]_-N\ar@{}[r]|-{\scriptstyle{\bullet}} & B.}
\end{equation}
\end{defi}
We obtain a non-full subcategory of $\caa{D}_1$, the category $\Mnd(\caa{D})$. These definitions can be found in \cite{Framedbicats}
under the names of \emph{monoids} and \emph{monoid homomorphisms} for fibrant double categories, as well as in \cite{Monadsindoublecats}
as monads and \emph{vertical} monad maps in a double category $\caa{D}$. In the latter work, $\Mnd(\caa{D})$ is the vertical 
category of $\caa{M}\B{nd}(\caa{D})$, a double category of monads, horizontal and vertical monad maps and monad squares.

Dually, we have the following definition.

\begin{defi}\label{Comonadindoublecat}
There is a category $\Cmd(\caa{D})$ with objects \emph{comonads} in $\caa{D}$, \emph{i.e.} horizontal 
endo-1-cells$\SelectTips{eu}{10}\xymatrix@C=.2in{C:A\ar[r]\ar@{}[r]|-{\scriptstyle{\bullet}} & A}$equipped with globular 2-morphisms
\begin{displaymath}
\xymatrix @C=.4in @R=.4in
{A\ar[rr]^-C\ar@{}[rr]|-{\scriptstyle{\bullet}} \ar[d]_-{\mathrm{id_A}} 
&& A\ar[d]^-{\mathrm{id}_A} \\
A\ar[r]_-C\ar@{}[r]|-{\scriptstyle{\bullet}}  
\rrtwocell<\omit>{<-4>\Delta}  
& A\ar[r]_-C\ar@{}[r]|-{\scriptstyle{\bullet}} & A,}\qquad
\xymatrix @C=.4in @R=.4in
{A\ar[r]^-C\ar@{}[r]|-{\scriptstyle{\bullet}}\ar[d]_-{\mathrm{id_A}}   
& A\ar[d]^-{\mathrm{id}_A} \\
A\ar[r]_-C\ar@{}[r]|-{\scriptstyle{\bullet}}  
\rtwocell<\omit>{<-4>\epsilon} 
& A}
\end{displaymath}
satisfying the usual coassociativity and counit axioms for a comonad in the horizontal bicategory $\ca{H}(\caa{D})$
(\cref{comonadbicat}). Arrows are \emph{comonad morphisms}, \emph{i.e.} $\alpha_f:C_A\to D_B$ in $\caa{D}_1^\bullet$
satisfying dual axioms to \cref{monadhom}.
\end{defi}

Observe that $\Mnd(\caa{D}^\op)=\Cmd(\caa{D})^\op$, and that the forgetful functors
$\Mnd(\caa{D}),\Cmd(\caa{D})\to\caa{D}_1^\bullet\to\caa{D}_1$ reflect isomorphisms.
Moreover, there are also forgetful functors to $\caa{D}_0$, mapping a horizontal endo-1-cell$\SelectTips{eu}{10}\xymatrix@C=.2in
{M:A\ar[r]\ar@{}[r]|-{\scriptstyle{\bullet}} & A}$to its source and target $A$, and a 2-morphism $\alpha_f\colon M_A\to N_B$
to its source and target $f\colon A\to B$; these are studied in detail below. 

In \cite[\S 8]{PhDChristina}, the above structures were called (co)monoids;
the current terminology is preferred due to the fact that it doesn't require any monoidal structure on the
double category. A monoid in a monoidal double category should correspond to a lax double functor
$\B{1}\to\caa{D}$, which comes down to a monoid in the vertical category $\caa{D}_0$ as the source and target of a monoid
in the horizontal category $\caa{D}_1$.

We now consider how different notions of double functors relate to the 
categories of endomorphisms, monads and comonads. The following can be deduced from the definition of a double functor \cref{F1mapping}
in a straightforward way.

\begin{cor}\label{F1bullet}
Suppose that $F\colon\caa{D}\to\caa{E}$ is a lax/colax/pseudo double functor.
Then $F_1$ naturally induces an ordinary functor $\Fendo\colon\Dendo\to\Eendo$.
\end{cor}

For monads and comonads, the following resembles to standard properties of monoidal functors,
see \cref{Monoidalcats}, which is also part of \cite[11.11]{Framedbicats}
\begin{prop}\label{MonFdouble}
Any lax double functor $F=(F_0,F_1):\caa{D}\to\caa{E}$ induces an ordinary functor
\begin{displaymath}
 \Mon F:\Mnd(\caa{D})\to\Mnd(\caa{E})
\end{displaymath}
between their categories of monads, by restricting 
$F_1$ to $\Mnd(\caa{D})$. Dually, any colax double functor induces
a functor between the categories of comonoids,
\begin{displaymath}
 \Comon F:\Cmd(\caa{D})\to\Cmd(\caa{E}).
\end{displaymath}
\end{prop}

\begin{proof}
A monad$\SelectTips{eu}{10}\xymatrix@C=.2in{M:A\ar[r]\ar@{}[r]|-{\scriptstyle{\bullet}} & A}$with $m:M\odot M\to M$
and $\eta:1_M\to M$ is mapped to$\SelectTips{eu}{10}\xymatrix@C=.2in{F_1M:F_0A\ar[r]\ar@{}[r]|-{\scriptstyle{\bullet}} & F_0A}$with 
multiplication and unit
\begin{displaymath}
 \xymatrix
{F_0A\ar[r]^-{F_1M}\ar@{}[r]|-{\scriptstyle{\bullet}}\ar@{=}[d]\rrtwocell<\omit>{<5>\;F_\odot}
& F_0A\ar[r]^-{F_1M}\ar@{}[r]|-{\scriptstyle{\bullet}} & F_0A\ar@{=}[d] \\
F_0A\ar[rr]_-{F_1(M\odot M)}\ar@{}[rr]|-{\scriptstyle{\bullet}}\ar@{=}[d]\rrtwocell<\omit>{<5>\quad F_1m}
&& F_0A\ar@{=}[d] \\
F_0A\ar[rr]_-{F_1 M}\ar@{}[rr]|-{\scriptstyle{\bullet}} && F_0A}\quad
\xymatrix
{\hole \\ \mathrm{and}}
\quad
\xymatrix
{F_0A\ar[r]^-{F_1(1_A)}\ar@{}[r]|-{\scriptstyle{\bullet}}\ar@{=}[d]\rtwocell<\omit>{<5>\;F_U} &
F_0A\ar@{=}[d] \\
F_0A\ar[r]_-{1_{F_0A}}\ar@{}[r]|-{\scriptstyle{\bullet}}\ar@{=}[d]\rtwocell<\omit>{<5>\quad F_1\eta} &
F_0A\ar@{=}[d] \\
F_0A\ar[r]_-{F_1M}\ar@{}[r]|-{\scriptstyle{\bullet}} & F_0A}
\end{displaymath}
and the axioms follow from the axioms for $F_{\odot}$ and 
$F_U$. A monad map $\alpha_f:M_A\to N_B$ is mapped to $F_1^\bullet\alpha$,
\begin{displaymath}
 \xymatrix
{F_0A\ar[r]^-{F_1 M}\ar@{}[r]|-{\scriptstyle{\bullet}}\ar[d]_-{F_0f}\rtwocell<\omit>{<4>\;\; F_1\alpha} &
F_0A\ar[d]^-{F_0f} \\
F_0B\ar[r]_-{F_1 N}\ar@{}[r]|-{\scriptstyle{\bullet}} & F_0B}
\end{displaymath}
which respects multiplications and units by naturality 
of $F_{\odot}$ 
and $F_U$.
Similarly for the induced functor between comonads.
\end{proof}

\begin{rmk}\label{doublemonadsaremonoids}
Recall by \cref{monadsaremonoids} that monads in a bicategory are mo\-no\-ids
in a monoidal endo-hom-category with horizontal composition. This point of view will be useful later,
hence we should equivalently view double categorical monads as $M_A\in\Mon(\HD(A,A),\odot,1_A)$
and comonads as $C_A\in\Comon(\HD(A,A),\odot,1_A)$. (Co)monad morphisms cannot be expressed
as arrows therein, since the respective 2-morphisms are more general than just globular ones;
this is exactly why categories of double (co)monads capture further desired structure.

Since (co)lax double functors induce (co)lax functors between the horizontal bicategories,
on the level of objects \cref{MonFdouble} coincides with \cref{laxfunctorspreservemonads}. 
\end{rmk}

As an application of the above, consider a monoidal double
category $(\caa{D},\otimes,\B{I})$ as in \cref{monoidaldoublecategory}.
The pseudo double functor $\otimes\colon\caa{D}\times\caa{D}\to\caa{D}$
induces, by \cref{F1bullet} and \cref{MonFdouble}, ordinary functors 
\begin{gather*}
\otimes_1^\bullet\colon\Dendo\times\Dendo\to\Dendo \\
\Mon\otimes\colon\Mnd(\caa{D})\times\Mnd(\caa{D})\to\Mnd(\caa{D}) \\
\Comon\otimes\colon\Cmd(\caa{D})\times\Cmd(\caa{D})\to\Cmd(\caa{D})
\end{gather*}
given by $\otimes_1$ \cref{D1monoidal} restricted to the specific subcategories of $\Dar$.
Along with the monoidal unit $\SelectTips{eu}{10}\xymatrix@C=.2in
{1_I:I\ar[r]\ar@{}[r]|-{\scriptstyle{\bullet}} & I}$, and since the forgetful functors to the monoidal $\Dar$
are conservative, we obtain the following.
\begin{prop}\label{DendoMonDComonDmonoidal}
If $\caa{D}$ is a monoidal double category, then the categories $\caa{D}_1^\bullet$,
$\Mnd(\caa{D})$ and $\Cmd(\caa{D})$ inherit a monoidal structure from $\caa{D}_1$.
When $\caa{D}$ is braided or symmetric, then so are the categories of endomorphisms, monads and comonads.
\end{prop}

We now further study these categories in the fibrant setting, \cref{fibrantdoublecat}. The following
is proved in detail, to serve as reference for future constructions.

\begin{prop}\label{D_1^.bifibred}
 If $\caa{D}$ is a fibrant double category, $\caa{D}_1^\bullet$ is bifibred over $\caa{D}_0$.
\end{prop}

\begin{proof}
Due to the correspondence of \cref{Grothendieckcorrespondence}, it is enough to define
pseudofunctors from $\caa{D}_0^{(\op)}$ which give rise to a fibration and opfibration
with total categories isomorphic to $\caa{D}_1^\bullet$ via the Grothendieck
construction; define
\begin{equation}\label{pseudofunctorsdoublebifibration}
 \ps{M}:
\xymatrix @R=.02in
{\caa{D}_0^\op\ar[r] & \B{Cat}, \\
A\ar @{|.>}[r]
\ar [dd]_-f & \ca{H}(\caa{D})(A,A) \\
\hole \\
B\ar @{|.>}[r] &
\ca{H}(\caa{D})(B,B)\ar[uu]_-{(\check{f}\odot\text{-}\odot\hat{f})}}
\qquad
\ps{F}:
\xymatrix @R=.02in
{\caa{D}_0\ar[r] & \B{Cat} \\
A\ar @{|.>}[r]
\ar [dd]_-f & \ca{H}(\caa{D})(A,A)
\ar[dd]^-{(\hat{f}\odot\text{-}\odot\check{f})} \\
\hole \\
B\ar @{|.>}[r] &
\ca{H}(\caa{D})(B,B).}
\end{equation}
In more detail, the first one is given by the mapping
on objects and arrows
\begin{displaymath}
\SelectTips{eu}{10}
\xymatrix{(B \ar @/^2ex/[r]|-{\scriptstyle{\bullet}}^-H
\ar@/_2ex/[r]|-{\scriptstyle{\bullet}}_-{H'}
\rtwocell<\omit>{\sigma} & B)
\ar @{|.>}[r] & (A\ar[r]|-{\scriptstyle{\bullet}}
^-{\hat{f}} & B \ar @/^2ex/[r]|-{\scriptstyle{\bullet}}^-H
\ar@/_2ex/[r]|-{\scriptstyle{\bullet}}_-{H'}
\rtwocell<\omit>{\sigma} & B
\ar[r]|-{\scriptstyle{\bullet}}^-{\check{f}} & A)}
\end{displaymath}
pre-composing with the companion and post-composing with the conjoint
of the given vertical 1-cell.
For these mappings to constitute a pseu\-do\-fu\-nctor $\ps{M}$,
we need certain natural isomorphisms
satisfying coherence conditions as in \cref{laxfunctor}.
For every triple of 0-cells $A,B,C$, there is a natural isomorphism
$\delta$ with components
\begin{displaymath}
\xymatrix @R=.04in
{& \ca{H}(\caa{D})(B,B)\ar@/^/[dr]^-{\ps{M}f} & \\
\ca{H}(\caa{D})(C,C)\ar@/^/[ur]^-{\ps{M}g}
\ar@/_3ex/[rr]_-{\ps{M}(g\circ f)}
\rrtwocell<\omit>{\quad\delta^{g,f}} && 
\ca{H}(\caa{D})(A,A)}
\end{displaymath}
for any $f:A\to B$ and $g:B\to C$, satisfying the commutativity
of \cref{laxcond1}. Explicitely, each $\delta^{g,f}$
has components, for each horizontal 1-cell$\SelectTips{eu}{10}
\xymatrix@C=.2in{J:C\ar[r]|-{\scriptstyle\bullet} & C,}$
\begin{equation}\label{Mdelta}
\delta^{g,f}_J:(\ps{M}f\circ\ps{M}g)J
\xrightarrow{\;\sim\;}\ps{M}(g\circ f)J=
\wc{f}\odot\wc{g}\odot J\odot\wh{g}\odot\wh{f}\xrightarrow{\;\sim\;}
\wc{gf}\odot J\odot\wh{gf}
\end{equation}
in $\ca{H}(\caa{D})(A,A)$, due to \cref{compconjprops}.
Moreover, for any 0-cell $A$ there is a natural isomorphism
$\gamma$ with components
\begin{displaymath}
\xymatrix @C=.5in{\ca{H}(\caa{D})(A,A)
\rrtwocell<5>^{\B{1}_{\ca{H}(\caa{D})(A,A)}}
_{\ps{M}(\mathrm{id}_A)}
{\quad\gamma^A} && \ca{H}(\caa{D})(A,A)}
\end{displaymath}
with components invertible arrows in $\ca{H}(\caa{D})(A,A)$ 
\begin{equation}\label{Mgamma}
\gamma^X_G:G\xrightarrow{\;\sim\;}\ps{M}(\mathrm{id}_A)G=\wc{\id_A}\odot\wh{\id_A}
\end{equation}
again by \cref{compconjprops}; it can be verified that the axioms \cref{laxcond2} are satisfied.

The Grothendieck category $\Gr{G}\ps{M}$ has as objects pairs $(G,A)$ where $A$ is a 0-cell and $G$ is in $\ca{H}(\caa{D})(A,A)$, and as arrows
$(\phi,f):(G,A)\to(H,B)$ pairs
\begin{displaymath}
\begin{cases} 
G\xrightarrow{\phi}\check{f}\odot H\odot \hat{f} 
&\text{in }\ca{H}(\caa{D})(A,A)\\
A\xrightarrow{f}B &\text{in }\caa{D}_0.
\end{cases}
\end{displaymath}
It is not hard to verify its isomorphism with $\caa{D}_1^\bullet$: objects are the same (horizontal endo-1-cells),
and there is a bijective correspondence between the morphisms, essentially given by \cref{cartliftframdebicat}.
Given an arrow $\alpha_f$ in $\caa{D}_1^\bullet$, we obtain a composite 2-cell
\begin{equation}\label{Grothiso}
\xymatrix @R=.25in
{A\ar[r]^-M\ar@{}[r]|-{\scriptstyle{\bullet}}
\rtwocell<\omit>{<3.5>\alpha} \ar[d]_-f & A\ar[d]^-f \\
B\ar[r]_-N\ar@{}[r]|-{\scriptstyle{\bullet}} & B}
\quad\xymatrix@R=.1in
{\hole \\ \mapsto }\quad
 \xymatrix @C=.5in @R=.25in
{A\ar[d]_-{\mathrm{id}_A} \ar[r]^-{1_A}\ar@{}[r]|-{\scriptstyle{\bullet}}
\rtwocell<\omit>{<3.5>\;p_2} &
 A\ar[r]^-M\ar@{}[r]|-{\scriptstyle{\bullet}} \ar[d]^-f 
\rtwocell<\omit>{<3.5>\alpha} & 
A\ar[d]_-f\ar[r]^-{1_A}\ar@{}[r]|-{\scriptstyle{\bullet}}
\rtwocell<\omit>{<3.5>\;q_2} & 
A\ar[d]^{\mathrm{id}_A} \\
A\ar[r]_-{\hat{f}}\ar@{}[r]|-{\scriptstyle{\bullet}} & 
B\ar[r]_-N\ar@{}[r]|-{\scriptstyle{\bullet}} & 
B\ar[r]_-{\check{f}}\ar@{}[r]|-{\scriptstyle{\bullet}} & B}
\end{equation}
which is a morphism in $\Gr{G}\ps{M}$, where $p_2$ and 
$q_2$ come with the companion and conjoint as in \cref{deficompconj}.
This assignment is an isomorphism, with inverse mapping $\beta\mapsto
(q_1\odot 1_N\odot p_1)\beta$ for
some $\beta:M\Rightarrow\check{f}\circ N\circ\hat{f}$
in $\ca{H}(\caa{D})(A,A)$. Thus $\ps{M}$ gives rise to a fibration
$\Gr{G}\ps{M}\to\caa{D}_0$ which is isomorphic to $\caa{D}_1^\bullet\to\caa{D}_0$ 
mapping $G_X$ to $X$ and $\alpha_f$ to $f$:
\begin{displaymath}
\xymatrix @C=.4in @R=.3in
{\Gr{G}\ps{M}\ar[rr]^-{\cong}
\ar[dr] && 
\caa{D}_1^\bullet\ar[dl] \\
& \caa{D}_0 &}
\end{displaymath}

In a very similar way, it can be checked that $\ps{F}$ from \cref{pseudofunctorsdoublebifibration}
is a pseudofunctor, using again standard properties of companions and conjoints.
Therefore, $\Gr{G}\ps{F}\cong\caa{D}_1^\bullet\to\caa{D}_0$ is an opfibration. 
\end{proof}

Notice that $\caa{D}_1^\bullet$ being a bifibration over $\caa{D}_0$ could be deduced from the fibration part combined
with \cref{rmkforadjointintexingbifr}, since we have an  adjunction $(\check{f}\odot\text{-}\odot\hat{f})\vdash
(\hat{f}\odot\text{-}\odot\check{f})$ for all $f$ (\cref{compconjprops}).

Athough the above result was independently established as a generalization of our case study of (co)categories, as will be clear later,
the fibration was also shown in \cite[Proposition 3.3]{Monadsindoublecats}, by
restricting the cartesian liftings \cref{cartliftframdebicat} of the fibration $(\Gr{s},\Gr{t})\colon\caa{D}_1\to\caa{D}_0$
to the category of endomorphisms, i.e.
\begin{gather}\label{cocartliftComonD}
\Cart(f,N)=p_1\odot 1_N\odot q_1\colon\check{f}\odot N\odot\hat{f}\Rightarrow N \\
\Cocart(f,N)=p_2\odot 1_N\odot q_2:N\Rightarrow\hat{f}\odot N\odot \check{f}\nonumber
\end{gather}
Still, the pseudofunctor formulation of the above proof gives clearer perspectives for the objects involved in the applications.

We can now adjust the above constructions to obtain similar results for categories of monads and comonads in fibrant double categories.
The following lemma ensures that is feasible. 

\begin{lem}\label{pseudofunctorsrestrict}
For any vertical 1-cell $f\colon A\to B$ in a fibrant double category
$\caa{D}$, the functors
\begin{gather*}
 \check{f}\odot-\odot\hat{f}\colon\HD(B,B)\longrightarrow\HD(A,A) \\
 \hat{f}\odot-\odot\check{f}\colon\HD(A,A)\longrightarrow\HD(B,B)
\end{gather*}
are lax and colax monoidal respectively, for the monoidal endo-hom-categories of the bicategory $\HD$ with horizontal composition.
\end{lem}

\begin{proof}
For horizontal endo-1-cells$\SelectTips{eu}{10}\xymatrix @C=.2in{M,N:B\ar[r]|-{\sbul} & B,}$the lax
monoidal structure map 
\begin{displaymath}
\phi_{M,N}: \check{f}\odot M\odot\hat{f}\odot\check{f}\odot N\odot\hat{f}
\Rightarrow\check{f}\odot M\odot N\odot\hat{f}
\end{displaymath}
is a natural transformation with components the composite 2-cells
\begin{displaymath}
 \xymatrix @R=.1in 
{&&& A\ar@/^/[dr]|-{\scriptstyle\bullet}^-{\hat{f}} &&& \\
A\ar[r]|-{\sbul}^-{\hat{f}} & B\ar[r]|-{\sbul}^-{N}
& B\ar@/^/[ur]|-{\sbul}^-{\check{f}}\ar[rr]|-{\sbul}_-{1_B} 
\rrtwocell<\omit>{<-2>\fepsilon}\ar@/_5ex/[rrr]|-{\sbul}_-M 
& \rtwocell<\omit>{<2>\cong} & B\ar[r]|-{\sbul}^-M &
B\ar[r]|-{\sbul}^-{\check{f}} &A}
\end{displaymath}
where $\fepsilon$ is the counit of the adjunction $\check{f}\dashv\hat{f}$.
Similarly, $\phi_0$ is given by
\begin{displaymath}
\xymatrix @R=.1in 
{A\ar@/^5ex/[rrr]|-{\scriptstyle\bullet}^-{1_A}\ar[r]|-{\sbul}_-{\hat{f}} &
B\ar[r]|-{\sbul}_-{1_B}\rtwocell<\omit>{<-3>\feta} & B\ar[r]|-{\sbul}_-{\check{f}} & A}
\end{displaymath}
where $\feta$ is the unit of $\check{f}\dashv\hat{f}$. These structure maps satisfy the usual conditions,
hence $\check{f}\odot-\odot\hat{f}$ is a lax monoidal functor; dually, the colax monoidal structure of
$\hat{f}\odot-\odot\check{f}\colon\HD(A,A)\to\HD(B,B)$ can be identified. 
\end{proof}

Due to this lax and colax monoidal structure, the induced monoid structure of $(\check{f}\odot N\odot\hat{f})$
for a monoid $(\proar{N}{B}{B},\mu,\eta)$ and the induced comonoid structure of $(\hat{f}\odot C\odot\check{f})$
for $(\proar{C}{A}{A},\Delta,\epsilon)$ are
\begin{equation}\label{compositemonoid}
\xymatrix @C=.25in @R=.05in
{&&& A\ar @/^/[dr]|-{\sbul}^-{\hat{f}} &&&\\
&& B\ar @/^/[ur]|-{\sbul}^-{\check{f}} 
\ar @/_3ex/[rr]|-{\sbul}_-{1_B}
\rrtwocell<\omit>{\fepsilon} 
&& B\ar @/^/[dr]|-{\sbul}^-N && \\
A\ar[r]|-{\sbul}^-{\hat{f}}
&B\ar @/^/[ur]|-{\sbul}^-N
\ar @/_6ex/[rrrr]|-{\sbul}_-N
& \rrtwocell<\omit>{<3>\;M} &&&
B\ar[r]|-{\sbul}^-{\check{f}} & A}
\xymatrix @C=.3in @R=.05in
{\hole \\
A\ar @/^3ex/[rrr]|-{\sbul}^-{1_A} 
\ar[dr]|-{\sbul}_-{\hat{f}} &&& A \\
& B\ar @/^2ex/[r]|-{\sbul}^-{1_B}
\ar @/_2ex/[r]|-{\sbul}_-N 
\rtwocell<\omit>{\eta} 
\rtwocell<\omit>{<-5>\feta} &
B,\ar[ur]|-{\sbul}_-{\check{f}} &}
\end{equation}

\begin{equation}\label{compositecomonoid}
\xymatrix @C=.25in @R=.05in
{B \ar[r]|-{\sbul}^-{\check{f}}
& A\ar @/^6ex/[rrrr]|-{\sbul}^-C
\ar @/_/[dr]|-{\sbul}^-C
& \rrtwocell<\omit>{<-3>\Delta}
&&&
A\ar[r]|-{\sbul}^-{\hat{f}} & B \\
&& A\ar @/_/[dr]|-{\sbul}_-{\hat{f}}
\ar @/^3ex/[rr]|-{\sbul}^-{1_A}
\rrtwocell<\omit>{\feta} &&
A\ar @/_/[ur]|-{\sbul}^-C && \\
&&&
B\ar @/_/[ur]|-{\sbul}_-{\check{f}} &&&}
\xymatrix @R=.05in @C=.3in
{& A\ar @/_2ex/[r]|-{\sbul}_-{1_A}
\ar @/^2ex/[r]|-{\sbul}^-C 
\rtwocell<\omit>{\epsilon} 
\rtwocell<\omit>{<5.3>\fepsilon} &
A\ar[dr]|-{\sbul}^-{\hat{f}} & \\
B\ar @/_3ex/[rrr]|-{\sbul}_-{1_B} 
\ar[ur]|-{\sbul}^-{\check{f}} &&& B}
\end{equation}
where $\feta,\fepsilon$ are the unit and counit of $\check{f}\dashv\hat{f}$.

The above lemma provides a different, again independent, approach
to the monad fibration proof of \cite[Proposition 3.3]{Monadsindoublecats}.

\begin{prop}\label{MonComonfibred}
If $\caa{D}$ is a fibrant double category, $\Mnd(\caa{D})$ is fibred over $\caa{D}_0$ and $\Cmd(\caa{D})$ is opfibred over $\caa{D}_0$.
\end{prop}

\begin{proof}
We will address the opfibration part; the fibration is established similarly.
Once again, we will construct a pseudofunctor $\ps{K}\colon\caa{D}_0\to\Cat$
for which the Grothendieck construction gives a total category
isomorphic to $\Cmd(\caa{D})$, along with the evident forgetful functor to
$\caa{D}_0$. It resembles $\ps{F}$ from \cref{pseudofunctorsdoublebifibration},
but with fibre categories capturing the desired comonad structure.

An object $A$ is mapped to the category $\Comon(\HD(A,A),\odot,1_A)$ where double categorical monads with source and target $A$ live
(\cref{doublemonadsaremonoids}). A vertical $1$-cell $f\colon A\to B$ is mapped to the functor
\begin{displaymath}
 \ps{K}f:=\hat{f}\odot-\odot\check{f}\colon\Comon(\HD(A,A))\to\Comon(\HD(B,B))
\end{displaymath}
which is precisely the induced $\Comon(\ps{F}f)$ by \cref{pseudofunctorsrestrict}.
The fact that these data form a pseudofunctor follows in a straightforward way from $\ps{F}$ being a pseudofunctor; the natural isomorphisms
$\delta$ and $\gamma$ are given in a dual way to \cref{Mdelta,Mgamma}
using \cref{compconjprops}.

The induced Grothendieck category $\Gr{G}\ps{K}$ has as objects pairs
$(C,A)$ where $C\in\Comon(\HD(A,A))$ for a $0$-cell $A$, and as arrows $(C,A)\to(D,B)$ pairs
\begin{displaymath}
 \begin{cases} 
\hat{f}\odot C\odot \check{f}\xrightarrow{\psi}D 
&\text{in }\Comon(\HD(B,B))\\
A\xrightarrow{f}B &\text{in }\caa{D}_0.
\end{cases}
\end{displaymath}
This category is isomorphic to $\Cmd(\caa{D})$ since they have the same
objects, and there is a bijection between morphisms dually to \cref{Grothiso}
\begin{displaymath}
\xymatrix @R=.25in
{A\ar[r]^-C\ar@{}[r]|-{\scriptstyle{\bullet}}
\rtwocell<\omit>{<3.5>\alpha} \ar[d]_-f & A\ar[d]^-f \\
B\ar[r]_-D\ar@{}[r]|-{\scriptstyle{\bullet}} & B}
\quad\xymatrix@R=.1in
{\hole \\ \mapsto }\quad
 \xymatrix @C=.5in @R=.25in
{B\ar@{=}[d] \ar[r]^-{\check{f}}|-{\scriptstyle{\bullet}} \rtwocell<\omit>{<3.5>\;q_1} &
 A\ar[r]^-C|-{\scriptstyle{\bullet}} \ar[d]^-f \rtwocell<\omit>{<3.5>\alpha} & 
A\ar[d]_-f\ar[r]^-{\hat{f}}|-{\scriptstyle{\bullet}}\rtwocell<\omit>{<3.5>\;p_1} & B\ar@{=}[d] \\
B\ar[r]_-{1_B}|-{\scriptstyle{\bullet}} & B\ar[r]_-D|-{\scriptstyle{\bullet}} & 
B\ar[r]_-{1_B}|-{\scriptstyle{\bullet}} & B} 
\end{displaymath}
where $q_1,p_2$ are as in \cref{deficompconj}; checking that $p_1\odot\alpha\odot q_1$
is a comonoid morphism follows from their properties. 
\end{proof}

Finally, as the following result shows,
these fibrations and opfibrations have a monoidal structure in the sense of \cref{monoidalfibration},
when $\caa{D}$ is moreover monoidal.

\begin{prop}\label{monadscomonadsmonoidalfibr}
Suppose that $\caa{D}$ is a fibrant monoidal double category. The bifibration $T\colon\caa{D}_1^\bullet\to\caa{D}_0$
as well as the fibration $S\colon\Mnd(\caa{D})\to\caa{D}_0$ and opfibration $W\colon\Cmd(\caa{D})\to\caa{D}_0$ are monoidal.
\end{prop}

\begin{proof}
By \cref{DendoMonDComonDmonoidal} and \cref{monoidaldoublecategory} of a monoidal double category, all categories
involved are monoidal. Moreover, for horizontal endo-1-cells$\proar{M}{A}{A}$and$\proar{N}{B}{B,}$
\begin{displaymath}
T(M\otimes_1 N)=A\otimes_0 B=T(M)\otimes_1T(N)
\end{displaymath}
and similarly for $P$ and $W$,
by the mapping of $\otimes_1$ \cref{D1monoidal} whose restriction on $\caa{D}_1^\bullet,\Mnd(\caa{D}),\Cmd(\caa{D})$ is their tensor product.

Finally, $\otimes_1^\bullet\colon\caa{D}_1^\bullet\times\caa{D}_1^\bullet\to\caa{D}_1^\bullet$ preserves cartesian arrows:
a pair of cartesian liftings $\Cart(f,M)$ and $\Cart(g,N)$ as in \cref{cocartliftComonD} is mapped to the top arrow
\begin{displaymath}
\xymatrix @C=.6in @R=.3in
{(\hat{f}\odot M\odot\check{f})\otimes_1(\hat{g}\odot N\odot\check{g})\ar[rr]^{\qquad\quad\Cart(f,M)\otimes_1\Cart(g,N)}
\ar @{-->}[d]_-{\cong} && M\otimes_1N\ar @{.>}[dd] &\\
\wh{(f\otimes_0 g)}\odot(M\otimes_1 N)\odot\wc{(f\otimes_0 g)}\ar[urr]_-{\quad\Cart(f\otimes_0g,M\otimes_1 N)}\ar @{.>}[d]
&&& \textrm{in }\caa{D}_1^\bullet \\
A\otimes_0C\ar[rr]_-{f\otimes_0g} && B\otimes_0C & \textrm{in }\caa{D}_0}
\end{displaymath}
where the left side isomorphism is obtained by \cref{monoidaldoubleiso} and \cref{compconjprops},
\begin{gather*}
 \left((\hat{f}\odot M)\odot\check{f}\right)\otimes_1\Big((\hat{g}\odot N)\odot\check{g}\Big)\cong
 \left((\hat{f}\odot M)\otimes_1(\hat{g}\odot N)\right)\odot(\check{f}\otimes_1\check{g}) \\
 \cong (\hat{f}\otimes_1\hat{g})\odot(M\otimes_1 N)\odot(\check{f}\otimes_1\check{g})
\end{gather*}
This vertical isomorphism can be shown to make the triangle commute. 
The proof is thus complete, since this vertical isomorphism is reflected to monads and a dual cocartesian lifting
triangle to comonads.
\end{proof}

\subsection{Locally closed monoidal double categories}\label{locallyclosedmonoidaldoublecats}

When $\caa{D}$ is a monoidal double category as in \cref{monoidaldoublecategory},
both vertical and horizontal categories are endowed with a monoidal structure, $(\caa{D}_0,\otimes_0,I)$ and $(\caa{D}_1,\otimes_1,1_I)$.
Naturally, one could expect that the appropriate notion of a \emph{monoidal closed}
double category would result in a similar `local' closed structure for the two categories $\caa{D}_0$ and $\caa{D}_1$.
For the following existing definition though, this does not seem to be the case.

\begin{defi}\cite[\S 5]{Adjointfordoublecats}
A \emph{(weakly) monoidal closed} pseudo double category $\caa{D}$ is a monoidal
double category such that each pseudo double functor $(\text{-}\otimes D):\caa{D}\to\caa{D}$
has a lax right adjoint.
\end{defi}

This definition uses \emph{lax adjunctions} between pseudo double categories, as described in \cite[3.2]{Adjointfordoublecats};
double adjunctions are also studied in \cite{Doubleadjunctionsandfreemonads}. Since these do not end up being relevant
to the current work, we omit their details and simply discuss what they mean in this particular context.

The lax double functor $\text{-}\otimes D$ consists of the ordinary functors
$(\text{-}\otimes_0D\colon\caa{D}_0\to\caa{D}_0,\text{-}\otimes_11_D\colon\caa{D}_1\to\caa{D}_1)$.
The existence of a lax right double adjoint, call it $\Hom^\caa{D}(D,-)$,
amounts in particular to two ordinary adjunctions
\begin{displaymath}
\xymatrix @C=.7in
{\caa{D}_0\ar@<+.8ex>[r]^-{-\otimes_0D}
\ar@{}[r]|-{\bot} &
\caa{D}_0\ar@<+.8ex>[l]^-{\Hom^\caa{D}_0(D,-)}}, \qquad
\xymatrix @C=.6in
{\caa{D}_1\ar@<+.8ex>[r]^-{-\otimes_11_D}
\ar@{}[r]|-{\bot} &
\caa{D}_1\ar@<+.8ex>[l]^-{\Hom^\caa{D}_1(1_D,-)}}
\end{displaymath}
for any 0-cell $D$ in $\caa{D}$, such that conditions expressing compatibility with the horizontal composition 
and identities are satisfied. It immediately follows that $\caa{D}_0$ is a monoidal closed
category; however this cannot be deduced for $\caa{D}_1$ as well, since
$1_D$ is not an arbitrary horizontal 1-cell.

Due to the application of these notions to our context of interest later,
we proceed to the definition of a different closed-like structure
which arises naturally in what follows.

\begin{defi}\label{loccloseddoublecat}
A monoidal (pseudo) double category $\caa{D}$ is called \emph{locally closed monoidal}
if it comes equipped with a lax double functor
\begin{displaymath}
 H=(H_0,H_1)\colon\caa{D}^\op\times\caa{D}\longrightarrow\caa{D}
\end{displaymath}
such that $\otimes_0\dashv H_0$ and $\otimes_1\dashv H_1$ are parametrized adjunctions.
\end{defi}
By definition \cref{F1mapping}, the functor $H_1$ in particular is the mapping
\begin{equation}\label{H1mapping}
H_1:\xymatrix @C=1.5in{\caa{D}^\op_1\times\caa{D}_1\ar[r] & \caa{D}_1\phantom{ABC}}
\end{equation}\vspace{-0.2in}
\begin{displaymath}
\xymatrix @C=.08in
{(X\ar[rrr]|-\sbul^M\ar[d]_-f &\rtwocell<\omit>{<4>{\alpha}}&& Y,\ar[d]^-g & Z\ar[rrr]^-N|-\sbul\ar[d]_-h
&\rtwocell<\omit>{<4>\beta}&& W)\ar[d]^-k
\ar@{|.>}[rrrr] &&&& H_0(X,Z)\ar[rrr]^-{H_1(M,N)}|-\sbul
\ar[d]_-{H_0(f,h)} &\rtwocell<\omit>{<4>\qquad H_1(\alpha,\beta)} && H_0(Y,W)\ar[d]^-{H_0(g,k)} \\
(X'\ar[rrr]_-{M'}|-\sbul &&& Y', & Z'\ar[rrr]_-{N'}|-\sbul &&& W')
\ar@{|.>}[rrrr] &&&& H_0(X',Z')\ar[rrr]_-{H_1(M',N')}|-\sbul &&& H_0(Y',W')} 
\end{displaymath}
Call $H$ the \emph{internal hom} of $\caa{D}$.
Clearly $H_0$ gives a monoidal closed structure on the vertical monoidal category
$(\caa{D}_0,\otimes_0,I)$ and $H_1$ on the horizontal category
$(\caa{D}_1,\otimes_1,1_I)$. The above arguments justify that a monoidal closed structure
on a double category does not imply a locally closed monoidal structure.

We will now explore the relations of this lax double functor $H$ on a locally closed monoidal double category $\caa{D}$
with the categories of endomorphisms, monads and comonads discussed in \cref{moncomondouble}. Recall that by \cref{DendoMonDComonDmonoidal},
all these categories inherit a monoidal structure from $\Dar$.

\begin{prop}\label{Dendoclosed}
Suppose $\caa{D}$ is a locally closed monoidal double category, with internal
hom $H=(H_0,H_1)$. Then $H_1^\bullet$ endows the category
of endomorphisms $\caa{D}_1^\bullet$ with a monoidal closed structure.
\end{prop}

\begin{proof}
The lax double functor $H$ induces $H_1^\bullet\colon{\Dendo}^\op\times\Dendo\to\Dendo$ by \cref{F1bullet}.
The natural isomorphism $\caa{D}_1(M\otimes_1 N,P)\cong\caa{D}_1(M,H_1(N,P))$
which defines the adjunction $(-\otimes_1 N)\dashv H_1(N,-)$ for the monoidal closed
category $\Dar$, considered only on endo-1-cells and endo-2-morphisms,
implies that $\caa{D}_1^\bullet$ is also a monoidal closed category via
$\otimes_1^\bullet\dashv H_1^\bullet$.
\end{proof}

By \cref{MonFdouble}, the lax double functor $H\colon\caa{D}^\op\times\caa{D}\to\caa{D}$  induces
an ordinary functor between the categories of (co)monads
\begin{equation}\label{MonHdouble}
\Mon H:\Cmd(\caa{D})^\op\times\Mnd(\caa{D})\to\Mnd(\caa{D})
\end{equation}
which is $H_1^\bullet$ restricted on $\Mnd(\caa{D}^\op\times\caa{D})\cong\Mnd(\caa{D}^\op)\times\Mnd(\caa{D})$.
This functor is fundamental in order to study enrichment relations between the category of monads and comonads
in a braided or symmetric locally closed monoidal (fibrant) double category, by applying results from \cref{sec:actionenrich,fibrations}.
For example, it is an action as explained below.

\begin{lem}\label{doubleMonHaction}
The functor $\Mon H\colon\Cmd(\caa{D})^\op\times\Mnd(\caa{D})\to\Mnd(\caa{D})$
in a braided locally closed monoidal double category, as well as $(\Mon H)^\op$, is an action.
\end{lem}
\begin{proof}
The induced monoidal closed structure $H_1^\bullet$ on the braided monoidal $\Dendo$, as well as its opposite functor,
are both actions of ${\Dendo}^\op$ and $\Dendo$ respectively, as is the case in any braided monoidal closed category by \cref{inthomaction}.
Therefore there are structure isomorphisms as in \cref{actionmaps},
\begin{displaymath}
H_1^\bullet(M\otimes N,P)\cong H_1^\bullet(M,H_1^\bullet(N,P)),\quad H_1^\bullet(I,P)\cong P 
\end{displaymath}
for any endo-1-cells $M,N,P$ satisfying compatibility conditions.
Since the forgetful functors from $\Cmd(\caa{D})$, $\Mnd(\caa{D})$ to $\Dendo$ reflect isomorphisms,
$\Mon H$ and its opposite come equipped with these isomorphisms applied to
$M,N\in\Cmd(\caa{D})$, $P\in\Mnd(\caa{D})$ thus are actions.
\end{proof}

\begin{thm}\label{MndenrichedCmnd}
Suppose that $(\caa{D},\otimes,\B{I})$ is a braided locally closed monoidal double category,
with internal hom $H$. If the induced functor $\Mon H$ has a parametrized right adjoint, then the category of monads $\Mnd(\caa{D})$ is
enriched in the category of comonads $\Cmd(\caa{D})$.

Moreover, if $\Cmd(\caa{D})$ is monoidal closed this enrichment is cotensored, and if each $\Mon H(M,-)$ also has a right adjoint,
the enrichment is tensored.
\end{thm}
\begin{proof}
The category of comonads is braided monoidal by \cref{DendoMonDComonDmonoidal}.
Since $\Mon H$ and $(\Mon H)^\op\colon\Cmd(\caa{D})\times\Mnd(\caa{D})^\op\to\Mnd(\caa{D})^\op$ are actions,
the existence of a right adjoint for the latter
\begin{displaymath}
S\colon\Mnd(\caa{D})^\op\times\Mnd(\caa{D})\longrightarrow\Cmd(\caa{D}) 
\end{displaymath}
induces the desired enrichment of $\Mnd(\caa{D})^\op$ and thus of $\Mnd(\caa{D})$, by \cref{actionenrich}.
The rest of the clauses follow by assumption.
\end{proof}

Notice that monoidal closedness of $\Cmd(\caa{D})$ does not seem to follow in a straightforward way from that of $\Dendo$,
\cref{Dendoclosed}. Even in the application that follows, this result is obtained after establishing (co)completeness
and (co)monadicity properties for the specific structure, which are heavily related to the double category we choose to work in;
see \cref{VCocatclosed}. 

Moving to the fibrant case, we would now want to combine the above enrichment with the (op)fibration structure of monads
and comonads over $\caa{D}_0$ to exhibit an enriched fibration, \cref{enrichedfibration}.
Towards that end, notice that
\begin{equation}\label{hopefulaction}
\xymatrix @C=.7in @R=.5in
{\Cmd(\caa{D})\times\Mnd(\caa{D})^\op\ar[r]^-{(\Mon H)^\op} \ar[d]_-{W\times S^\op} & \Mnd(\caa{D})^\op\ar[d]^{S^\op} \\
\caa{D}_0\times(\caa{D}_0)^\op\ar[r]_-{H_0^\op} & \caa{D}_0^\op}
\end{equation}
commutes by definition of the involved functors, and moreover $W$, $S^\op$ are monoidal opfibrations by \cref{monadscomonadsmonoidalfibr}.
Finally, the actions $(\Mon H)^\op$ and $(H_0)^\op$ are above each other in the sense of \cref{actionsabove}:
by the mapping \cref{H1mapping} that restricts as $\Mon H$ to (co)monads, the source and target
of the action isomorphisms in $\Mnd(\caa{D})$ is precisely the action isomorphism in $\caa{D}_0$, e.g.
\begin{displaymath}
 \xymatrix @C=.8in
{H_0(X\otimes_0Y,Z)\ar[d]_-{\cong}\ar[r]|-\sbul^-{H_1(C\otimes_1D,M)}\rtwocell<\omit>{<4>\cong} & H_0(X\otimes_0Y,Z)\ar[d]^-{\cong} \\
H_0(X,H_0(Y,Z))\ar[r]_-{H_1(C,H_1(D,M))}|-\sbul & H_0(X,H_0(Y,Z))}
\end{displaymath}
As a result, when $\Mon H$ preserves cartesian liftings, then $W$ acts on $S^\op$ in the sense of \cref{Trepresentation}.
We can now apply the dual of \cref{thmactionenrichedfibration}.

\begin{thm}\label{MndfibredenrichedCmnd}
Suppose $\caa{D}$ is a braided locally closed monoidal fibrant double category.
If $(\Mon H)^\op$ is cocartesian, and \cref{hopefulaction} has an opfibred parametri\-zed
adjoint as in \cref{generaloplaxparametrized},
then the fibration $S\colon\Mnd(\caa{D})\to\caa{D}_0$ is enriched in the symmetric monoidal opfibration
$W\colon\Cmd(\caa{D})\to\caa{D}_0$.
\end{thm}

\section{Enriched matrices and (co)categories}\label{sec:Enrichedmatrices}

In this section, we initially study the double category of $\ca{V}$-matrices. 
After establishing its monoidal fibrant double categorical structure,
special focus will be given to its well-known horizontal bicategory of enriched matrices. The main references
for that are \cite{VarThrEnr} and \cite{KellyLack}; in the former, the more general bicategory
$\ca{W}$-$\Mat$ of matrices enriched in a bicategory $\ca{W}$ was studied, leading to the theory of bicategory-enriched categories.

Furthermore, we investigate the categories of monads and comonads in the double category
$\VMMat$. These are specifically $\VCat$ of $\ca{V}$-enriched categories
and functors \cite{Kelly}, and $\VCocat$ of $\ca{V}$-enriched \emph{cocategories} and \emph{cofunctors}.
Applying earlier results, passing from the double categorical to the bicategorical view according to our needs,
the goal is to establish an enrichment of $\ca{V}$-categories in $\ca{V}$-cocategories.

As an intermediate step, $\ca{V}$-enriched graphs are given special attention; they provide a natural
common framework for (co)categories, since both are graphs with extra structure.
In \cite{MacLane}, the notion of a small (directed) graph
is employed to describe the free category construction (in analogy with the free monoid construction on a set)
and also $O$-graphs with a fixed set of objects $O$ inspires the fibrational view of these categories.
For $\ca{V}$-$\B{Grph}$ and $\ca{V}$-$\B{Cat}$ from a more traditional point 
of view, rather than the matrices approach followed here, Wolff's \cite{Wolff} is a classical reference for a symmetric monoidal closed base.

There is a 2-dimensional aspect for all the categories studied in this chapter,
e.g. $\ca{V}$-natural transformations. We choose to omit its description in this treatment, since it is not relevant to
our main objectives.

\subsection{\texorpdfstring{$\ca{V}$}{V}-matrices}\label{bicatVMat}

Suppose that $\ca{V}$ is a monoidal category with coproducts which are preserved by the tensor product functor
$-\otimes-$ in each variable; for example, this is certainly the case when $\ca{V}$ is monoidal closed.

For sets $X$ and $Y$, a $\ca{V}$\emph{-matrix}$\SelectTips{eu}{10}\xymatrix @C=.2in
{S:X\ar[r]|-{\object@{|}} & Y}$from $X$ to $Y$ is defined to be a functor $S:X\times Y\to\ca{V}$,
where the set $X\times Y$ is viewed as a discrete category. This can equivalently given by a family of objects in $\ca{V}$
$$\{S(x,y)\}_{(x,y)\in X\times Y}$$
sometimes also denoted as $\{S_{x,y}\}_{X\times Y}$.
For example, each set $X$ gives rise to a $\ca{V}$-matrix$\SelectTips{eu}{10}
\xymatrix @C=.2in{1_X:X\ar[r]|-{\object@{|}} & X}$called the \emph{identity matrix} given by
\begin{displaymath}
1_X(x,x')=\begin{cases}
I,\quad \mathrm{if  }\;x=x'\\
0,\quad \mathrm{ otherwise}
\end{cases}
\end{displaymath}
where $I$ is the unit object in $\ca{V}$ and $0$ is the initial object.

There is a double category $\ca{V}$-$\MMat$ of $\ca{V}$-matrices as in \cref{def:doublecats}, with vertical category $\ca{V}$-$\MMat_0$
the usual category of sets and functions $\Set$. Its horizontal category $\ca{V}$-$\MMat_1$
consists of $\ca{V}$-matrices$\SelectTips{eu}{10}\xymatrix@C=.2in
{S:X\ar[r]|-{\object@{|}} & Y}$as horizontal 1-cells, and 2-morphisms $^f\alpha^g:S\Rightarrow T$
are natural transformations
\begin{equation}\label{VMat2morphism}
\xymatrix@R=.1in
{X\ar[r]^-S\ar@{}[r]|-{\tick}\rtwocell<\omit>{<5>\alpha} \ar[dd]_-f & Y\ar[dd]^-g  &&&\\
&& \textrm{=} & X\times Y \ar@/^5ex/[rr]^-S\ar@/_/[dr]_{f\times g}\rrtwocell<\omit>{\alpha} && \ca{V} \\
Z\ar[r]_-T\ar@{}[r]|-{\tick} & W &&& Z\times W\ar@/_/[ur]_-T &}
\end{equation}
given by families of arrows $\alpha_{x,y}:S(x,y)\to T(fx,gy)$ in $\ca{V}$, for all $x\in X$ and $y\in Y$.
There is a functor $\B{1}\colon\VMMat_0\to\VMMat_1$ which gives the identity $\ca{V}$-matrix$\SelectTips{eu}{10}
\xymatrix@C=.2in{1_X:X\ar[r]|-{\object@{|}} & X}$for each $X$, and the unit 2-morphism $1_f$ with components
\begin{displaymath}
 (1_f)_{x,x'}:1_X(x,x')\to1_X(x,x')\equiv
\begin{cases}
 I\xrightarrow{1_I}I, &\textrm{ if}\;x=x' \\
0\to 0, &\textrm{ if}\;x\neq x'.
\end{cases}
\end{displaymath}
The source and target functors give the evident sets and functions, and the horizontal composition functor
\begin{displaymath}
 \odot:\ca{V}\text{-}\MMat_1{\times_{\ca{V}\text{-}\MMat_0}}
\ca{V}\text{-}\MMat_1\to\ca{V}\text{-}\MMat_1
\end{displaymath}
maps two composable $\ca{V}$-matrices$\SelectTips{eu}{10}\xymatrix @C=.2in{T:Y\ar[r]|-{\object@{|}} & Z}$and$\SelectTips{eu}{10}
\xymatrix @C=.2in{S:X\ar[r]|-{\object@{|}} & Y}$to the matrix
$\SelectTips{eu}{10}\xymatrix @C=.2in{T\circ S:X\ar[r]|-{\object@{|}} & Z,}$given
by the family of objects in $\ca{V}$
\begin{equation}\label{horizontalcompositionVmatrices}
(T\circ S)(x,z)=\sum_{y\in Y} S(x,y)\otimes T(y,z)
\end{equation}
for all $z\in Z$ and $x\in X$, reminiscent of the usual matrix multiplication.
The horizontal composite of 2-morphisms $^f(\beta\odot\alpha)^g:T\circ S\Rightarrow T'\circ S'$ as in
\cref{2cellscomp} is given by the composite arrows
\begin{equation}\label{horizontalcompositionVmatricearrows}
\xymatrix @C=1in @R=.25in
{\sum_{y} S(x,y)\otimes T(y,z)\ar[r]^-{\sum\alpha_{x,y}\otimes\beta_{y,z}} \ar@/_3ex/@{-->}[dr] &
\sum_{y}S'(fx,gy)\otimes T'(gy,hz)\ar@{_(->}[d] \\
& \sum_{y'}S'(fx,y')\otimes T'(y',hz)}
\end{equation}
in $\ca{V}$, for all $x\in X$ and $z\in Z$. Compatibility conditions of source and target 
functors with composition can be easily checked.

For composable $\ca{V}$-matrices$\SelectTips{eu}{10}\xymatrix @C=.2in{X\ar[r]|-{\object@{|}}^-S & 
Y\ar[r]|-{\object@{|}}^-T & Z\ar[r]|-{\object@{|}}^-R & W,}$
the associator $\alpha$ has components globular isomorphisms $\alpha^{R,T,S}:
(R\odot T)\odot S\stackrel{\sim}{\longrightarrow}R\odot(T\odot S)$
given by the family $\{\alpha_{x,w}\}$ of composite isomorphisms
\begin{displaymath}
\xymatrix @R=.53in
{\sum_z\left(\sum_y S_{x,y}\otimes T_{y,z}\right)\otimes R_{z,w}\ar@{-->}[r] \ar[d]_-{\cong} &
\sum_y S_{x,y}\otimes\left(\sum_z T_{y,z}\otimes R_{z,w}\right) \\
\sum_{y,z}\left((S_{x,y}\otimes T_{y,z})\otimes R_{z,w}\right)\ar[r]_-{\sum{a}} & 
\sum_{y,z}\left(S_{x,y}\otimes(T_{y,z}\otimes R_{z,w})\right)\ar[u]_-{\cong}}
\end{displaymath}
where $a$ is the associativity constraint of $\ca{V}$ and the invertible arrows express the fact that $\otimes$ commutes with colimits.
Finally, for each $\ca{V}$-matrix$\SelectTips{eu}{10}\xymatrix @C=.2in{S:X\ar[r]|-{\object@{|}} & Y,}$ 
the unitors $\lambda, \rho$ have components globular $\lambda^S:1_Y\odot S\simrightarrow S,
\rho^S:S\odot 1_X\simrightarrow S$ given by families
\begin{align*}
\lambda^S_{x,y}&:\sum_{y'\in Y}{S(x,y')\otimes1_Y(y',y) }\equiv S(x,y)\otimes I\xrightarrow{\;r_{S(x,y)}\;} S(x,y) \\
\rho^S_{x,y}&:\sum_{x'\in X}{1_X(x,x')\otimes S(x',y)}\equiv I\otimes S(x,y)\xrightarrow{\;l_{S(x,y)}\;} S(x,y).
\end{align*}
of morphisms in $\ca{V}$. The respective coherence conditions are satisfied, thus these data indeed define a double category.

Moreover, $\VMMat$ is fibrant as in \cref{fibrantdoublecat}. Any function $f:X\to Y$ determines two $\ca{V}$-matrices, $\SelectTips{eu}{10}
\xymatrix @C=.2in{f_*:X\ar[r]|-{\object@{|}} & Y}$and$\SelectTips{eu}{10}\xymatrix @C=.2in
{f^*:Y\ar[r]|-{\object@{|}} & X,}$ given by
\begin{equation}\label{f*}
f_*(x,y)=f^*(y,x)=\begin{cases}
I,\quad \mathrm{if  }\;f(x)=y\\
0,\quad \mathrm{ otherwise}
\end{cases}
\end{equation}
These are precisely the companion and the conjoint of the vertical 1-cell $f$, since
they come equipped with appropriate 2-cells as in \cref{deficompconj}:
\begin{displaymath}
\xymatrix{X\ar[r]^-{f_*}\ar@{}[r]|-{\tick} \rtwocell<\omit>{<4>\;p_1} \ar[d]_-f & Y\ar@{=}[d] \\
Y\ar[r]_-{1_Y}\ar@{}[r]|-{\tick} & Y}\qquad
\xymatrix{X\ar[r]^-{1_X}|-{\tick}\rtwocell<\omit>{<4>\;p_2} \ar@{=}[d] &  X\ar[d]^-{f} \\
X\ar[r]_-{f_*}\ar@{}[r]|-{\tick} & Y}\quad
\xymatrix{\hole \\
\textrm{are given by}}
\end{displaymath}
\begin{gather*}
f_*(x,y)\xrightarrow{(p_1)_{x,y}}1_Y(fx,y)=
\begin{cases}
I\xrightarrow{\id}I, & \text{if }y=fx\\
0\xrightarrow{\id}0, & \text{otherwise}
\end{cases} \\
1_X(x,x')\xrightarrow{(p_2)_{x,x'}}f_*(x,fx')=
\begin{cases}
I\xrightarrow{\id} I, & \text{if }x=x' \\
0\xrightarrow{!}\begin{cases}I, & fx=fx' \\
0, &\text{else}
\end{cases} & \text{if }x\neq x'
\end{cases}
\end{gather*}
satisfying the required relations, and similarly for $f^*$.

When $\ca{V}$ is braided monoidal, the double category $\VMMat$ has a monoidal structure as in \cref{monoidaldoublecategory}.
The required double functors $\otimes=(\otimes_0,\otimes_1)$ and $\B{I}=(\B{I}_0,\B{I}_1)$ consist of the cartesian monoidal structure
on the vertical category $(\B{Set},\times,\{*\})$ and
\begin{equation}\label{VMMat1monoidal}
\otimes_1:\xymatrix @C=1.2in
{\ca{V}\text{-}\MMat_1\times\ca{V}\text{-}\MMat_1
\ar[r] & \ca{V}\text{-}\MMat_1}\phantom{ABC}
\end{equation}\vspace{-0.2in}
\begin{displaymath}
 \xymatrix @C=.025in
{(X\ar[rrr]|-{\object@{|}}^S\ar[d]_-f
&\rtwocell<\omit>{<4>{\alpha}}&& Y\ar[d]^-g
& , & Z\ar[rrr]^-T|-{\object@{|}}\ar[d]_-h
&\rtwocell<\omit>{<4>\beta}&& W)\ar[d]^-k
\ar@{|.>}[rrrr] &&&& X\times Z\ar[rrr]^-{S\otimes T}|-{\object@{|}}
\ar[d]_-{f\times h} &\rtwocell<\omit>{<4>\quad\alpha\otimes\beta}
&& Y\times W\ar[d]^-{g\times k} \\
(X'\ar[rrr]_-{S'}|-{\object@{|}} &&& Y' & , &
Z'\ar[rrr]_-{T'}|-{\object@{|}} &&& W')
\ar@{|.>}[rrrr] &&&& X'\times Z'\ar[rrr]_-{S'\otimes T'}|-{\object@{|}} 
&&& Y'\times W'} 
\end{displaymath}
which is defined on objects and morphisms by the families in $\ca{V}$
\begin{gather}\label{monoidalVMat}
(S\otimes T)\left((x,z),(y,w)\right):=S(x,y)\otimes T(z,w) \\
(\alpha\otimes\beta)_{(x,z),(y,w)}:= S(x,y)\otimes T(z,w)\xrightarrow{\alpha_{x,y}\otimes\beta_{z,w}}S'(fx,gy)\otimes T'(hz,kw).
\nonumber
\end{gather}
Along with the $\ca{V}$-matrix $\SelectTips{eu}{10}\xymatrix@C=.2in{\ca{I}:\{*\}\ar[r]|-{\object@{|}} & \{*\}}$given by
$\ca{I}(*,*)=I_\ca{V}$, this defines a monoidal structure of $\VMMat_1$. The conditions for $\Gr{s}$ and $\Gr{t}$
are satisfied, and the globular isomorphisms \cref{monoidaldoubleiso} come down to the tensor product in
$\ca{V}$ commuting with coproducts. More specifically,
for matrices$\matr{S}{X}{Y,}\matr{T}{Z}{W,}$ $\matr{S'}{Y}{U,}\matr{T'}{W}{V}$we can compute the isomorphic families in $\ca{V}$
\begin{gather*}
\left((S\otimes_1 T)\odot(S'\otimes_1 T')\right)_{(x,z),(u,v)}=
\sum_{(y,w)}S_{x,y}\otimes T_{z,w}\otimes S'_{y,u}\otimes T'_{w,v} \\
\left((S\odot S')\otimes_1(T\odot T')\right)_{(x,z),(u,v)}=\sum_{y}\left(S_{x,y}\otimes S'_{y,u}\right)\otimes
\sum_{w}\left(T_{z,w}\otimes T'_{w,v}\right)
\end{gather*}
via the braiding, and also $1_{X\times Y}\cong 1_X\otimes_1 1_Y$ in a straightforward way.
This monoidal structure on $\VMMat$ is symmetric, when
the base monoidal category $\ca{V}$ is symmetric. Then $\Set$ and $\VMMat_1$ are both symmetric monoidal categories,
and the rest of the axioms follow.

\begin{prop}\label{VMatmonoidaldoublefibrant}
If $\ca{V}$ is a monoidal category with coproducts, such that the tensor product preserves them in both variables,
the double category $\VMMat$ is fibrant. Moreover, $\VMMat$ is monoidal if $\ca{V}$ is braided monoidal, and
inherits the braided or symmetric structure from $\ca{V}$.
\end{prop}

\begin{rmk}
It is evident that there is a strong relation between $\VMMat$ and $\ca{V}$-$\caa{P}\B{rof}$, the double category of $\ca{V}$-profunctors.
On a first level, we can see that $\ca{V}$-matrices are special cases of $\ca{V}$-profunctors when the latter
are considered only on discrete categories. In \cite[11.8]{Framedbicats}, a more elaborate relation between these double categories
is established: $\caa{M}\B{od}(\VMMat)=\ca{V}$-$\caa{P}\B{rof}$ for a construction $\caa{M}\B{od}$ building new fibrant double categories
from old, with vertical category that of monads.
\end{rmk}

When $\ca{V}$ is moreover monoidal closed with products, we can determine a locally closed monoidal structure on $\VMMat$.
Following \cref{loccloseddoublecat}, we are after a lax double functor
\begin{equation}\label{HforVMMat}
 H=(H_0,H_1)\colon\VMMat^\op\times\VMMat\to\VMMat
\end{equation}
which endows $\VMMat_0=\Set$ and $\VMMat_1$ with a monoidal closed structure.
For the vertical category, we clearly have the exponentiation functor
\begin{displaymath}
 H_0:\B{Set}^\op\times\B{Set}\xrightarrow{\;(-)^{(-)}\;}\B{Set}
\end{displaymath}
as the internal hom. For the horizontal category, we can define
\begin{equation}\label{H1functor}
 H_1:\xymatrix @C=1.3in
{\ca{V}\text{-}\MMat_1^\op\times\ca{V}\text{-}\MMat_1
\ar[r] & \ca{V}\text{-}\MMat_1}\phantom{ABC}
\end{equation}\vspace{-0.2in}
\begin{displaymath}
 \xymatrix @C=.025in @R=.25in
{(X\ar[rrr]|-{\object@{|}}^S\ar[d]_-f
&\rtwocell<\omit>{<4>{\alpha}}&& Y\ar[d]^-g
& , & Z\ar[rrr]^-T|-{\object@{|}}\ar[d]_-h
&\rtwocell<\omit>{<4>\beta}&& W)\ar[d]^-k
\ar@{|.>}[rrr] &&& Z^X\ar[rrrr]^-{H_1(S,T)}|-{\object@{|}}
\ar[d]_-{h^f} &\rtwocell<\omit>{<4>\qquad H_1(\alpha,\beta)}
&&& W^Y\ar[d]^-{k^g} \\
(X'\ar[rrr]_-{S'}|-{\object@{|}} &&& Y' & , &
Z'\ar[rrr]_-{T'}|-{\object@{|}} &&& W')
\ar@{|.>}[rrr] &&& Z'^{X'}\ar[rrrr]_-{H_1(S',T')}|-{\object@{|}} 
&&&& {W'}^{Y'}}
\end{displaymath}
on horizontal 1-cells given by families of objects in $\ca{V}$, for $n\in Z^X, m\in W^Y$, 
\begin{equation}\label{H1onobjects}
H_1(S,T)(n,m)=\prod_{x,y}[S(x,y),T(n(x),m(y))] 
\end{equation}
and on 2-morphisms $H_1(\alpha,\beta):H_1(S,T)(n,m)\to H_1(S',T')(h^f(n),k^g(m))$ by families of arrows in $\ca{V}$
\begin{equation}\label{H1onarrows}
\prod_{x,y} [S(x,y),T(n(x),m(y))]\to \prod_{x',y'}[S'(x',y'),T'(hnf(x'),kmg(y'))]
\end{equation}
which correspond under $(\text{-}\otimes X)\dashv [X,-]$ in $\ca{V}$ for fixed $x',y'$ to the composite
\begin{displaymath}
\xymatrix @C=.8in
{\prod\limits_{x,y}[S(x,y),T(nx,my)]\otimes S'(x',y')\ar @{-->}[r]
\ar[d]_-{1\otimes \alpha_{x',y'}} & T'(hnfx',kmgy') \\
\prod\limits_{x,y}[S(x,y),T(nx,my)]\otimes S(fx',gy')\ar[d]_-{\pi_{fx',gy'}\otimes 1} & \\
[S(fx',gy'),T(nfx',mgy')]\otimes S(fx',gy')\ar[r]^-{\textrm{ev}}& T(nfx',mgy').\ar[uu]_-{\beta_{nfx',mgy'}}}
\end{displaymath}
The globular transformations from \cref{defi:doublefunctor} for a lax double functor
\begin{displaymath}
 \xymatrix @R=.02in @C=.8in
{& W^Z \ar @/^/[dr]|-{\object@{|}}^-{H_1(R,O)} & \\
Y^X \ar @/^/[ur]|-{\object@{|}}^-{H_1(S,T)}
\rrtwocell<\omit>{\qquad\qquad\delta_{(S,T),(R,O)}}
\ar @/_3ex/[rr]|-{\object@{|}}_-{H_1(R\circ S,O\circ T)} && V^U}
\quad
\xymatrix @C=1in @R=.02in
{\hole \\
Y^X\ar @/^3ex/[r]|-{\object@{|}}^-{1_{Y^X}}
\ar@/_3ex/[r]|-{\object@{|}}_-{H_1(1_X,1_Y)}
\rtwocell<\omit>{\qquad\;\;\gamma_{(X,Y)}} & Y^X}
\end{displaymath}
for each$\SelectTips{eu}{10}\xymatrix @C=.2in{(R:Z\ar[r]|-{\object@{|}} & U,}\SelectTips{eu}{10}\xymatrix @C=.2in
{O:W\ar[r]|-{\object@{|}} & V)}$and$\SelectTips{eu}{10}\xymatrix @C=.2in{(S:X\ar[r]|-{\object@{|}} & Z,}\SelectTips{eu}{10}
\xymatrix @C=.2in{T:Y\ar[r]|-{\object@{|}} & W)}$are given by families of arrows in $\ca{V}$
\begin{gather*}
\sum_{q\in W^Z}{H_1(S,T)(k,q)\otimes H_1(R,O)(q,t)}\xrightarrow{\delta_{k,t}}
\prod_{(x,u)}{[(R\circ S)(x,u),(O\circ T)(kx,tu)]} \\
\gamma_{k,k}:I\xrightarrow{\;\sim\;}[1_X(x,x),1_Y(kx,kx)]=[I,I]
\end{gather*}
for all $k\in Y^X, t\in V^U$, $x=x'\in X$. These again can be understood via their transposes under the tensor-hom adjunction, \emph{i.e.}  
composites of projections, inclusions, braidings and evaluations, using the fact that the tensor product preserves sums.
The coherence axioms of \cref{laxfunctor} as for a lax functor of bicategories are satisfied, therefore $H=(H_0,H_1)$
is a lax double functor.

\begin{prop}\label{VMMat1closed}
Under the above assumptions, the functor $H_1$ \cref{H1functor} constitutes a monoidal closed structure
for $(\ca{V}\text{-}\MMat_1,\otimes_1,1_I)$.
\end{prop}
\begin{proof}
We need to show that $-\otimes_1 T\dashv H_1(T,-)\colon\VMMat_1\to\VMMat_1$ for any $\ca{V}$-matrix $\matr{T}{Z}{W,}$
i.e. there is a natural bijection between 2-morphisms
\begin{displaymath}
\xymatrix @C=.6in{X\times Z\ar[r]^-{S\otimes_1 T}|-{\tick}\rtwocell<\omit>{<4>\alpha} \ar[d]_-f & Y\times W\ar[d]^-g \\
U\ar[r]_-P|-{\tick} & V}\quad\textrm{ and }\quad
\xymatrix @C=.6in{X\ar[r]^-{S}|-{\tick}\rtwocell<\omit>{<4>\beta} \ar[d]_-k & Y\ar[d]^-l \\
U^Z\ar[r]_-{H_1(T,P)}|-{\tick} & V^W}
\end{displaymath}
Taking components of the left side 2-morphism \cref{VMat2morphism} and using the monoidal closed structure of $\ca{V}$,
we deduce that any $\alpha_{(x,z),(y,x)}\colon S(x,y)\otimes T(z,w)\to P(f(x,z),g(y,w))$ bijectively corresponds,
for all $x\in X,y\in Y,z\in Z,w\in W$, to some $\ca{V}$-morphism $S(x,y)\to[T(z,w),P(f(x,z),g(y,w))]$. Since $\ca{V}$ has products
and $\Set$ is closed, this uniquely corresponds to some
$$\beta_{x,y}\colon S(x,y)\to\prod_{z,w}[T(z,w),P(f_x(z),g_y(w))]$$
for the transpose functions $f_x\in U^X$, $g_y\in V^W$ and the proof is complete.
\end{proof}

Hence we built a lax double functor $H$ which satisfies the conditions of \cref{loccloseddoublecat}; the desired result follows.

\begin{cor}\label{VMatlocallymoidalclosed}
If $\ca{V}$ is a braided monoidal closed category with products and coproducts, 
the monoidal double category $\VMMat$ is locally closed monoidal.
\end{cor}

The horizontal bicategory $\ca{H}(\ca{V}\text{-}\MMat)$ of the double category of $\ca{V}$-matrices is the well-known bicategory $\ca{V}$-$\Mat$.
In more detail, sets and $\ca{V}$-matrices are the 0- and 1-cells, and 2-cells between $\ca{V}$-matrices $S$ and $S'$ are
globular 2-morphisms, i.e. natural transformations
\begin{displaymath}
\xymatrix{X\ar@/^2ex/[rr]|-{\object@{|}}^-S \ar@/_2ex/[rr]|-{\object@{|}}_-{S'}
\rrtwocell<\omit>{\sigma} && Y}:=
\xymatrix{X\times Y\rrtwocell^{S}_{S'}{\;\sigma} && \ca{V}}
\end{displaymath}
given by families $\sigma_{x,y}:S(x,y)\to S'(x,y)$ of arrows in $\ca{V}$.
The horizontal composition $\circ:\ca{V}\textrm{-}\Mat(Y,Z)\times\ca{V}\textrm{-}\Mat(X,Y)
\to\ca{V}\textrm{-}\Mat(X,Z)$ is given by matrix multiplication \cref{horizontalcompositionVmatrices}
on objects, and by the special case of \cref{horizontalcompositionVmatricearrows} mapping
$(\sigma,\tau)$ to $\{(\tau*\sigma)_{x,z}\}=\{\sum{\sigma_{x,y}\otimes\tau_{y,z}}\}$ on morphisms.

Many useful properties of the companion and conjoint \cref{f*} can be deduced in the bicategorical context from \cref{compconjprops}.
For example, for any $f\colon X\to Y$ we have an adjunction $f_*\dashv f^*$ in the bicategory $\ca{V}$-$\B{Mat}$, with unit and counit
\begin{equation}\label{fetafepsilon}
\xymatrix @C=.6in
{X \ar @/^2ex/[r]|-{\object@{|}}^-{1_X}
\ar@/_2ex/[r]|-{\object@{|}}_-{f^*\circ f_*}
\rtwocell<\omit>{\;\feta} & X}
\qquad\mathrm{and}\qquad
\xymatrix @=.6in
{Y \ar @/^2ex/[r]|-{\object@{|}}^-{f_*\circ f^*}
\ar@/_2ex/[r]|-{\object@{|}}_-{1_Y}
\rtwocell<\omit>{\;\fepsilon} & Y}
\end{equation}
with components arrows in $\ca{V}$
\begin{gather*}
\fepsilon_{y,y'}:(f_*\circ f^*)(y,y')\to 1_Y(y,y')\equiv
\begin{cases}
\sum\limits_{x\in f^{-1}(y)}{I\otimes I}\xrightarrow{r_I}I, & \text{if }y=y' \\
\phantom{\sum\limits_{x\in f^{-1}(y)}}0\xrightarrow{!}0, & \text{if }y\neq y'
\end{cases} \\
\feta_{x,x'}:1_X(x,x')\to(f^*\circ f_*)(x',x)\equiv
\begin{cases}
I\xrightarrow{(r_I)^{-1}}I\otimes I, & \text{if }x'=x \\
0\xrightarrow{!}{\begin{cases} I\otimes I, & fx=fx'\\
0, & \text{else}
\end{cases}} & \text{if }x'\neq x
\end{cases}
\end{gather*}
Notice that $\feta$ and $\fepsilon$ are isomorphisms if and only if $f$ is a bijection.

For explicit calculations in the context of $\ca{V}$-matrices, it will be useful 
to compute that for any $F\colon X\to Y,$ $\SelectTips{eu}{10}\xymatrix @C=.2in{S:Y\ar[r]|-{\object@{|}} & Z}$and$\SelectTips{eu}{10}
\xymatrix @C=.2in{T:Z\ar[r]|-{\object@{|}} & Y,}$
\begin{equation}\label{matrixmachine}
(S\circ f_*)(x,z)=\sum_{y\in Y}{f_*(x,y)\otimes S(y,z)}=I\otimes S(fx,z)\stackrel{r}{\cong}S(fx,z)
\end{equation}
\begin{displaymath}
(f^*\circ T)(z,x)=\sum_{y\in Y}{T(z,y)\otimes f^*(y,x)}=T(z,fx)\otimes I\stackrel{l}{\cong}T(z,fx)
\end{displaymath}
are the families in $\ca{V}$ that define the composite matrices $S\circ f_*$ and $f^*\circ T$.
Using such machinery, we re-obtain the following results for composites of companions and conjoints, \cref{compositecompconj}.

\begin{lem}\label{isosofstars}
Let $f:X\to Y$ and $g:Y\to Z$ be functions. There exist isomorphisms
\begin{gather*}
\zeta^{g,f}:g_*\circ f_*\cong(gf)_*:
\SelectTips{eu}{10}\xymatrix
{X\ar[r]|-{\object@{|}} & Z} \\
\xi^{g,f}:f^*\circ g^*\cong(gf)^*:
\SelectTips{eu}{10}\xymatrix
{Z\ar[r]|-{\object@{|}} & X}
\end{gather*}
which are families of invertible arrows
\begin{equation}\label{zeta}
\zeta^{g,f}_{x,z}=\xi^{g,f}_{z,x}:\begin{cases}
I\otimes I\xrightarrow{r_I=l_I}I, &\textrm{if }g(f(x))=z \\
\quad\quad 0\xrightarrow{\;!\;}0, & \textrm{otherwise}
\end{cases}
\end{equation}
\end{lem}

Under the assumptions of \cref{VMatmonoidaldoublefibrant}, $\VMMat$ is a fibrant monoidal double category
therefore \cref{monoidalhorizontalbicategory} applies.

\begin{prop}\label{bicatVMatmonoidal}
If $\ca{V}$ is a braided monoidal category with coproducts such that $\otimes$ preserves them in both enries,
the bicategory $\ca{V}$-$\Mat$ is a monoidal bicategory; if $\ca{V}$ is symmetric then so is $\VMat$.
\end{prop}

The monoidal unit is the unit  $\ca{V}$-matrix $\ca{I}$ and the induced tensor product pseudofunctor
$\otimes:\ca{V}\text{-}\Mat\times\ca{V}\text{-}\Mat\to\ca{V}\text{-}\Mat$
maps two sets $X,Y$ to their cartesian product
$X\times Y$, and the functor
\begin{displaymath}
\otimes_{(X,Y),(Z,W)}:\ca{V}\text{-}\Mat(X,Z)\times
\ca{V}\text{-}\Mat(Y,W)\to\ca{V}\text{-}\Mat(X\times Y, Z\times W),
\end{displaymath}
is defined as in \cref{VMMat1monoidal} for globular 2-morphisms.

When $\ca{V}$ is moreover monoidal closed with products, its locally closed monoidal structure $H=(H_0,H_1)$ \cref{HforVMMat}
induces a lax functor of bicategories
\begin{displaymath}
\Hom:(\ca{V}\textrm{-}\Mat)^{\textrm{co}}\times\ca{V}\textrm{-}\Mat\longrightarrow\ca{V}\textrm{-}\Mat
\end{displaymath}
where $\ca{V}$-$\Mat^\textrm{co}$ is the bicategory of $\ca{V}$-matrices with reversed 2-cells,
since $\ca{H}(\caa{D}^\op)=\left(\ca{H}(\caa{D})\right)^\textrm{co}$. On objects is given by exponentiation, 
and the functor on hom-categories, for all $(X,Y),(Z,W)$, is
\begin{equation}\label{Hom_}
\Hom_{(X,Y),(Z,W)}\colon\ca{V}\textrm{-}\Mat(X,Z)^\op\times\ca{V}\textrm{-}\Mat(Y,W)\to\ca{V}\textrm{-}\Mat(Y^X,W^Z)
\end{equation}
is given by $\Hom(S,T)=H_1(S,T)$ as in \cref{H1onobjects} on objects and 
$\Hom(\sigma,\tau)=H_1(\sigma,\tau)$ as in \cref{H1onarrows} on globular 2-morphisms, i.e. 2-cells.

The hom-categories of the bicategory $\VMat$ are the functor categories $\ca{V}$-$\B{Mat}(X,Y)=\ca{V}^{X\times Y}$.
The endo-hom-categories for a fixed set $X$ will play an important role; the following proposition exhibits
some useful properties.

\begin{prop}\label{propVMat}
Let $\ca{V}$ be a monoidal category with all colimits such that $\otimes$ preserves them on both entries. For any 0-cell $X$,
the hom-category $\ca{V}$-$\B{Mat}(X,X)=[X\times X,\ca{V}]$ is
\begin{enumerate}[(i)]
\item cocomplete and has all limits that exist in $\ca{V}$;
\item a monoidal category, and $\otimes=\circ$ preserves any colimit on both entries;
\item locally presentable when $\ca{V}$ is; 
\item monoidal closed when $\ca{V}$ is monoidal closed with products.
\end{enumerate}
\end{prop} 
\begin{proof}\hfill

$(i)$ They are formed pointwise from those in $\ca{V}$.

$(ii)$ $(\ca{V}\text{-}\Mat(X,X),\circ,1_X)$ is monoidal for any bicategory, \cref{monadsaremonoids}.

If $(G_j\to G\,|\,j\in\ca{J})$ is a colimiting cocone of shape $\ca{J}$ in $\ca{V}$-$\B{Mat}(X,X)$, for any $x,y\in X$
the arrows $G_j(x,y)\to G(x,y)$ form colimiting cocones in $\ca{V}$. If we apply $S\circ -$,
we obtain a collection of 2-cells $(S\circ G_j\to S\circ G\,|\,j\in\ca{J})$ in $\ca{V}$-$\Mat$.
For this to be a colimit, for any $x,z\in X$ the arrows
\begin{displaymath}
\sum_{y\in X}{G_j(x,y)\otimes S(y,z)}\longrightarrow
\sum_{y\in X}{{\mathrm{colim}_j}G_j(x,y)\otimes S(y,z)}
\end{displaymath}
must be colimiting in $\ca{V}$, which is the case since $(-\otimes A)$ preserves colimits:
\begin{align*}
\sum_{y\in X}{({\mathrm{colim}_j}G_j(x,y))\otimes S(y,z)}
&\cong\sum_{y\in X}{{\mathrm{colim}_j}(G_j(x,y)\otimes S(y,z))} \\
&\cong{\mathrm{colim}_j}(\sum_{y\in X}{G_j(x,y)\otimes S(y,z)}).
\end{align*}

$(iii)$ This follows from \cite[1.54]{LocallyPresentable}: any functor category $[\ca{A},\ca{V}]$ over a presentable category
$\ca{V}$ is also presentable.

$(iv)$ This is obtained by a restriction of \cref{H1onobjects} on globular 2-morphisms. It is not hard to establish a bijective
correspondence
\begin{displaymath}
\xymatrix @R=.02in{\qquad\quad S\circ T \ar[rr] && R\phantom{ABCDE} 
&\mathrm{in}\;\ca{V}\textrm{-}\B{Mat}(X,X)\\ 
\ar@{-}[rr] &&& \\  
\qquad\qquad S \ar[rr] && F(T,R)\phantom{ABC} 
& \mathrm{in}\;\ca{V}\textrm{-}\B{Mat}(X,X)} 
\end{displaymath}
for $G(T,R)(x,y):=\prod_{z\in X}{[T(y,z),R(x,z)]}$.
\end{proof}


\begin{rmk}
The endo-hom-categories of $\VMat$ have in fact a \emph{duoidal} structure, since they are equipped with a second, pointwise monoidal product
\begin{displaymath}
(S\bullet T)(x,y):=S(x,y)\otimes T(x,y),\quad J(x,y)=I,
\end{displaymath}
This point of view is discussed in \cite[\S 7]{Hopfcats}, without mentioning $\ca{V}$-matrices.
This fact is crucial for expressing the so-called \emph{semi-Hopf categories} (categories enriched in comonoids)
as bimonoids in $(\VMat(X,X),\circ,\bullet,1_X,J)$, and subsequently \emph{Hopf categories} as Hopf monoids
in that duoidal category. This approach is relevant to work in progress \cite{HopfcatsasHopfmonads}
regarding the expression of Hopf categories as \emph{Hopf monads}
in a double categorical context. Work in similar direction, establishing
an abstract context for various Hopf-structure generalized notions, can be found in \cite{BohmLack,Gabipolyads};
in fact, the above duoidal structure can be seen as a consequence of all sets being \emph{opmap monoidales} in the monoidal bicategory $\VMat$.
\end{rmk}


Due to the first three part of the above proposition, we obtain the following corollary to \cref{moncomonadm}.
\begin{cor}\label{cofreecomonVMat}
If $\ca{V}$ is a locally presentable monoidal category where $\otimes$ preserves colimits in both entries, the forgetful functors
\begin{gather*}
S:\Mon(\ca{V}\textrm{-}\Mat(X,X))\to\ca{V}\textrm{-}\Mat(X,X) \\
U:\Comon(\ca{V}\textrm{-}\Mat(X,X))\to\ca{V}\textrm{-}\Mat(X,X)
\end{gather*}
are monadic and comonadic respectively, and all categories are locally presentable.
\end{cor}
Notice that $(\ca{V}\text{-}\Mat(X,X),\circ,1_X)$ is non-braided monoidal, therefore
the categories of monoids and comonoids cannot inherit a monoidal structure.

\subsection{\texorpdfstring{$\ca{V}$}{V}-graphs and \texorpdfstring{$\ca{V}$}{V}-categories}\label{Vgraphs}

In this section, we will describe $\ca{V}$-graphs and $\ca{V}$-categories within the context of $\ca{V}$-matrices.
This allows us, by dualizing certain arguments, to later construct the category of $\ca{V}$-cocategories in a natural way. This is motivated
by realizing enriched categories and cocategories as `many-object' generalizations of monoids and comonoids in a monoidal category:
a one-object $\ca{V}$-category is an object in $\Mon(\ca{V})$.
For enriched graphs or categories, usually there are no required assumptions on the monoidal base $\ca{V}$
as it is clear from their definition. In this matrices context though, we ask that $\ca{V}$ has coproducts
preserved by the tensor product.

A (small) $\ca{V}$-\emph{graph} $\ca{G}$ consists of a set of objects $\ob\ca{G}$, and for every pair of objects $x,y\in\ob\ca{G}$
an object $\ca{G}(x,y)\in\ca{V}$. If $\ca{G}$ and $\ca{H}$ are $\ca{V}$-graphs, a $\ca{V}$-\emph{graph morphism}
$F:\ca{G}\to\ca{H}$ consists of a function $f:\ob\ca{G}\to\ob\ca{H}$ between their sets of objects, together with
arrows $F_{x,y}:\ca{G}(x,y)\to\ca{H}(fx,fy)$ in $\ca{V}$, for each pair of objects $x,y$ in $\ca{G}$.
These data, with appropriate compositions and identities, form a category $\ca{V}$-$\B{Grph}$.
There is an evident forgetful functor $Q\colon\ca{V}\textrm{-}\B{Grph}\to\Set$ which maps a graph
to its set of objects and a graph morphism to its underlying function.

If $\VMMat$ is the double category of $\ca{V}$-matrices, it follows that the category of graphs is precisely that of its endomorphisms
as described in \cref{moncomondouble}, i.e. $\ca{V}\text{-}\MMat_1^\bullet=\ca{V}\text{-}\B{Grph}$.
Indeed, objects are endo-$\ca{V}$-matrices$\SelectTips{eu}{10}\xymatrix@C=.2in
{G:X\ar[r]|-{\object@{|}} & X}$given by families of objects $\{G(x,x')\}_X$ in $\ca{V}$, and morphisms between them are $\alpha_f:G_X\to H_Y$
as in \cref{endo2morphism}, given by arrows $\alpha_{x,x'}:G(x,x')\to H(fx,fx')$ in $\ca{V}$ by \cref{VMat2morphism}.

\begin{rmk}
This viewpoint is very similar to that of \cite[Remark 2.5]{Monadsindoublecats}, where it is observed that the category $\B{Grph}_\ca{E}$ of graphs
and graph morphisms internal to a finitely complete $\ca{E}$ is identified with the category of endomorphisms
and vertical endomorphism maps in the double category $\caa{S}\B{pan}_\ca{E}$, i.e. in our notation
$\caa{S}\B{pan}_\ca{E}^\bullet=\B{Grph}_\ca{E}$.
\end{rmk}

In fact, $\ca{V}$-graph morphisms can equivalently be seen as functions $f:X\to Y$ between the sets of objects, equipped with a 2-cell
\begin{displaymath}
\xymatrix @C=.6in
{X\ar @/^2ex/[r]|-{\object@{|}}^-{G}
\ar@/_2ex/[r]|-{\object@{|}}_-{f^*\circ H\circ f_*}
\rtwocell<\omit>{\;\phi} & X}
\end{displaymath}
in $\ca{V}$-$\Mat$, where $f_*$ and $f^*$ are as in \cref{f*}. This is clear by the following
corollary to \cref{D_1^.bifibred}, since $\VMMat$ is a fibrant double category. 

\begin{prop}\label{VGrphbifibr}
The category $\ca{V}$-$\B{Grph}$ is a bifibration over $\B{Set}$.
\end{prop}

The pseudofunctors giving rise to the fibred and opfibred structure are precisely given by \cref{pseudofunctorsdoublebifibration}
in this case,
\begin{equation}\label{Grphpseudofunctors}
 \ps{M}:\xymatrix @R=.02in{\Set^\op\ar[r] & \B{Cat}, \\
X\ar @{|.>}[r]\ar[dd]_-f & \VMat(X,X) \\ \hole \\
Y\ar @{|.>}[r] & \VMat(Y,Y)\ar[uu]_-{f^*\circ\text{-}\circ f_*}}
\qquad \ps{F}:\xymatrix @R=.02in
{\Set\ar[r] & \B{Cat} \\
X\ar @{|.>}[r] \ar[dd]_-f & \VMat(X,X)\ar[dd]^-{f_*\circ\text{-}\circ f^*} \\\hole \\
Y\ar @{|.>}[r] & \VMat(Y,Y)}
\end{equation}
and the Grothendieck categories give the following isomorphic characterization of the category of $\ca{V}$-graphs.

\begin{lem}\label{charactVGrph}
The category $\ca{V}$-$\B{Grph}$ has objects $(G,X)\in\ca{V}\textrm{-}\B{Mat}(X,X)\times\B{Set}$
and arrows $(\phi,f):(G,X)\to(H,Y)$ or equivalently $(\psi,f)$ given by
\begin{displaymath}
\begin{cases}
\phi:G\to f^*Hf_* &\textrm{in }\ca{V}\textrm{-}\Mat(X,X)\\
f:X\to Y & \textrm{in }\B{Set}
\end{cases} \;\textrm{or}\;
\begin{cases}
\psi:f_*Gf^*\to H &\textrm{in }\ca{V}\textrm{-}\Mat(Y,Y)\\
f:X\to Y & \textrm{in }\B{Set}.
\end{cases}
\end{displaymath}
\end{lem}
To see how this works, using \cref{matrixmachine} we can explicitly compute the composite
\begin{equation}\label{machine1}
(f^*Hf_*)_{x,x'} = 
I\otimes(Hf_*)_{x,fx'}
= I\otimes H_{fx,fx'}\otimes I
\end{equation}
hence $\phi$ has components $\phi_{x,x'}:G_{x,x'}\to I\otimes H_{fx,fx'}\otimes I\cong H_{fx,fx'}$.
Similarly,
\begin{equation}\label{machine2}
(f_*Gf^*)_{y,y'}=\sum_{fx=y,fx'=y'}{I\otimes G_{x,x'}\otimes I}
\cong\sum_{fx=y,fx'=y'}{G_{x,x'}}
\end{equation}
so $\psi_{y,y'}:\sum_{\scriptscriptstyle{\stackrel{fx=y}{fx'=y'}}}{I\otimes G_{x,x'}\otimes I}\to H_{y,y'}$
which, for fixed $x\in f^{-1}(y)$ and $x'\in f^{-1}(y')$ corresponds uniquely to $\phi_{x,x'}$.

Notice that the above lemma is completely in terms of the bicategory $\ca{V}$-$\Mat$; moving between this and
the double $\VMMat_1^\bullet$ perspective will be efficient in the proofs that follow.

When $\ca{V}$ is braided, $\VMMat$ is monoidal double by \cref{VMatmonoidaldoublefibrant} thus
its category of endomorphisms is also monoidal by \cref{DendoMonDComonDmonoidal}: the tensor product is given
like in \cref{monoidalVMat} for endo-1-cells and the monoidal unit is the unit $\ca{V}$-matrix$\matr{I}{\{*\}}{\{*\}}$.
Braiding or symmetry is also inherited from $\ca{V}$.

Recall from \cref{VMatlocallymoidalclosed} that when $\ca{V}$ is furthermore closed with products,
the double category of $\ca{V}$-matrices is locally closed monoidal. Hence by \cref{Dendoclosed}, we deduce
that $\VMMat^\bullet_1=\ca{V}\textrm{-}\B{Grph}$ is also monoidal closed.

\begin{prop}\label{VGrphclosed}
Suppose $\ca{V}$ is a braided monoidal closed category with products and coproducts. The restriction of
\cref{H1functor} on the endomorphism category
\begin{displaymath}
H_1^\bullet\colon\ca{V}\textrm{-}\B{Grph}^\mathrm{op}\times
\ca{V}\textrm{-}\B{Grph}\to\ca{V}\textrm{-}\B{Grph}
\end{displaymath}
mapping $(G_X$, $H_Y)$ to the graph $H_1^\bullet(G,H)_{Y^X}$ given by
$H_1^\bullet(G,H)(s,k):=\prod_{x,x'}[G(x,x'),H(sx,kx')]$
for $s,k\in Y^X$ is the internal hom of $\ca{V}$-$\B{Grph}$. 
\end{prop}

Following \cite{KellyLack,Wolff} or the more general case of bicategory enrichment \cite{VarThrEnr}, we gather some of
the main categorical properties of $\ca{V}$-$\B{Grph}$.

\begin{prop}\label{Vgraphprops}\hfill
\begin{enumerate}
\item $\ca{V}$-$\B{Grph}$ is complete when $\ca{V}$ is;
\item $\ca{V}$-$\B{Grph}$ is cocomplete when $\ca{V}$ is;
\item \cite[4.4]{KellyLack} $\ca{V}$-$\B{Grph}$ is locally presentable when $\ca{V}$ is.
\end{enumerate}
\end{prop}
\begin{proof}
(1) The constructions of limits of enriched graphs are built up from limits in $\Set$ and $\ca{V}$ in a straightforward way.

(2) Suppose $F$ is a diagram of shape $\ca{J}$ in $\ca{V}$-$\B{Grph}$
\begin{equation}\label{diagraminVgraph}
 F:\xymatrix @R=.01in @C=.5in
{\ca{J}\ar[r] & \ca{V}\textrm{-}\B{Grph} \\
j\ar@{.>}[r]
\ar[dd]_-\theta & 
(G_j,X_j)\ar[dd]^-{(\psi_\theta,f_\theta)} \\
\hole \\
k\ar@{.>}[r] & (G_k,X_k).}
\end{equation}
By \cref{charactVGrph}, $f_\theta$ is a function between the sets and
$(f_\theta)_*G_j(f_\theta)^*\stackrel{\psi_\theta}{\Rightarrow}G_k$ is a 2-cell in $\ca{V}$-$\Mat$.
The composite $\ca{J}\to\ca{V}\text{-}\B{Grph}\to\B{Set}$ has a colimiting cocone $(\tau_j:X_j\to X\,|\,j\in\ca{J})$ in $\Set$;
since $\tau_j=f_\theta\tau_k$ for any $f_\theta:X_j\to X_k$, we have isomorphisms of $\ca{V}$-matrices 
\begin{displaymath}
 \xymatrix
{X_j\ar[rr]|-{\object@{|}}^-{(\tau_j)_*}
_-{\stackrel{\stackrel{\phantom{A}}{\zeta}}{\cong}}
\ar[dr]|-{\object@{|}}_-{(f_\theta)_*} && X, \\
& X_k\ar[ur]|-{\object@{|}}_-{(\tau_k)_*} &}\qquad
\xymatrix
{X\ar[rr]|-{\object@{|}}^-{(\tau_j)^*}
_-{\stackrel{\stackrel{\phantom{A}}{\xi}}{\cong}}
\ar[dr]|-{\object@{|}}_-{(\tau_k)^*} && X_j \\
& X_k\ar[ur]|-{\object@{|}}_-{(f_\theta)^*} &}
\end{displaymath}
where $\zeta$ and $\xi$ are defined as in \cref{zeta}. Now consider the functor 
\begin{equation}\label{defKdiagram}
K:\xymatrix @C=.6in @R=.02in
{\ca{J}\ar[r]
& \ca{V}\textrm{-}
\B{Mat}(X,X)\qquad\qquad\qquad\qquad\quad \\
j\ar @{|.>}[r] \ar[dd]_-{\theta}& 
(\tau_j)_*G_j(\tau_j)^*
{\scriptstyle{\cong}}(\tau_k)_*(f_\theta)_*G_j
(f_\theta)^*(\tau_k)^*
\ar@<-14ex>
[dd]^-{(\tau_k)_*\psi_\theta(\tau_k)^*}\\
\hole \\
k\ar@{.>}[r] & 
(\tau_k)_*G_k(\tau_k)^*
\phantom{\;\cong(\tau_k)_*(f_\theta)_*G_j
(\tau_k)^*(f_\theta)^*}}
\end{equation}
which explicitly maps an arrow $\theta:j\to k$ in $\ca{J}$ to the composite 2-cell
\begin{equation}\label{Konarrows}
\xymatrix @R=.4in @C=.8in
{X\ar@/_3.5ex/[dr]|-{\object@{|}}_-{(\tau_k)^*}
\drtwocell<\omit>{'\stackrel{\xi}{\cong}}
\ar[r]|-{\object@{|}}^-{(\tau_j)^*} &
 X_j\ar[r]|-{\object@{|}}
^-{G_j} & X_j \ar[d]|-{\object@{|}}
^-{(f_\theta)_*}\ar[r]|-{\object@{|}}^-
{(\tau_j)_*} & X.
\dltwocell<\omit>{'\stackrel{\zeta}{\cong}}\\
 & X_k\ar[u]|-{\object@{|}}
^-{(f_\theta)^*}\ar[r]|-{\object@{|}}_-{G_k}
\rtwocell<\omit>{<-4>\;\psi_\theta} & 
X_k\ar@/_3.5ex/[ur]|-{\object@{|}}_-{(\tau_k)_*} &}
\end{equation}
The colimit of $K$ is formed pointwise in $[X\times X,\ca{V}]$, so there is a colimiting cocone 
$(\lambda_j:(\tau_j)_*G_j(\tau_j)^*\to G\,|\,j\in\ca{J})$. These data allow us to form a new cocone 
\begin{displaymath}
\big((G_j,X_j)\xrightarrow{(\lambda_j,\tau_j)}
(G,X)\,|\,j\in\ca{J}\big)
\end{displaymath}
for the initial diagram $F$ in $\ca{V}$-$\B{Grph}$, since the pairs $(\lambda_j,\tau_j)$ commute accordingly with the $(\psi_\theta,f_\theta)$'s; 
this cocone can be checked to be colimiting, since $\tau_j$ and $\lambda_j$ are.

(3) Briefly, if $\ca{V}$ is a locally $\lambda$-presentable category and the set $\ps{G}$ of objects constitutes
a strong generator of $\ca{V}$, it can be shown that the set
\begin{displaymath}
 \{(\bar{G},2)\;/\;G\in\ps{G}\;\textrm{or}\;G=0\}
\end{displaymath}
constitutes a strong generator of $\ca{V}$-$\B{Grph}$, where the graph $(\bar{G},2)$ has as set of objects
$2=\{0,1\}$ and is given by $\{\bar{G}(0,0)=G,\bar{G}(0,1)=\bar{G}(1,0)=\bar{G}(1,1)=0\}$ in $\ca{V}$. Also, this set is $\lambda$-presentable
in that the functors $\ca{V}\textrm{-}\B{Grph}((\bar{G},2),-):\ca{V}\textrm{-}\B{Grph}\to\B{Set}$ preserve $\lambda$-filtered colimits.
\end{proof}

Passing on to $\VCat$, again following the more general \cite{VarThrEnr}, a $\ca{V}$-category is defined to be a monad in
the bicategory $\ca{V}$-$\B{Mat}$. Unravelling \cref{monadbicat}, it consists 
of a set $X$ together with an endoarrow$\SelectTips{eu}{10}\xymatrix @C=.2in{A:X\ar[r]|-{\object@{|}} & X}$ (i.e. a $\ca{V}$-graph $A_X$)
equipped with two 2-cells, the multiplication and the unit
\begin{displaymath}
\xymatrix @R=.1in @C=.4in
{& X \ar[dr]|-{\object@{|}}
^-{A} &\\
X\ar[ru]|-{\object@{|}}^-A
\ar @/_/[rr]|-{\object@{|}}_-A
\rrtwocell<\omit>{<-1.3>\;M} && X}
\quad
\xymatrix @C=.3in @R=.1in
{\hole \\ \textrm{and} }
\quad
\xymatrix @C=.3in @R=.1in
{\hole \\
X \rrtwocell<\omit>{\eta}
\ar @/^2.2ex/ [rr]|-{\object@{|}}^-{1_X}
\ar @/_2.2ex/ [rr]|-{\object@{|}}_-A && X}
\end{displaymath}
satisfying the following axioms:
\begin{displaymath}
\xymatrix @R=.2in
{& X\ar[r]|-{\object@{|}}^-A &
X\ar @/^/[dr]|-{\object@{|}}^-A & \\
X\ar @/^/[ur]|-{\object@{|}}^-A 
\ar @/_2ex/[urr]|-{\object@{|}}_-A
\ar @/_3ex/[rrr]|-{\object@{|}}_-A
\urrtwocell<\omit>{<-.3>\;M}
&\urrtwocell<\omit>{<1.3>\;M} && X}
\quad\xymatrix @R=.2in
{\hole\\
=}\quad
\xymatrix @R=.2in
{\drrtwocell<\omit>{<1.3>\;M} & 
X\ar[r]|-{\object@{|}}^-A 
\ar @/_2ex/[drr]|-{\object@{|}}_-A 
\drrtwocell<\omit>{<-.3>\;M} &
X\ar @/^/[dr]|-{\object@{|}}^-A & \\
X\ar @/^/[ur]|-{\object@{|}}^-A 
\ar @/_3ex/[rrr]|-{\object@{|}}_-A
&&& X,}
\end{displaymath}
\begin{displaymath}
\xymatrix @C=.5in @R=.2in
{& X \ar @/^/[dr]|-{\object@{|}}^-A &\\
X\urrtwocell<\omit>{<1.3>\;M}
\urtwocell<\omit>{\eta}
\ar @/^2ex/[ur]|-{\object@{|}}^-{1_X}
\ar @/_2ex/[ur]|-{\object@{|}}_-A 
\ar @/_2ex/[rr]|-{\object@{|}}_-A
&& X}
\xymatrix @C=.5in @R=.2in
{\hole \\
=}
\xymatrix @C=.5in @R=.2in
{\hole\\
X\rtwocell<\omit>{\;1_A}
\ar @/^2.3ex/[r]|-{\object@{|}}^-A
\ar @/_2.3ex/[r]|-{\object@{|}}_A 
& X}
\xymatrix @C=.5in @R=.2in
{\hole \\
=}
\xymatrix @C=.4in @R=.2in
{\drrtwocell<\omit>{<1.3>\;M}
& X \ar @/^2ex/[dr]|-{\object@{|}}^-{1_X} 
\ar @/_2ex/[dr]|-{\object@{|}}_-A 
\drtwocell<\omit>{\eta} &\\
X \ar @/^/[ur]|-{\object@{|}}^-A
\ar @/_2ex/[rr]|-{\object@{|}}_-A
&& X.}
\end{displaymath}
In terms of components, they are given by
\begin{displaymath}
M_{x,z}\colon\sum_{y\in X}{A(x,y)\otimes A(y,z)}\to A(x,z)\quad\mathrm{and}\quad\eta_x\colon I\to A(x,x)
\end{displaymath}
which are the usual composition law and identity elements.
The relations that $M$ and $\eta$ have to satisfy give the usual associativity and unit axioms. By \cref{monadsaremonoids}, a monad
in a bicategory is the same as a monoid in the appropriate endoarrow hom-category, \emph{i.e.} a $\ca{V}$-category
$A$ with set of objects $X$ is a monoid in the monoidal category ($\ca{V}$-$\Mat(X,X)$,$\circ$,$1_X$).

A $\ca{V}$-functor $F:\ca{A}\to\ca{B}$ between two $\ca{V}$-categories $A_X$ and $B_Y$ is as usual defined as a morphism
of graphs $\alpha_f:A_X\to B_Y$ which respects the composition law and the identities.
Naturally, one could ask whether this corresponds to the notion of a monad morphism; as will be clear by what follows, this is not the case,
see \cref{Vfunct=monadopfunct}.

In fact, the category $\ca{V}$-$\B{Cat}$ is fully encompassed as the category of monads $\Mnd(\VMMat)$ of \cref{Monadindoublecat}
for the double category of $\ca{V}$-matrices. Indeed, objects are monads$\SelectTips{eu}{10}\xymatrix@C=.2in
{A:X\ar[r]|-{\object@{|}} & X}$in its horizontal
bicategory, and morphisms are arrows between the underlying graphs
that respect the structure: writing down what diagrams \cref{monadhom} give in components for $\caa{D}=\VMMat$, we end up to the usual axioms
\begin{equation}\label{Vfunctaxioms}
\xymatrix @R=.5in @C=.5in
{A(x,y)\otimes A(y,z)\ar[r]^-{M^A_{x,y,z}}\ar[d]_-{\alpha_{x,y}\otimes\alpha_{y,z}} & A(x,z)\ar[d]^-{\alpha_{x,z}}\\
B(fx,fy)\otimes B(fy,fz)\ar[r]_-{M^B_{fx,fy,fz}} & B(fx,fz),}\qquad
\xymatrix @R=.5in @C=.5in
{I\ar[r]^-{\eta_x}\ar[dr]_-{\eta_{fx}} & A(x,x)\ar[d]^-{\alpha_{xx}}\\
& B(fx,fx)}
\end{equation}

Due to the fibrant structure of $\VMMat$, we obtain the following as a corollary to
\cref{MonComonfibred}.

\begin{prop}\label{VCatfibred}
The category $\ca{V}$-$\B{Cat}$ is a fibration over $\B{Set}$.
\end{prop}

The pseudofunctor that gives rise to this fibration is $\ps{M}$ from \cref{Grphpseudofunctors},
restricted between the categories of monoids of the endo-hom-categories. For this to be well-defined, we note
the following corollary to \cref{pseudofunctorsrestrict} for $\caa{D}=\VMMat$.
\begin{cor}\label{B*monoid}
Let $B_Y$ be a $\ca{V}$-category. For any function $f:X\to Y$,
\begin{displaymath}
\SelectTips{eu}{10}\xymatrix @C=.3in
{X\ar[r]|-{\object@{|}}^-{f_*} & Y\ar[r]|-{\object@{|}}^-B & Y\ar[r]|-{\object@{|}}^-{f^*} & X}
\end{displaymath}
is a monoid in $\ca{V}$-$\B{Mat}(X,X)$, \emph{i.e.} $(f^*Bf_*)_X$ is a $\ca{V}$-category.
\end{cor}

The new composition and unit are given by \cref{compositemonoid} in this case, i.e.
$M'=f^*\left(M\cdot(B\fepsilon B)\right)f_*$ and $\eta'=(f^*\eta f_*)\cdot\feta$ using pasting operations,
where $\feta,\fepsilon$ are the unit and counit of $f_*\dashv f^*$ as in \cref{fetafepsilon}.

Once again, the Grothendieck construction that gives rise to the above fibration provide
the following isomorphic characterization of the category of enriched categories and functors;
compare with \cref{charactVGrph}.

\begin{lem}\label{charactVCat}
The objects of $\ca{V}$-$\B{Cat}$ are $(A,X)\in\Mon(\ca{V}\textrm{-}\B{Mat}(X,X))\times\B{Set}$
and morphisms are $(\phi,f):(A,X)\to(B,Y)$ where
\begin{displaymath}
\begin{cases}
\phi:A\to f^*Bf_* &\textrm{in }\Mon(\ca{V}\textrm{-}\B{Mat}(X,X))\\
f:X\to Y & \textrm{in }\B{Set}.
\end{cases}
\end{displaymath}
\end{lem}

Using the mates correspondence for the appropriate 2-cells as in \cref{matesforcompconj}
so that $\phi$ corresponds to $\bar{\phi}:f_*A\Rightarrow Bf_*$, we can the $\ca{V}$-functors axioms in the form
\begin{align*}
\xymatrix @C=.5in @R=.5in
{X\ar[r]|-{\object@{|}}^-A\ar[d]|-{\object@{|}}_-{f_*}\drtwocell<\omit>{\bar{\phi}}
& X\ar[r]|-{\object@{|}}^-A\ar[d]|-{\object@{|}}^-{f_*}\drtwocell<\omit>{\bar{\phi}}
& X \ar[d]|-{\object@{|}}^-{f_*}\\
Y \ar[r]|-{\object@{|}}_-B\ar @/_6ex/[rr]|-{\object@{|}}_-B\rrtwocell<\omit>{<3>\;M}
& Y \ar[r]|-{\object@{|}}_-B & Y} 
&\xymatrix @R=.2in{\hole \\ = \\ \hole}
\xymatrix @C=.5in @R=.3in
{& X \ar @/^/[dr]|-{\object@{|}}^-{A} &\\
X\ar[d]|-{\object@{|}}_-{f_*}\ar @/^/[ru]|-{\object@{|}}^-A\ar[rr]|-{\object@{|}}_-A\rrtwocell<\omit>{<-3>\;M}
\drrtwocell<\omit>{\bar{\phi}} && X\ar[d]|-{\object@{|}}^-{f_*}\\
Y \ar[rr]|-{\object@{|}}_-B && Y,} \\
\xymatrix @C=1.2in @R=.5in
{X\ar @/^5ex/[r]|-{\object@{|}}^-{1_X} \rtwocell<\omit>{<-2>\eta}\ar[r]|-{\object@{|}}_-{A}\ar[d]|-{\object@{|}}_-{f_*}
& X \ar[d]|-{\object@{|}}^-{f_*} \\
Y\ar[r]|-{\object@{|}}_-{B}\rtwocell<\omit>{<-4>\bar{\phi}} & Y}
&\xymatrix{ = \\ \hole}
\xymatrix @C=1.2in  @R=.5in
{X\ar[r]|-{\object@{|}}^-{1_X}\ar[d]|-{\object@{|}}_-{f_*}^{\phantom{ab}\cong}\ar @{.>}[dr]|-{\object@{|}}^-{f_*}
& X\ar[d]|-{\object@{|}}^-{f_*}_{\cong\phantom{ab}} \\
Y\ar[r]|-{\object@{|}}^-{1_Y}\ar @/_5ex/ [r]|-{\object@{|}}_-B\rtwocell<\omit>{<2>\eta} & Y.}
\end{align*}
The mate's components are $\bar{\phi}_{x,y}:I\otimes A(x,x')\to B(fx,fx')\otimes I$ for $x'\in f^{\text{-1}}y$,
and these diagrams in components agree with \cref{Vfunctaxioms} up to tensoring with $I$'s and composing
with the left and right unit constraints of $\ca{V}$.

\begin{rmk}\label{Vfunct=monadopfunct}
Notice that $(f_*,\bar{\phi})$ constitutes a colax monad functor (\cref{monadfunctor}) between the monads $A_X$ and $B_Y$ in 
the bicategory $\ca{V}$-$\B{Mat}$. Evidently, however, it is not true that any colax monad functor given by the data
\begin{displaymath}
\xymatrix @C=.45in @R=.3in
{X\ar[r]|-{\object@{|}}^-A \ar[d]|-{\object@{|}}_-S \rtwocell<\omit>{<4>\chi} & X\ar[d]|-{\object@{|}}^-S\\
Y\ar[r]|-{\object@{|}}_-B & Y}
\end{displaymath}
is a $\ca{V}$-functor, since not every$\SelectTips{eu}{10}\xymatrix @C=.2in
{S:X\ar[r]|-{\object@{|}} & Y}$is of the form $f_*$ for some function $f:X\to Y$.
This explains why the category $\ca{V}$-$\B{Cat}$ cannot be characterized as $\B{Mnd}(\ca{V}$-$\B{Mat})$ for the bicategory,
even if they have the same objects. Similar issues were discussed in a bigger depth in \cite{GarnerShulman}. 

This provides a distinct advantage when cosidering categories of monads in double categories rather than in the horizontal bicategory;
the notion of a $\ca{V}$-functor properly matches the motion of a double monad map in $\VMMat$.
\end{rmk}

As it is well-known, but also deduced from \cref{DendoMonDComonDmonoidal}, when $\ca{V}$ is a braided monoidal category,
$\ca{V}$-$\B{Cat}$ is monoidal via the pointwise tensor product \cref{monoidalVMat} like graphs.
We now move on to further properties of this category.

Similarly to the free monoid construction on an object in a monoidal $\ca{V}$, there is an endofunctor on $\ca{V}$-$\B{Grph}$ inducing the 
`free $\ca{V}$-category' monad; for the proof below, see also \cite{VarThrEnr,KellyLack}.
We spell it out in detail in our context, in order to later dualize for $\ca{V}$-cocategories.
\begin{prop}\label{freeVcatfunctor}
Let $\ca{V}$ be a monoidal category with coproducts, such that $\otimes$ preserves them on both sides. The forgetful functor 
$\tilde{S}:\ca{V}\textrm{-}\B{Cat}\to\ca{V}\textrm{-}\B{Grph}$
has a left adjoint $\tilde{L}$, which maps a $\ca{V}$-graph$\SelectTips{eu}{10}\xymatrix @C=.2in{G:X\ar[r]|-{\object@{|}} & X}$to
the geometric series
\begin{displaymath}
\SelectTips{eu}{10}\xymatrix @C=.3in
{\sum\limits_{n\in\caa{N}}{G^{\otimes n}}:X\ar[r]
|-{\object@{|}} & X.}
\end{displaymath}
\end{prop}
\begin{proof}
Recall that by \cref{propVMat}, $\ca{V}$-$\B{Mat}(X,X)$ admits the same class of colimits as $\ca{V}$, 
and also $\otimes=\circ$ preserves them. Hence, the forgetful $S$ from its category of monoids has a left adjoint
as in \cref{freemonoidprop}
\begin{displaymath}
L:\xymatrix @R=.02in
{\ca{V}\textrm{-}\Mat(X,X)\ar[r] & 
\Mon(\ca{V}\textrm{-}\Mat(X,X))\\
G\ar@{|->}[r] & \sum_{n\in\caa{N}}{G^n}}
\end{displaymath}
which by \cref{charactVCat} means that this geometric series is a $\ca{V}$-category with set of objects $X$. 
The claim is that the induced functor
\begin{equation}\label{deftildeL}
\tilde{L}:\xymatrix @R=.02in @C=.4in
{\ca{V}\textrm{-}\B{Grph}\ar[r] & \ca{V}\textrm{-}\B{Cat}\\
(G,X)\ar@{|->}[r] & (\sum_nG^n,X)}
\end{equation}
is a left adjoint of $\tilde{S}$, i.e. for any $\ca{V}$-category $B_Y$ and $\phi_f\colon G_X\to\tilde{S}(B_Y)$
a $\ca{V}$-graph morphism, there exists a unique $\ca{V}$-functor $H:(\sum_nG^n)_X\to B_Y$
such that
\begin{equation}\label{univprop}
\xymatrix @R=.2in
{(G,X)\ar[rr]^{\tilde{\eta}} \ar[dr]_-F && 
\tilde{S}(\sum\limits_{n\in\caa{N}}{G^n},X)
\ar @{.>}[dl]^-{\tilde{S}H}\\
& \tilde{S}(B,Y) &}
\end{equation}
commutes, for $\tilde{\eta}:(G,X)\to\tilde{S}\tilde{L}(G,X)$ the identity-on-objects inclusion of the summand $G$ into the series.

By \cref{charactVGrph}, a $\ca{V}$-graph functor $F$ is a pair $(\phi,f)$ with $\phi:G\to f^*Bf_*$ an arrow in
$\ca{V}$-$\B{Mat}(X,X)$, and furthermore \cref{B*monoid} ensures that $f^*Bf_*$ obtains a monoid structure. 
Since $L(G)$ is the free monoid on $G\in\ca{V}$-$\B{Mat}(X,X)$, $\phi$ extends uniquely to a monoid morphism
$\chi:LG\to f^*Bf_*$ such that
\begin{displaymath}
\xymatrix @R=.15in
{G\ar[rr]^{\eta} \ar[dr]_-{\phi} && 
\sum\limits_{n\in\caa{N}}{G^n}\ar @{.>}[dl]^-{S\chi}\\
& f^*Bf_* &}
\end{displaymath}
commutes in $\ca{V}$-$\B{Mat}(X,X)$, where $\eta$ and $S$ are respectively the unit and forgetful functor of the `free monoid'
adjunction $L\dashv S$.
By \cref{charactVCat}, this 2-cell $\chi:\sum_{n}{G^n}\Rightarrow f^*Bf_*$ in $\ca{V}$-$\Mat$ determines
a $\ca{V}$-functor $H=(\chi,f):(LG,X)\to(B,Y)$ satisfying the universal property \cref{univprop}. These data suffice
to define an adjoint functor $\tilde{L}$ \cref{deftildeL}, thus the `free $\ca{V}$-category' adjunction
$\tilde{L}\dashv\tilde{S}:\VCat\to\VGrph$ is established.
\end{proof}

Also, as proved in detail in \cite{Wolff} and later generalized in \cite{VarThrEnr},
$\ca{V}$-$\B{Cat}$ has and the forgetful functor $\tilde{S}$ reflects split coequalizers when $\ca{V}$ is 
cocomplete. By Beck's monadicity theorem, since $\tilde{S}$ also reflects isomorphisms, we have the following well-known result. 

\begin{prop}\label{VCatmonadic}
If $\ca{V}$ is a cocomplete monoidal category such that $\otimes$ preserves colimits on both variables,
$\tilde{S}:\ca{V}$-$\B{Cat}\to\ca{V}$-$\B{Grph}$ is monadic.
\end{prop}

Since $\ca{V}$-$\B{Grph}$ is complete when $\ca{V}$ is, we obtain the following corollary.

\begin{cor}\label{VCatcomplete}
The category $\ca{V}$-$\B{Cat}$ is complete when $\ca{V}$ is.
\end{cor}

Moreover, by \cref{Vgraphprops} $\ca{V}$-$\B{Grph}$ admits all colimits if $\ca{V}$ does;
since any category of Eilenberg-Moore algebras has colimits if it has coequalizers of reflexive pairs and
its base has colimits by a standard result in \cite{Linton}, the following is also true.

\begin{cor}\label{VCatcocomplete}
The category $\ca{V}$-$\B{Cat}$ is cocomplete when $\ca{V}$ is.
\end{cor}

Finally, as shown in \cite{KellyLack} the monad $\tilde{S}\tilde{L}$ is finitary. Thus by \cite[Satz 10.3]{GabrielUlmer}
which states that the category of algebras for a finitary monad over a locally presentable category retains that
structure, we obtain the following.

\begin{thm*}~\cite[4.5]{KellyLack}
If $\ca{V}$ is a monoidal closed category whose underlying ordinary category is locally $\lambda$-presentable,
then $\ca{V}$-$\B{Cat}$ is also $\lambda$-presentable.
\end{thm*}

\subsection{\texorpdfstring{$\ca{V}$}{A}-cocategories}\label{VcatsandVcocats}

We now proceed to the dualization of the concept of a $\ca{V}$-category in the context of $\ca{V}$-matrices,
following \cref{comonadbicat}. Henceforth $\ca{V}$ is again a monoidal category with
coproducts, such that the tensor product $\otimes$ preserves them on both entries.

\begin{defi}\label{cocategory}
A $\ca{V}$-\emph{cocategory}  $C$ is a comonad in the bicategory $\ca{V}$-$\B{Mat}$. It consists of a set $X$ with 
an endoarrow$\SelectTips{eu}{10}\xymatrix @C=.2in{C:X\ar[r]|-{\object@{|}} & X}$(\emph{i.e.} a $\ca{V}$-graph $C_X$)
equipped with two 2-cells, the comultiplication and the counit
\begin{displaymath}
\xymatrix @R=.1in @C=.4in
{X \ar[dr]|-{\object@{|}}_-{C} \ar @/^3ex/[rr]|-{\object@{|}}^-C\rrtwocell<\omit>{<.5>\Delta} && X \\
 & X\ar[ur]|-{\object@{|}}_-C  &}
\quad\textrm{and}\quad
\xymatrix @C=.3in
{X \rrtwocell<\omit>{<0>\epsilon}\ar @/^2.2ex/ [rr]|-{\object@{|}}^-{C}\ar @/_2.2ex/ [rr]|-{\object@{|}}_-{1_X} && X}
\end{displaymath}
satisfying the following axioms:
\begin{gather}\label{Vfunctaxioms2cells}
\xymatrix @R=.2in
{X\ar @/_/[dr]|-{\object@{|}}_-C\ar @/^3ex/[drr]|-{\object@{|}}^-C\ar @/^4ex/[rrr]|-{\object@{|}}^-C\drrtwocell<\omit>{<+.3>\Delta}
&\drrtwocell<\omit>{<-1.3>\Delta} && X \\
& X\ar[r]|-{\object@{|}}_-C & X\ar @/_/[ur]|-{\object@{|}}_-C &} 
\quad = \quad
\xymatrix @R=.2in
{X\ar @/_/[dr]|-{\object@{|}}_-C\ar @/^4ex/[rrr]|-{\object@{|}}^-C&&& X \\
\urrtwocell<\omit>{<-1.3>\Delta} & X\ar[r]|-{\object@{|}}_-C\ar @/^3ex/[urr]|-{\object@{|}}^-C\urrtwocell<\omit>{<+.3>\Delta} &
X,\ar @/_/[ur]|-{\object@{|}}_-C &} \\
\xymatrix @C=.5in @R=.2in
{X\drrtwocell<\omit>{<-2.3>\Delta}\drtwocell<\omit>{<-0.4>\epsilon}\ar @/_2ex/[dr]|-{\object@{|}}_-{1_X}\ar @/^3ex/[dr]|-{\object@{|}}^-C 
\ar @/^4ex/[rr]|-{\object@{|}}^-C && X \\
& X \ar @/_/[ur]|-{\object@{|}}_-C &}
\xymatrix @C=.5in @R=.2in{=\\\hole}
\xymatrix @C=.5in @R=.2in
{X\rtwocell<\omit>{\;1_C}\ar @/_2.3ex/[r]|-{\object@{|}}_-C\ar @/^2.3ex/[r]|-{\object@{|}}^-C 
& X \\ \hole}
\xymatrix @C=.5in @R=.2in{= \\\hole}
\xymatrix @C=.5in @R=.2in
{X \ar @/_/[dr]|-{\object@{|}}_-C\ar @/^4ex/[rr]|-{\object@{|}}^-C && X \\
\urrtwocell<\omit>{<-2.3>\Delta} & X. \ar @/_2ex/[ur]|-{\object@{|}}_-{1_X}\ar @/^3ex/[ur]|-{\object@{|}}^-C 
\urtwocell<\omit>{<-0.4>\epsilon} &}\nonumber
\end{gather}
\end{defi}
In terms of components, the \emph{cocomposition} of a $\ca{V}$-cocategory $ C $ is given by
\begin{displaymath}
\Delta_{x,z}:C(x,z)\to\sum_{y\in X}{C(x,y)\otimes C(y,z)}
\end{displaymath}
for any two objects $x,y\in X$, and the \emph{coidentity elements} are given by
\begin{displaymath}
 \epsilon_{x,y}:C(x,y)\to 1_X(x,y)\equiv
\begin{cases}
C(x,x)\xrightarrow{\epsilon_{x,x}}I,\quad \mathrm{if }\;x=y\\
C(x,y)\xrightarrow{\epsilon_{x,y}}0,\quad \mathrm{if }\;x\neq y
\end{cases}
\end{displaymath}
for all objects $x\in X$. The coassociativity and counity axioms are
\begin{displaymath}
\xymatrix @C=.3in @R=.1in
{& C_{x,w}\ar[dl]_-{\Delta}\ar[dr]^-{\Delta} &\\
\sum\limits_{z}{C_{x,z}\otimes C_{z,w}}\ar[dd]_-{\sum\limits_{z}{\Delta\otimes 1}} &&
\sum\limits_{y}{C_{x,y}\otimes C_{y,w}}\ar[dd]^-{\sum\limits_{y}{1\otimes\Delta}} \\
\hole \\
\sum\limits_{z}(\sum\limits_{y}{C_{x,y}\otimes C_{y,z}})\otimes C_{z,w}\ar[rr]_-{\alpha}^-\sim &&
\sum\limits_{y}{C_{x,y}\otimes(\sum\limits_{z}{C_{y,z}\otimes C_{y,w}})}}
\end{displaymath}
\begin{displaymath}
\xymatrix @C=.9in @R=.4in
{\sum\limits_{z}{C_{x,z}\otimes C_{z,y}}\ar[d]_-{\sum\limits_{z}{\epsilon\otimes 1}} &
C_{x,y}\ar[dr]_-{\rho^{\text{-}1}}\ar[dl]^-{\lambda^{\text{-}1}}\ar[l]_-{\Delta}\ar[r]^-{\Delta} &
\sum\limits_{z}{C_{x,z}\otimes C_{z,y}}\ar[d]^-{\sum\limits_{z}{1\otimes\epsilon}} \\
I\otimes C_{x,y} && C_{x,y}\otimes I}
\end{displaymath}
where $\alpha$ is the associator and $\lambda$, $\rho$ are the unitors of $\ca{V}$-$\Mat$.

As for any comonad in a bicategory, a $\ca{V}$-cocategory $ C $ with $\ob C =X$ is the same as a comonoid
in the monoidal category $(\ca{V}$-$\B{Mat}(X,X),\circ,1_X)$. Thus a one-object $\ca{V}$-cocategory
is the same as a comonoid in $\ca{V}$.

A $\ca{V}$-cofunctor between two $\ca{V}$-cocategories $C_X$, $D_Y$ should be a $\ca{V}$-graph morphism
$F_f\colon C_X\to D_Y$ that respects cocomposition and coidentities. As a result, we obtain the following definition.

\begin{defi}\label{cofunctor}
A $\ca{V}$-\emph{cofunctor} $F_f:C_X\to D_Y$ between two $\ca{V}$-cocategories consists of a function $f:X\to Y$ between their
sets of objects and arrows $F_{x,z}:C(x,z)\to D(fx,fz)$ in $\ca{V}$ for any $x,z\in\ob C$,
satisfying the commutativity of
\begin{equation}\label{cofuncts}
\xymatrix @C=.5in @R=.27in
{C(x,z)\ar[r]^-{\Delta_{x,z}}\ar[dd]_-{F_{x,z}} &
\sum\limits_{y}{C(x,y)\otimes C(y,z)}\ar[d]^-{\sum\limits_{y}{F_{x,y}\otimes F_{y,z}}} \\
& \sum\limits_{fy}{D(fx,fy)\otimes D(fy,fz)}\ar @{>->}[d] \\
D(fx,fz)\ar[r]_-{\Delta_{fx,fz}} & \sum\limits_{w}{D(fx,w)\otimes D(w,fz)}}
\qquad
\xymatrix @R=.43in
{C(x,x)\ar[rd]^-{\epsilon_{x,x}} \ar[dd]_-{F_{x,x}} & \\ & I\\
D(fx,fx)\ar[ur]_-{\epsilon_{fx,fx}} &}
\end{equation}
\end{defi}

Along with compositions and identities of cofunctors that follow those of graphs,
we obtain a category $\VCocat$. As was the case for $\VCat$, this category is fully encapsulated as
the category of comonads in the double category of $\ca{V}$-matrices, i.e. $\Comon(\VMMat)=\VCocat$
as in \cref{Comonadindoublecat}. Indeed, an object is just a comonad in the horizontal bicategory,
and the comonad morphism axioms in components give \cref{cofuncts}.

\begin{rmk}
At \cite[\S 9]{Monoidalbicatshopfalgebroids}, a $\ca{V}$-\emph{opcategory} $\ca{A}$ is defined, as a category
enriched in $\ca{V}^\op$. In particular, the cocomposition and counit arrows are simply
\begin{displaymath}
\Delta\colon\ca{A}(x,z)\to\ca{A}(y,z)\otimes\ca{A}(x,y),\quad \varepsilon\colon\ca{A}(x,x)\to I.
\end{displaymath}
Also, a $\ca{V}$-\emph{opfunctor} $F\colon\ca{A}\to\ca{B}$ is a $\ca{V}^\op$-functor, mapping $x\mapsto Fx$ and
with arrows $\ca{B}(Fx,Fy)\to\ca{A}(x,y)$ in $\ca{V}$ satisfying respective axioms.
The category $\ca{V}$-$\B{opCat}$ is in a sense broader than
$\VCocat^{(\op)}$, since for example for $\ca{V}=\Mod_R$ any cocategory is a special case of an opcategory: due to
\begin{displaymath}
\xymatrix@C=.5in
{A(x,y)\ar[r]^-{\Delta_{x,y}}\ar@{-->}[dr]_-{\Delta_{x,z,y}} & \sum\limits_{z\in X}{A(x,z)\otimes A(z,y)}\ar@{>->}[r]
& \prod\limits_{z\in X}A(x,z)\otimes A(z,y)\ar[dl]^-{\pi_z} \\
& A(x,z)\otimes A(z,y) &}
\end{displaymath}
a $\Mod_R$-cocategory is a $\Mod_R$-opcategory for which the cocomposition $\Delta_{x,z,y}$ vanishes for all
but finitely many objects $z$.
However, for the current development it is crucial that a $\ca{V}$-cocategory can be expressed as a comonoid in the monoidal endo-hom-categories
of $\VMat$. A $\ca{V}$-opcategory on the other hand does not seem to be expressed as a comonoid, since for example 
the tensor product of such a monoidal category would have to commute with the products rather than the sums. 

Moreover, the cocategories point of view seems to be useful in other settings too;
in a series of papers related to $A_\infty$-categories, see \cite{Ainfinitycategories,Ainfinityalgebras,Equalizerscocompletecocategories},
the authors study and employ \emph{cocategories} and \emph{cocategory homorphisms} which are precisely our $\Mod_R$-enriched case,
in order to express $A_\infty$-functor categories as internal hom-objects, building models of a closed structure of the homotopy category
of differential graded categories.
\end{rmk}

Due to the fibrant structure of $\VMMat$, we again the following as a corollary to \cref{MonComonfibred}.

\begin{prop}\label{VCocatopfibred}
The category $\ca{V}$-$\B{Cocat}$ is an opfibration over $\B{Set}$.
\end{prop}

The pseudofunctor giving rise to this opfibration is $\ps{F}$ from \cref{Grphpseudofunctors}, now restricted to the categories
of comonoids of the endo-hom-categories. Again for this to be well-defined, dually to \cref{B*monoid}
we have the following consequence of $f_*\circ\textrm{-}\circ f^*$ being colax monoidal, \cref{pseudofunctorsrestrict}.

\begin{lem}\label{C*comonoid}
Let $C_X$ be a $\ca{V}$-cocategory. If $f:X\to Y$ is a function, then
\begin{displaymath}
\SelectTips{eu}{10}\xymatrix
{Y\ar[r]|-{\object@{|}}^-{f^*} & X\ar[r]|-{\object@{|}}^-C & X\ar[r]|-{\object@{|}}^-{f_*} & Y}
\end{displaymath}
is a comonoid in $\ca{V}$-$\B{Mat}(Y,Y)$, i.e. $(f_*Cf^*)_Y$ is a $\ca{V}$-cocategory.
\end{lem}
The new comultiplication and counit come from \cref{compositecomonoid}; using pasting operations,
we can express them as $\Delta'=f_*\left((C\feta C)\cdot\Delta\right)f^*$, $\epsilon'=\fepsilon\cdot(f_*\epsilon f^*)$.

Once more, the Grothendieck category gives an isomorphic characterization of the category of
$\ca{V}$-cocategories and cofunctors, completely in terms of $\VMat$.

\begin{lem}\label{charactVCocat}
Objects in $\ca{V}$-$\B{Cocat}$ are $(C,X)\in\Comon(\ca{V}\textrm{-}\B{Mat}(X,X))\times\B{Set}$ and morphisms are $(\psi,f):(C,X)\to(D,Y)$ where 
\begin{displaymath}
\begin{cases}
\psi:f_*Cf^*\to D &\textrm{in }\Comon(\ca{V}\textrm{-}\B{Mat}(Y,Y))\\
f:X\to Y & \textrm{in }\B{Set}.
\end{cases}
\end{displaymath}
\end{lem}

Comparing with the two equivalent formulations for $\ca{V}$-graph morphisms of \cref{charactVGrph}, notice that $\ca{V}$-functors 
are expressed as pairs $(\phi,f)$ and $\ca{V}$-cofunctors are expressed as pairs $(\psi,f)$, where the 2-cells
$\phi:G\Rightarrow f^*Hf_*$ and $\psi:f_*Gf^*\Rightarrow H$ are mates under $f_*\dashv f^*$.

Writing explicitly what it means for $\psi$ to be a comonoid morphism, which is clearer in terms of its mate 
$\hat{\phi}:f_*C\Rightarrow Df_*$, we obtain
\begin{align}\label{cofunctaxioms}
\xymatrix @C=.5in @R=.5in
{X \ar[r]|-{\object@{|}}^-C\ar[d]|-{\object@{|}}_-{f_*}\rrtwocell<\omit>{<-3>\Delta}\drtwocell<\omit>{\hat{\phi}}\ar @/^6ex/[rr]|-{\object@{|}}^-C
& X\ar[r]|-{\object@{|}}^-C\ar[d]|-{\object@{|}}^-{f_*}\drtwocell<\omit>{\hat{\phi}} & X \ar[d]|-{\object@{|}}^-{f_*}\\
Y \ar[r]|-{\object@{|}}_-D & Y \ar[r]|-{\object@{|}}_-D & Y} 
&\xymatrix @R=.2in{\hole \\ = \\ \hole}
\xymatrix @C=.5in @R=.3in
{X\ar[d]|-{\object@{|}}_-{f_*}\ar[rr]|-{\object@{|}}^-C\drrtwocell<\omit>{\hat{\phi}} && X \ar[d]|-{\object@{|}}^-{f_*}\\
Y \rrtwocell<\omit>{<4>\Delta} \ar @/_/[dr]|-{\object@{|}}_-D \ar[rr]|-{\object@{|}}_-D && Y\\
& Y \ar @/_/[ur]|-{\object@{|}}_-D &} \\
\xymatrix @C=1in  @R=.4in
{X\rtwocell<\omit>{<-2>\epsilon}\ar @/^5ex/[r]|-{\object@{|}}^-C\ar[r]|-{\object@{|}}_-{1_X}\ar[d]|-{\object@{|}}_-{f_*}^{\phantom{ab}\cong}
\ar @{.>}[dr]|-{\object@{|}}_-{f_*} & X \ar[d]|-{\object@{|}}^-{f_*}_{\cong\phantom{ab}} \\
Y\ar[r]|-{\object@{|}}_-{1_Y} & Y}
&\xymatrix{ = \\ \hole}
\xymatrix @C=1in @R=.4in
{X\ar[r]|-{\object@{|}}^-C \ar[d]|-{\object@{|}}_-{f_*} & X \ar[d]|-{\object@{|}}^-{f_*} \\
Y\ar @/_5ex/[r]|-{\object@{|}}_-{1_Y} \rtwocell<\omit>{<2.5>\epsilon}\ar[r]|-{\object@{|}}^-D\rtwocell<\omit>{<-4>\hat{\phi}} & Y}\nonumber
\end{align}
The components of $\hat{\phi}$ are given by $\sum_{x'\in f^{\text{-}1}y}I\otimes C(x,x')\to D(fx,fx')\otimes I$
and the above relations componentwise give \cref{cofuncts} up to isomorphism.
Dually to \cref{Vfunct=monadopfunct}, cofunctors correspond to specific types of lax comonad functors
in $\VMat$; viewing them as comonad homomorphisms in $\VMMat$ is conceptually simpler and less technical,
but both expressions will be of use.

When $\ca{V}$ is braided monoidal, by \cref{DendoMonDComonDmonoidal} $\ca{V}$-$\B{Cocat}$ obtains a mo\-noi\-dal 
structure, as for $\ca{V}$-graphs and categories: for two $\ca{V}$-cocategories $C_X$ and $D_Y$,
$(C\otimes D)_{X\times Y}$ is given by $(C\otimes D)\left((x,y),(z,w)\right)=C(x,z)\otimes D(y,w)$
in $\ca{V}$. Writing down in components what the dual of \cref{MonFdouble} gives for the pseudo double functor
$\otimes$ on $\VMMat$, we deduce 
that the cocomposition is
\begin{displaymath}
 \xymatrix @C=.7in @R=.2in
{C(x,z)\otimes D(y,w)\ar@{-->}[r] \ar@/_8ex/[ddr]_-{\Delta^C_{x,z}\otimes\Delta^D_{y,w}} &
\sum\limits_{(x',y')}{C(x,x')\otimes D(y,y')\otimes C(x',z)\otimes D(y',w)} \\
& \sum\limits_{(x',y')}{C(x,x')\otimes C(x',z)\otimes D(y,y')\otimes D(y',w)}\ar[u]^-s_-{\cong} \\
& \sum\limits_{x'}{C(x,x')\otimes C(x',z)}\otimes\sum\limits_{y'}{D(y,y')\otimes D(y',w)}
\ar[u]_-{\cong}}
\end{displaymath}
and the coidentity element is $C(x,x)\otimes D(y,y)\xrightarrow{\;\epsilon^C_{x,x}\otimes\epsilon^D_{y,y}\;}I\otimes I\cong I$.
The monoidal $\VCocat$ also inherits the braiding or symmetry from $\ca{V}$.

Dually to \cref{freeVcatfunctor}, we now construct the cofree $\ca{V}$-cocategory functor using the cofree comonoid construction.
As discussed in \cref{Monoidalcats}, the existence of the cofree comonoid usually requires more assumptions on $\ca{V}$
than the free monoid, and the following construction is no exception.

\begin{prop}\label{cofreeVcocatfunctor}
Suppose $\ca{V}$ is a locally presentable monoidal category, such that $\otimes$ preserves colimits in both variables. 
Then, the evident forgetful functor
\begin{displaymath}
\tilde{U}:\ca{V}\textrm{-}\B{Cocat}\longrightarrow\ca{V}\textrm{-}\B{Grph}
\end{displaymath}
has a right adjoint $\tilde{R}$, which maps a $\ca{V}$-graph $G_Y$ to the cofree comonoid $(RG,Y)$ on $G\in\ca{V}$-$\B{Mat}(Y,Y)$.
\end{prop}
\begin{proof}
The forgetful $\tilde{U}$ maps any $\ca{V}$-cocategory $C_X$ to its underlying $\ca{V}$-graph $(UC)_X$, where $U$ is the forgetful from
the category of comonoids of the monoidal ($\ca{V}$-$\B{Mat}(Y,Y),\circ,1_Y)$. By \cref{cofreecomonVMat}, $U$ has a right adjoint
\begin{displaymath}
R:\ca{V}\textrm{-}\Mat(Y,Y)\longrightarrow\Comon(\ca{V}\textrm{-}\Mat(Y,Y)),
\end{displaymath}
the cofree comonoid functor. By \cref{charactVCocat}, $(RG)_Y$ for any$\SelectTips{eu}{10}\xymatrix @C=.2in{G:Y\ar[r]|-{\object@{|}} & Y}$is
a $\ca{V}$-cocategory; we claim that
\begin{displaymath}
\tilde{R}:\xymatrix @R=.02in
{\ca{V}\textrm{-}\B{Grph}\ar[r] & \ca{V}\textrm{-}\B{Cocat}\\
G_Y\ar@{|->}[r] & (RG)_Y}
\end{displaymath}
is a right adjoint of $\tilde{U}$. It suffices to show that for $\varepsilon$ the counit of $U\dashv R$,
the $\ca{V}$-graph arrow $\tilde{\varepsilon}=(\varepsilon,\mathrm{id}_Y):\tilde{U}\tilde{R}(G_Y)\to G_Y$ is universal. This means that for 
any $\ca{V}$-cocategory $C_X$ and any $\ca{V}$-graph morphism $F$ from its underlying $\tilde{U}(C_X)$ to $G_Y$, there exists 
a unique $\ca{V}$-cofunctor $H:C_X\to (RG)_Y$ such that
\begin{equation}\label{thisuniversality}
\xymatrix @R=.35in
{\tilde{U}(RG_Y)\ar[rr]^-{\tilde{\varepsilon}} && G_Y\\ & \tilde{U}(C_X)\ar @{.>}[ul]^-{\tilde{U}H}\ar[ur]_-F &}
\end{equation} 
commutes. Indeed, suppose $F$ is $(\psi,f)$ with $f:X\to Y$ and $\psi:f_*Cf^*\to G$ in $\ca{V}$-$\B{Mat}(Y,Y)$. Since
$f_*Cf^*$ is in $\Comon(\ca{V}$-$\B{Mat}(Y,Y))$ due to \cref{C*comonoid} and $RG$ is the cofree comonoid on $G$, $\psi$ extends
uniquely to a comonoid arrow $\chi:f_*Cf^*\to RG$ such that
\begin{displaymath}
\xymatrix @R=.35in
{RG \ar[rr]^-{\varepsilon} && G\\
& f_*Cf^* \ar[ur]_-{\psi}
\ar @{.>}[ul]^-{U\chi} &}
\end{displaymath}
commutes in $\ca{V}$-$\B{Mat}(Y,Y)$. By \cref{charactVCocat}, this $\chi$ along with the function $f:X\to Y$ determines
a $\ca{V}$-cofunctor $H:(C,X)\to(RG,Y)$ which makes \cref{thisuniversality} commute.
Therefore $\tilde{R}$ extends to a functor establishing the `cofree $\ca{V}$-cocategory'
adjunction $\tilde{U}\dashv\tilde{R}:\ca{V}\text{-}\B{Grph}\to\ca{V}\text{-}\B{Cocat}$.
\end{proof}

At this point, properties of $\ca{V}$-$\B{Cocat}$ cease to be straightforward dualizations of the $\ca{V}$-$\B{Cat}$
ones. The results that follow more or less employ similar techniques as for the one-object case,
that of comonoids in a monoidal category, generalized to this context.
 
The construction of colimits in $\ca{V}$-$\B{Cocat}$ follows from that in $\ca{V}$-$\B{Grph}$ of \cref{Vgraphprops}, with 
an induced extra structure on the colimiting cocone.

\begin{prop}\label{VCocatcocomplete}
If $\ca{V}$ is a locally presentable monoidal category such that $\otimes$ preserves colimits in both terms,
the category $\ca{V}$-$\B{Cocat}$ has all colimits.
\end{prop}

\begin{proof}
Consider a diagram in $\ca{V}$-$\B{Cocat}$ given by
\begin{displaymath}
D:\xymatrix @R=.05in @C=.6in
{\ca{J}\ar[r] & \ca{V}\textrm{-}\B{Cocat} \\
j\ar@{|.>}[r]
\ar[dd]_-\theta & 
(C_j,X_j)\ar[dd]^-{(\psi_\theta,f_\theta)} \\
\hole \\
k\ar@{|.>}[r] & (C_k,X_k)}
\end{displaymath}
for a small category $\ca{J}$. By \cref{charactVCocat}, $f_\theta:X_j\to X_k$ is a function and $\psi_\theta$ is an arrow
$(f_\theta)_*C_j(f_\theta)^*\to C_k$ in $\Comon(\ca{V}$-$\Mat(X_k,X_k))$. First constructing the colimit of the underlying $\ca{V}$-graphs,
we obtain a colimiting cocone
\begin{equation}\label{colimgraph}
 \big((C_j,X_j)\xrightarrow{\;(\lambda_j,\tau_j\;)}
(C,X)\,|\,j\in\ca{J}\big)
\end{equation}
in $\ca{V}$-$\B{Grph}$, where $(\tau_j:X_j\to X\,|\,j\in\ca{J})$ is the colimit of the sets of objects of the $\ca{V}$-cocategories in $\B{Set}$,
and $(\lambda_j:(\tau_j)_*C_j(\tau_j)^*\to C\,|\,j\in\ca{J})$ is the colimiting cocone of the diagram $K$ as in \cref{defKdiagram}
in the cocomplete $\ca{V}$-$\Mat(X,X)$. In fact, $K:\ca{J}\to\ca{V}\textrm{-}\Mat(X,X)$ lands inside $\Comon(\ca{V}$-$\Mat(X,X))$:
\cref{C*comonoid} ensures that $\ca{V}$-matrices of the form $f_*Cf^*$ for any comonoid $C$ inherit a comonoid structure, and the composite
arrows \cref{Konarrows} where the middle 2-cell is now the comonoid arrow $\psi_\theta$ ensure that $K\theta$ are comonoid morphisms.

By \cref{cofreecomonVMat}, the category of comonoids is comonadic over $\ca{V}\textrm{-}\Mat(X,X)$ hence the forgetful functor
creates all colimits and so$\SelectTips{eu}{10}\xymatrix@C=.2in{C:X\ar[r]|-{\object@{|}} & X}$obtains a unique comonoid structure.
Moreover, the legs of the cocone
\begin{displaymath}
\xymatrix
{X_j\ar[r]|-{\object@{|}} ^-{C_j} & X_j \ar[d]|-{\object@{|}} ^-{(\tau_j)_*}\\
X\ar[u]|-{\object@{|}} ^-{(\tau_j)^*}\ar[r]|-{\object@{|}}_-{C} \rtwocell<\omit>{<-4>\lambda_j} & X}
\end{displaymath}
are comonoid arrows, so together with the functions $\tau_j$ they form $\ca{V}$-cofunctors. Therefore the colimit \cref{colimgraph}
lifts in $\ca{V}$-$\B{Cocat}$.
\end{proof}

We will now apply techniques similar to \cref{moncomonadm} regarding the expression of the category $\Comon(\ca{V})$ as an equifier,
to obtain the following.

\begin{prop}\label{VCocatlocpresent}
Suppose that $\ca{V}$ is a locally presentable monoidal category, such that $\otimes$ preserves colimits in both terms.
Then, $\ca{V}$-$\B{Cocat}$ is a locally presentable category.
\end{prop}

\begin{proof}
Using \cref{charactVGrph}, we can define an endofunctor on $\ca{V}$-graphs by
\begin{displaymath}
 F:\xymatrix @R=.03in @C=.5in
{\ca{V}\textrm{-}\B{Grph}\ar[r] & \ca{V}\textrm{-}\B{Grph}\quad \\
(G,X)\ar@{|.>}[r]\ar[dd]_-{(\psi,f)} & (G\circ G,X)\times(1_X,X) \ar[dd]^-{F(\psi,f)} \\
\hole \\ (H,Y)\ar@{|.>}[r] & (H\circ H,Y)\times(1_Y,Y)}
\end{displaymath}
with explicit mapping on arrows, for a 2-cell $\psi:f_*Gf^*\Rightarrow H$,
\begin{equation}\label{Fonarrows}
\xymatrix @C=.4in @R=.5in
{X\rrtwocell<\omit>{<6>\psi}\ar[rr]|-{\object@{|}}^-G && X \ar[d]|-{\object@{|}}_-{f_*}\ar[r]|-{\object@{|}}^-{1_X}
\rtwocell<\omit>{<3>\feta}
& X\rtwocell<\omit>{<6>\psi}\ar[r]|-{\object@{|}}^-G & X\ar[d]|-{\object@{|}}^-{f_*} \\
Y\ar[u]|-{\object@{|}}^-{f^*}\ar[rr]|-{\object@{|}}_-H &&Y\ar@/_/[ur]|-{\object@{|}}_-{f^*} \ar[rr]|-{\object@{|}}_-H && Y}
\xymatrix @R=.1in{\hole \\ \times}
\xymatrix @C=.4in @R=.5in
{X\ar[rr]|-{\object@{|}}^-{1_X}\ar[drr]|-{\object@{|}}^-{f_*} & \rtwocell<\omit>{<4>\cong} & X\ar[d]|-{\object@{|}}^-{f_*} \\
Y\ar[u]|-{\object@{|}}^-{f^*} \ar[rr]|-{\object@{|}}_-{1_Y}\rtwocell<\omit>{<-4>\fepsilon} &&  Y.}
\end{equation} 
The category of functor coalgebras $\Coalg F$ has as objects $\ca{V}$-graphs $(C,X)$
with a morphism $\alpha:C\to C\circ C\times 1_X$, \emph{i.e.} two $\ca{V}$-graph arrows 
\begin{displaymath}
 \alpha_1:(C,X)\to (C\circ C,X)\quad\textrm{and}\quad\alpha_2:(C,X)\to
(1_X,X).
\end{displaymath}
A morphism $(C,\alpha)\to(D,\beta)$
is a $\ca{V}$-graph morphism $(\psi,f):(C,X)\to(D,Y)$
which is compatible with $\alpha$ and $\beta$,
\emph{i.e.} satisfy the equalities
\begin{gather*}
 \xymatrix @R=.45in
{X\ar@/^5ex/[rrrr]|-{\object@{|}}^-C
\rrtwocell<\omit>{<5>\psi}
\ar[rr]|-{\object@{|}}^-C &
\rrtwocell<\omit>{<-3>\;\alpha_1} & X
\rtwocell<\omit>{<3>\feta}
\ar[d]|-{\object@{|}}_-{f_*}
\ar[r]|-{\object@{|}}^-{1_X}
& X\rtwocell<\omit>{<5>\psi}
\ar[r]|-{\object@{|}}^-C & X\ar[d]|-{\object@{|}}^-{f_*} \\
Y\ar[u]|-{\object@{|}}^-{f^*}\ar[rr]|-{\object@{|}}_-D &&
Y\ar@/_/[ur]|-{\object@{|}}_-{f^*} \ar[rr]|-{\object@{|}}_-D && Y}
\xymatrix @R=.1in{\hole \\ = \\ \hole}
\xymatrix @R=.15in
{X \ar[rr]|-{\object@{|}}^-C
\rrtwocell<\omit>{<4>\psi}&& X
\ar[dd]|-{\object@{|}}^-{f_*}\\
& \\
Y \ar[uu]|-{\object@{|}}^-{f^*}
\rrtwocell<\omit>{<2.5>\beta_1}
\ar @/_/[dr]|-{\object@{|}}_-D
\ar[rr]|-{\object@{|}}^-D && Y\\
& Y \ar @/_/[ur]|-{\object@{|}}_-D &} \\
\xymatrix @C=.5in @R=.4in
{X\rrtwocell<\omit>{<-2>\;\alpha_2}
\ar @/^4ex/[rr]|-{\object@{|}}^-C
\ar[rr]|-{\object@{|}}_-{1_X}
\ar @{.>}[drr]|-{\object@{|}}_-{f_*}
&& X\ar[d]|-{\object@{|}}^-{f_*}
_{\cong\phantom{ab}} \\
Y\ar[u]|-{\object@{|}}^-{f^*}
\rtwocell<\omit>{<-3.5>\fepsilon}
\ar[rr]|-{\object@{|}}_-{1_Y}
 && Y}
\xymatrix{ = \\ \hole}
\xymatrix @C=1.1in @R=.4in
{X\ar[r]|-{\object@{|}}^-C  & X \ar[d]|-{\object@{|}}^-{f_*} \\
Y\ar[u]|-{\object@{|}}^-{f^*}\ar @/_4ex/[r]|-{\object@{|}}_-{1_Y} \rtwocell<\omit>{<2>\;\beta_2}
\ar[r]|-{\object@{|}}^-D 
\rtwocell<\omit>{<-5>\psi} & Y.}
\end{gather*}
Clearly $\Coalg F$ contains $\ca{V}$-$\B{Cocat}$ as a full subcategory: the morphisms satisfy the same axioms
\cref{cofunctaxioms} for the mate of $\psi$, and objects are $\ca{V}$-graphs equipped with cocomposition and coidentities arrows
that don't necessarily satisfy coassociativity and counity. 

Since $\ca{V}$-$\B{Cocat}$ is cocomplete by \cref{VCocatcocomplete}, we only need to show that it is accessible.
It is enough to express it as an equifier of a family of pairs of natural transformations between
accessible functors. First of all, $F$ preserves all filtered colimits: take a colimiting cocone for a small filtered category $\ca{J}$
\begin{displaymath}
\big((G_j,X_j)\xrightarrow{\;(\lambda_j,\tau_j)\;}(G,X)\,|\,j\in\ca{J}\big)
\end{displaymath}
in $\ca{V}$-$\B{Grph}$ for a diagram like (\ref{diagraminVgraph}) constructed in that proof, \emph{i.e.} $(\tau_j:X_j\to X)$ is colimiting
in $\B{Set}$ and $(\lambda_j:(\tau_j)_*C_j(\tau_j)^*\to C)$ is colimiting in $\ca{V}$-$\Mat(X,X)$. We require its image under $F$
\begin{equation}\label{imageunderF}
 F(\lambda_j,\tau_j):
(G_j\circ G_j,X_j)\times(1_{X_j},X_j)\to
(G\circ G,X)\times(1_X,X)
\end{equation}
to be colimiting in $\ca{V}$-$\B{Grph}$. For its first part \cref{Fonarrows}, we can deduce that 
\begin{displaymath}
(\tau_j)_*\circ G_j\circ(\tau_j)^*\circ(\tau_j)_*\circ G_j\circ(\tau_j)^*
\xrightarrow{\;\lambda_j*\lambda_j\;}G\circ G
\end{displaymath}
is a colimit in $(\ca{V}$-$\Mat(X,X),\circ,1_X)$, as the tensor product (horizontal composite) of two colimiting cocones.
Pre-composing this with the unit 
\begin{displaymath}
1*\feta*1:(\tau_j)_*\circ G_j\circ1_{X_j}\circ G_j\circ(\tau_j)^*
\to(\tau_j)_*\circ G_j\circ(\tau_j)^*\circ(\tau_j)_*\circ G_j\circ(\tau_j)^*
\end{displaymath}
still gives a colimiting cocone: if we take components in $\ca{V}$
of the respective 2-cells in $\ca{V}$-$\Mat$, this comes down to showing that 
the inclusion
\begin{displaymath}
 \sum^{\scriptscriptstyle{\stackrel{\tau_ju=x'}{\tau_jw=x}}}_{z\in X_j}
{G_j(u,z)\otimes G_j(z,w)}\hookrightarrow
\sum^{\scriptscriptstyle{\stackrel{\tau_ju=x'}{\tau_jw=x}}}_{\tau_ja=\tau_jb}
{G_j(u,a)\otimes G_j(b,w)}
\end{displaymath}
for any two fixed $x,x'\in X$, where $u,w,a,b\in X_j$, does not 
alter the colimit. One way of showing this is by considering
the following discrete opfibrations over the filtered shape 
$\ca{J}$:
\begin{align*}
\ca{L}&=\{(j,a,b)\,|\,j\in\ca{J},a,b\in X_j,\tau_ja=\tau_jb\} \\
\ca{M}&=\{(j,z)\,|\,j\in\ca{J},z\in X_j\}
\end{align*}
where for example the arrows $(j,a,b)\to(j',a',b')$ in $\ca{L}$ 
are determined by arrows $\theta:j\to j'$ 
in $\ca{J}$ such that $a'=f_\theta(a)$ and $b'=f_\theta(b)$
(the function $f_\theta:X_j\to X_{j'}$ 
is the image of the diagram (\ref{diagraminVgraph}) in $\B{Set}$).
We can now define diagrams of shape $\ca{L}$ and 
$\ca{M}$ in $\ca{V}$
\begin{displaymath}
 L:\xymatrix@R=.02in{\ca{L}\ar[r] & \ca{V}\qquad\qquad \\
(j,a,b)\ar@{|->}[r] & G_j(u,a)\otimes G_j(b,w)}\qquad
M:\xymatrix@R=.02in{\ca{M}\ar[r] & \ca{V}\qquad\qquad \\
(j,z)\ar@{|->}[r] & G_j(u,z)\otimes G_j(z,w)}
\end{displaymath}
and appropriately on morphisms. The colimits for these diagrams in 
$\ca{V}$, taking into account that the fibres 
are discrete categories, are
\begin{align*}
 \colim L & \cong\colim_j\sum_{\tau_ja=\tau_jb}
{G_j(u,a)\otimes G_j(b,w)} \\
\colim M & \cong\colim_j\sum_{z\in X_j}
{G_j(u,z)\otimes G_j(z,w)}.
\end{align*}
Finally, notice that there exists a functor 
$T:\ca{M}\to\ca{L}$ mapping each $(j,z)$ to $(j,z,z)$
and making the triangle
\begin{displaymath}
 \xymatrix
{\ca{M}\ar[rr]^-T\ar[dr]_-M && \ca{L}\ar[dl]^-L \\
&\ca{V}&}
\end{displaymath}
commute. Due to the construction of filtered colimits in $\B{Set}$,
it is not hard to show that the slice category $\big((j,z,w)\downarrow T\big)$
is non-empty and connected. Hence $T$ is a final 
functor and we can restrict the diagram on $\ca{L}$
to $\ca{M}$ without changing the colimit, as claimed.

For the second part of the diagram \cref{Fonarrows}, it is enough to show that
\begin{displaymath}
\xymatrix @C=.5in @R=.1in
{\rrtwocell<\omit>{<4>\fepsilon}  &
X_j\ar[dr]|-{\object@{|}}^-{(\tau_j)_*} & \\
X\ar@/_2ex/[rr]|-{\object@{|}}_-{1_X} 
\ar[ur]|-{\object@{|}}^-{(\tau_j)^*} && Y}
\end{displaymath}
is a colimiting cocone in $\ca{V}$-$\Mat(X,X)$, for the diagram mapping each $j$ to 
\begin{displaymath}
\xymatrix{X\ar[r]|-{\object@{|}}^-{(\tau_j)^*} & 
X_j\ar[r]|-{\object@{|}}^-{1_{X_j}} & 
X_j\ar[r]|-{\object@{|}}^-{(\tau_j)_*} & X.} 
\end{displaymath}
This can be established by first verifying that $\fepsilon$ is a cocone,  and then that it has the required universal property. 
We have thus shown that the cocone (\ref{imageunderF}) is indeed colimiting, hence $F$ is an accessible functor as required.

Since $\ca{V}$-$\B{Grph}$ is locally presentable and the endofunctor $F$ preserves filtered colimits, $\Coalg F$ is a
locally presentable category by \cref{functoralgebrasprops} and the forgetful functor $\overline{V}:\Coalg F\to\ca{V}$-$\B{Grph}$
creates all colimits. We consider the following pairs of transformations between functors from $\Coalg F$ to $\ca{V}$-$\B{Grph}$:
\begin{displaymath}
\phi^1,\psi^1:\overline{V}\Rightarrow FF\overline{V},\quad
\phi^2,\psi^2:\overline{V}\Rightarrow (-\circ 1_X)\overline{V},\quad
\phi^3,\psi^3:\overline{V}\Rightarrow \overline{V}(-\circ 1_X)
\end{displaymath}
with natural components
\begin{align*}
\phi^1_{(C,X)}:
\xymatrix @R=.1in
{X\ar @/_/[dr]|-{\object@{|}}_-C
\ar @/^3ex/[drr]|-{\object@{|}}^-C
\ar @/^4ex/[rrr]|-{\object@{|}}^-C
\drrtwocell<\omit>{<+.3>\;\alpha_1}
&\drrtwocell<\omit>{<-1.3>\;\alpha_1} && X, \\
& X\ar[r]|-{\object@{|}}_-C &
X\ar @/_/[ur]|-{\object@{|}}_-C &}&\quad
\psi^1_{(C,X)}:
\xymatrix @R=.1in
{X\ar @/_/[dr]|-{\object@{|}}_-C
\ar @/^4ex/[rrr]|-{\object@{|}}^-C
&&& X \\
\urrtwocell<\omit>{<-1.3>\;\alpha_1} & 
X\ar[r]|-{\object@{|}}_-C 
\ar @/^3ex/[urr]|-{\object@{|}}^-C 
\urrtwocell<\omit>{<+.3>\;\alpha_1} &
X\ar @/_/[ur]|-{\object@{|}}_-C &} \\
\phi^2_{(C,X)}:
\xymatrix @C=.6in @R=.05in
{X\drrtwocell<\omit>{<-2.3>\;\alpha_1}
\drtwocell<\omit>{<-0.4>\;\alpha_2}
\ar @/_2ex/[dr]|-{\object@{|}}_-{1_X}
\ar @/^3ex/[dr]|-{\object@{|}}^-C 
\ar @/^4ex/[rr]|-{\object@{|}}^-C
&& X, \\
& X \ar @/_/[ur]|-{\object@{|}}_-C &}&\quad
\psi^2_{(C,X)}:
\xymatrix @R=.05in @C=.6in
{X \ar@/_/[dr]|-{\object@{|}}
_-{1_X} \ar @/^3ex/[rr]|-{\object@{|}}^-C
\rrtwocell<\omit>{'\cong} && X \\
& X\ar@/_/[ur]|-{\object@{|}}_-{C}  &} \\
\phi^3_{(C,X)}:
\xymatrix @R=.05in @C=.6in
{X \ar @/_/[dr]|-{\object@{|}}_-C
\ar @/^4ex/[rr]|-{\object@{|}}^-C
&& X, \\
\urrtwocell<\omit>{<-2.3>\;\alpha_1}
& X \ar @/_2ex/[ur]|-{\object@{|}}_-{1_X} 
\ar @/^3ex/[ur]|-{\object@{|}}^-C 
\urtwocell<\omit>{<-0.4>\;\alpha_2} &}&\quad
\psi^3_{(C,X)}:
\xymatrix @R=.05in @C=.6in
{X \ar@/_/[dr]|-{\object@{|}}
_-{C} \ar @/^3ex/[rr]|-{\object@{|}}^-C
\rrtwocell<\omit>{'\cong} && X. \\
 & X\ar@/_/[ur]|-{\object@{|}}_-{1_X} &}
\end{align*}
The full subcategory of $\Coalg F$ spanned by those objects $(C,X)$ which satisfy $\phi^i_{(C,X)}=\psi^i_{(C,X)}$
is precisely the category of $\ca{V}$-cocategories by \cref{Vfunctaxioms2cells}, thus
\begin{displaymath}
 \B{Eq}((\phi^i,\psi^i)_{i=1,2,3})=\ca{V}\textrm{-}\B{Cocat}.
\end{displaymath}
Since all categories and functors involved are accessible, $\ca{V}$-$\B{Cocat}$ is accessible. 
\end{proof}

The fact that $\ca{V}$-$\B{Cocat}$ is a locally presentable category is very useful for the proof of existence of various adjoints, as seen below.
\begin{prop}\label{VCocatcomonadic}
Suppose $\ca{V}$ is a locally presentable monoidal category such that $\otimes$ preserves colimits in both entries.
The forgetful functor $\tilde{U}:\ca{V}$-$\B{Cocat}\to\ca{V}$-$\B{Grph}$ is comonadic.
\end{prop}

\begin{proof}
By \cref{cofreeVcocatfunctor} the forgetful $\tilde{U}$ has a right adjoint, namely the cofree $\ca{V}$-cocategory functor $\tilde{R}$.
By adjusting \cref{diagforComoncomonadicity}, consider the commutative
\begin{displaymath}
\xymatrix @C=.7in @R=.4in
{\ca{V}\textrm{-}\B{Cocat}\ar[dr]_-{\tilde{U}}
\ar@{^(->}[r]^-{\iota} & \Coalg F 
\ar[d]^-{\overline{V}} \\
& \ca{V}\textrm{-}\B{Grph}}
\end{displaymath}
where the top functor is the inclusion in the functor coalgebra category as described above, and the
respective forgetful functors discard the structures maps $\alpha$ of the coalgebras. By \cref{functoralgebrasprops}
$\overline{V}$ creates equalizers of split pairs,
so it is enough to show that the inclusion $\iota$ also creates equalizers of split pairs, since we already have
$\tilde{U}\dashv\tilde{R}$.
Both $\ca{V}$-$\B{Cocat}$ and $\ca{V}$-$\B{Grph}$ are locally presentable categories so in particular complete,
and it is easy to see that $\iota$ preserves and reflects, thus creates, all limits. Hence $\tilde{U}$ satisfy the conditions of 
Precise Monadicity Theorem and the result follows.
\end{proof}

\begin{prop}\label{VCocatclosed}
Suppose that $\ca{V}$ is a locally presentable symmetric monoidal closed category. Then the category of $\ca{V}$-cocategories is symmetric
monoidal closed as well.
\end{prop}

\begin{proof}
If $\otimes\colon\VCocat\times\VCocat\to\VCocat$ is the symmetric monoidal structure of $\ca{V}$-$\B{Cocat}$ described earlier,
we can form the commutative
\begin{displaymath}
\xymatrix @C=.7in @R=.5in
{\ca{V}\textrm{-}\B{Cocat}\ar[r]^-{-\otimes D_Y}\ar[d]_-{\tilde{U}} & \ca{V}\textrm{-}\B{Cocat}\ar[d]^-{\tilde{U}} \\
\ca{V}\textrm{-}\B{Grph}\ar[r]_-{-\otimes\tilde{U}(D_Y)} & \ca{V}\textrm{-}\B{Grph}}
\end{displaymath}
where the comonadic $\tilde{U}$ creates all colimits and the bottom arrow preserves them by the monoidal closed structure of $\VGrph$,
\cref{VGrphclosed}. Therefore $(-\otimes D_Y)$ preserves colimits, and \cref{Kellyadj} ensures it has a right adjoint
since $\ca{V}$-$\B{Cocat}$ is a locally presentable category. If we call it $\HOM(D_Y,-)$, we obtain a parametrized adjunction
\begin{displaymath}
\xymatrix @C=.8in
{\ca{V}\textrm{-}\B{Cocat} \ar @<+.8ex>[r]^-{-\otimes D_Y}\ar@{}[r]|-{\bot}
& \ca{V}\textrm{-}\B{Cocat}\ar @<+.8ex>[l]^-{\HOM(D_Y,-)}}
\end{displaymath}
which exhibits the uniquely induced $\HOM\colon\VCocat^\op\times\VCocat\to\VCocat$ as its internal hom.
\end{proof}

\subsection{Enrichment of \texorpdfstring{$\ca{V}$}{V}-categories in \texorpdfstring{$\ca{V}$}{V}-cocategories}
\label{enrichmentofVcatsinVcocats}

Having described the categories of $\ca{V}$-graphs, $\ca{V}$-categories and $\ca{V}$-cocategories in terms of $\ca{V}$-matrices,
and specified some of their categorical properties relatively to limits and colimits, local presentability and monoidal closed structure,
we are now in position to explore an enrichment relation (\cref{VCatenrichedinVCocat})
as a many-object generalization of \cref{monoidenrichment} for monoids
and comonoids. In particular, viewing $\VCat$ and $\VCocat$ as monads and comonads in the locally symmetric monoidal closed fibrant double
category $\VMMat$, the desired result will follow from the relevant development in \cref{locallyclosedmonoidaldoublecats}.

Supposed that $\ca{V}$ is a braided monoidal closed category with products and coproducts.
Recall that the locally closed monoidal structure of $\VMMat$ 
by \cref{VMatlocallymoidalclosed}
is given by a lax double functor $$H=(H_0,H_1)\colon\VMMat^\op\times\VMMat\to\VMMat$$ where $H_0$ is the exponentiation
in $\Set$ and $H_1$ is defined in \cref{H1functor}. This induces a functor $\Mon H$ \cref{MonHdouble}
as the restriction of $H_1^\bullet$ between the category of monads and comonads, which in this context becomes
\begin{equation}\label{defK}
K\colon\VCocat^\op\times\VCat\longrightarrow\VCat
\end{equation}
mapping a $\ca{V}$-cocategory and a $\ca{V}$-category $(C_X,B_Y)$ to $K(C,B)_{Y^X}$ given by
\begin{displaymath}
K(C,B)(s,k)=\prod_{x,x'}[G(x,x'),H(sx,kx')]\textrm{ for }s,k\in Y^X.
\end{displaymath}
Of course this comes from the induced internal hom in $\VGrph$, \cref{VGrphclosed}, with the extra structure of a category coming from
the more general \cref{MonFdouble}. 
Explicitly, for each triple $s,k,t\in Y^X$, the composition $M:K(C,B)(s,k)\otimes K(C,B)(k,t)\to K(C,B)(s,t)$ for $\ca{K}(C,B)$ is an arrow
\begin{displaymath}
\prod_{a,a}{[C(a,a'),B(sa,ka')]}\otimes\prod_{b,b'}{[C(b,b'),B(kb,tb')]}\to\prod_{c,c'}{[C(c,c'),B(sc,tc')].}
\end{displaymath}
This is defined via its adjunct under the tensor-hom adjunction
\begin{displaymath}
\xymatrix @C=.25in
{\prod\limits_{a,a'}{[C_{a,a'},B_{sa,ka'}]\otimes\prod\limits_{b,b'}[C_{b,b'},B_{kb,tb'}]\otimes 
C(c,c')}\ar[d]_-{1\otimes\Delta_{c,c'}}\ar@{-->}[r] & B_{sc,tc'}\\
\prod\limits_{a,a'}[C_{a,a'},B_{sa,ka'}]\otimes\prod\limits_{b,b'}[C_{b,b'},B_{kb,tb'}]
\otimes\sum\limits_{c''}C_{c,c''}\otimes C_{c'',c'}\ar[d]_-{\cong} &\\
\sum\limits_{c''}\prod\limits_{a,a'}[C_{a,a'},B_{sa,ka'}]\otimes C_{c,c''}\otimes
\prod\limits_{b,b'}[C_{b,b'},B_{kb,tb'}]\otimes C_{c',c''} \ar[d]_-{\pi_{c,c''}\otimes1\otimes\pi_{c'',c'}\otimes1} & \\
\sum\limits_{c''}[C_{c,c''},B_{sc,kc''}]\otimes C_{c,c''}\otimes[C_{c'',c'},B_{kc'',tc'}]\otimes C_{c'',c'}
\ar[r]_-{\mathrm{ev}\otimes\mathrm{ev}} & \sum\limits_{c''}B_{sc,kc''}\otimes B_{kc'',tc'} \ar[uuu]_-{M_{sc,tc'}}}
\end{displaymath}
for fixed $c,c'$.
The identities for each object $s\in Y^X$ are arrows
\begin{displaymath}
\eta_d:I\to K(C,B)(d,d)=\prod_{a,a'\in X}{[C(a,a'),B(sa,sa')]}
\end{displaymath}
which correspond uniquely for fixed $a=a'\in X$ 
to the composite
\begin{displaymath}
\xymatrix @R=.3in
{I\otimes C_{a,a}\ar @{-->}[rrr]\ar @/_/[dr]_-{1\otimes\epsilon_{a,a}} &&& B_{sa,sa}. \\
& I\otimes I\ar[r]_-{r_I} & I\ar @/_/[ur]_-{\eta_{sa,sa}} &}
\end{displaymath} 

In order to apply results from the theory of actions and enrichment as presented in \cref{sec:actionenrich},
we need to realize the functor $K$ as an action of $\VCocat$ on $\VCat$. Due to \cref{doubleMonHaction}, we can deduce this in
a straightforward way.

\begin{prop}\label{Kaction}
If $\ca{V}$ is a symmetric monoidal closed category with products and coproducts, the functor $K$ (\ref{defK}) is an action
of the symmetric monoidal $\VCocat^\op$, and so is its opposite $K^\op\colon\VCocat\times\VCat^\op\to\VCat^\op.$
\end{prop}

What is left to show is that this action $K^\op$ has a parametrized adjoint, which will induce the enrichment of the category on
which the monoidal category acts. This can be deduced when $\ca{V}$ is furthermore locally presentable.

\begin{prop}\label{Texistence}
Suppose that $\ca{V}$ is a locally presentable symmetric monoidal closed category.
The functor $K^\op$ has a parametrized adjoint
\begin{equation}\label{defT}
T:\ca{V}\textrm{-}\B{Cat}^\op\times\ca{V}\textrm{-}\B{Cat}\longrightarrow\ca{V}\textrm{-}\B{Cocat},
\end{equation}
given by adjunctions $K(-,B_Y)^\op\dashv T(-,B_Y)$ for every $\ca{V}$-category $B_Y$.
\end{prop}

\begin{proof}
If $H_1^\bullet$ is the internal hom of $\VGrph$ as in \cref{VGrphclosed}, we can form a square which commutes by definition of $K$
\begin{displaymath}
\xymatrix @C=1in @R=.5in
{\VCocat\ar[r]^-{K(-,B_Y)^\op}\ar[d]_-{\tilde{U}} & \VCat^\op \ar[d]^-{\tilde{S}} \\
\VGrph\ar[r]_-{H_1^\bullet(-,\tilde{S}B_Y)^\op} & \VGrph^\op}
\end{displaymath}
where the left and right legs create all colimits by \cref{VCatmonadic} and \ref{VCocatcomonadic} and the bottom arrow preserves all colimits by 
$H_1^\bullet(-,G_Y)^\op\dashv H_1^\bullet(-,G_Y)$ in any monoidal closed category. Therefore $K(-,B_Y)^\op$ is cocontinuous, and since its domain
is locally presentable by \cref{VCocatlocpresent}, \cref{Kellyadj} provides adjunctions
$K(-,B_Y)^\op\dashv T(-,B_Y):\VCat^\op\to\VCocat$ for all $\ca{V}$-categories $B_Y$; this uniquely induces \cref{defT}.
\end{proof}

The functor $T$ which generalizes the universal measuring comonoid functor \cref{Sweedlerhom} is called
the \emph{generalized Sweedler hom}. Morever, the functor $K^\op$ fixed in the second variable has also a right adjoint,
which generalizes the Sweedler product \cref{Sweedlerprod}.

\begin{lem}
There is an adjunction $K(C_X,-)^\op\dashv(C_X\triangleright-)^\op$, for any $\ca{V}$-cocategory $C_X$.
\end{lem}

\begin{proof}
Consider the commutative diagram
\begin{displaymath}
\xymatrix @C=.9in @R=.5in
{\VCat\ar[r]^-{K(C_X,-)}\ar[d]_-{\tilde{S}} & \VCat\ar[d]^-{\tilde{S}} \\
\ca{V}\textrm{-}\B{Grph}\ar[r]_-{H_1^\bullet(\tilde{U}C_X,-)} & \ca{V}\textrm{-}\B{Grph}}
\end{displaymath}
where $\tilde{S}$ is the monadic forgetful functor and the locally presentable category $\ca{V}$-$\B{Cat}$
has all coequalizers. Thus by Dubuc's Adjoint Triangle Theorem \cite{AdjointTriangles},
the existence of a left adjoint of the bottom arrow by monoidal closedness of $\VGrph$ implies the existence of a left adjoint 
$(C_X\triangleright-)$ of the top arrow.
\end{proof}

The induced functor of two variables $\triangleright:\VCocat\times\VCat\to\VCat$ is called the \emph{generalized Sweedler product}.
Now \cref{MndenrichedCmnd} applies to provide the desired enrichment,
via the action of the symmetric monoidal closed category $\VCocat$ (\cref{VCocatclosed}).

\begin{thm}\label{VCatenrichedinVCocat}
Suppose $\ca{V}$ is a symmetric monoidal closed category which is locally presentable, and $T$ is the generalized Sweedler hom functor.
\begin{enumerate}
\item $\VCat^\op$ is enriched in $\VCocat$, tensored and cotensored, with hom-objects $\underline{\VCat^\op}(A_X,B_Y)=T(B_Y,A_X).$
\item $\VCat$ is a tensored and cotensored $(\VCocat)$-enriched category, with
\begin{displaymath}
\underline{\VCat}(A_X,B_Y)=T(A_X,B_Y)
\end{displaymath}
cotensor product $K(C,B)_{Y^Z}$ and tensor product $C_Z\triangleright A_X$, for any $\ca{V}$-co\-ca\-te\-go\-ry $C_Z$ and 
any $\ca{V}$-categories $A_X,B_Y$.
\end{enumerate}
\end{thm}

The final goal is to combine the above enrichment with the (op)fibrations that these categories form,
and characterize them as enriched fibrations, \cref{enrichedfibration}. First of all, the fibration
$P\colon\VCat\to\Set$ (\cref{VCatfibred}) as well as the opfibration $W\colon\VCocat\to\Set$ (\cref{VCocatopfibred})
are both monoidal by \cref{monadscomonadsmonoidalfibr}, due to their expression as categories of monads and comonads in $\VMMat$;
they inherit their symmetry from $\ca{V}$.
\cref{MndfibredenrichedCmnd} applied to $\VMMat$ will give the required result.

First of all, we need to show that $K^\op$ preserves cocartesian liftings, or equivalently $K$ preserves cartesian
liftings between the fibrations 
\begin{equation}\label{fibredaction}
\xymatrix @C=.8in @R=.5in
{\VCocat^\op\times\VCat\ar[r]^-K\ar[d]_-{W^\op\times P} & \VCat\ar[d]^-P \\
\Set^\op\times\Set\ar[r]_-{(-)^{(-)}} & \Set.}
\end{equation}
Using the canonical (co)cartesian liftings \cref{canonicalcartesianlift} for any Grothendieck fibration,
i.e. $\Cocart(f,C)=(1_{f_*Cf^*},f)\colon B\to f_*Cf^*$ for a cocategory $C_X$ and $f\colon X\to Z$,
$\Cart(g,B)=(1_{g^*Bg_*},g)\colon g^*Bg_*\to B$ for a category $B_Y$ and $g\colon W\to Y$,
we want to deduce that $K$ maps them to the chosen cartesian lifting
\begin{displaymath}
\xymatrix @C=.6in @R=.3in
{K(f_*Cf^*,g^*Bg_*)\ar[rr]^{\qquad K(\Cart(f,C),\Cart(g,B))}
\ar @{-->}[d]_-{\cong} && K(C,B)\ar @{.>}[dd] &\\
(g^f)^*K(C,B)(g^f)_*\ar[urr]_-{\quad\Cart(g^f,K(C,B))}\ar @{.>}[d]
&&& \textrm{in }\VCat \\
W^Z\ar[rr]_-{g^f} && Y^X & \textrm{in }\Set}
\end{displaymath}
Using \cref{machine1,machine2}, we can initially compute
\begin{align*}
K(f_*Cf^*,g^*Bg_*)_{s,k} &=\prod_{z,z'\in Z}[(f_*Cf^*)_{z,z'},(g^*Bg_*)_{sz,kz'}] \\
\phantom{A} &\cong\prod_{z,z'\in Z}[\sum_{fx=z}^{fx'=z'}C_{x,x'},B_{gsz,gkz'}], \\
\left((g^f)^*K(C,B)(g^f)_*\right)_{s,k} &=K(C,B)_{gsf,gsk}=\prod_{x,x'\in X}[C_{x,x'},B_{gsfx,gsfx'}]
\end{align*}
which are isomorphic since the internal hom maps sums to products in the first variable,
and the triangle commutes by also applying $K$ to the maps. 

The only thing left to show is that the opposite of \cref{fibredaction} has a general lax parametrized opfibred adjoint,
i.e. the opfibred 1-cell for any $\ca{V}$-category $B_Y$
\begin{displaymath}
\xymatrix @C=.8in @R=.4in
{\VCocat\ar[r]^-{K(-,B_Y)^\op}\ar[d]_-{W} & \VCat^\op\ar[d]^-{P^\mathrm{op}} \\
\B{Set}\ar[r]_-{{Y^{(-)}}^\op} & \B{Set}^\mathrm{op}}
\end{displaymath}
has a lax opfibred adjoint; this will be deduced from \cref{totaladjointthm}.
Indeed, there is an adjunction between the base categories
\begin{displaymath}
\xymatrix @C=.5in
{\B{Set} \ar @<+.8ex>[r]^-{{Y^{(-)}}^\op}
\ar@{}[r]|-{\bot} & 
\B{Set}^\op\ar @<+.8ex>[l]^-{Y^{(-)}}}
\end{displaymath}
since $\Set$ is cartesian monoidal closed, with counit $\varepsilon$. Moreover, we need to show is that the composite \cref{specialfunctor}
between the fibres 
\begin{displaymath}
\ca{V}\textrm{-}\B{Cocat}_{Y^Z}\xrightarrow{\;K_{Y^Z}(-,B_Y)^\op\;}\ca{V}\textrm{-}\B{Cat}_{Y^{Y^Z}}^\mathrm{op}
\xrightarrow{\quad(\varepsilon_Z)_!\quad}\ca{V}\textrm{-}\B{Cat}_Z^\mathrm{op}
\end{displaymath}
has a right adjoint. We can rewrite it as
\begin{displaymath}
\xymatrix @C=1.2in @R=.5in
{\Comon(\VMat(Y^Z,Y^Z))\ar@{-->}[dr]\ar[r]^-{\Mon(\Hom(-,B_Y))^\op} &
\Mon(\ca{V}\text{-}\Mat({Y^Y}^Z,{Y^Y}^Z))^\op\ar[d]^-{(\varepsilon)^*\circ-\circ(\varepsilon)_*} \\
& \Mon(\ca{V}\text{-}\Mat(Z,Z))^\op}
\end{displaymath}
where the top functor is \cref{Hom_} between the categories of monoids (as a restriction of $\Mon H$ on globular 2-cells)
and the side functor is the reindexing functor for the fibration $P$, \cref{VCatfibred}.
By \cref{cofreecomonVMat}, the domain $\Comon(\ca{V}\text{-}\Mat(Y^Z,Y^Z))$ is a locally presentable category for any set $Y^Z$.
Moreover, the following commutative
\begin{displaymath}
\xymatrix @C=1.3in
{\Comon(\ca{V}\textrm{-}\B{Mat}(Y^Z,Y^Z))
\ar[r]^-{\Mon(\Hom(-,B_Y)^\op} \ar[d]_-U &
\Mon(\ca{V}\textrm{-}\B{Mat}({Y^Y}^Z,{Y^Y}^X))^\op \ar[d]^-{S^\op} \\
\ca{V}\textrm{-}\B{Mat}(Y^Z,Y^Z) 
\ar[r]_-{\Hom(-,B_Y)^\op} &
\ca{V}\textrm{-}\B{Mat}({Y^Y}^Z,{Y^Y}^Z)^\op}
\end{displaymath}
for a fixed $\ca{V}$-category $B_Y$ shows that the top arrow is cocontinuous: $U$ and $S^\op$ are comonadic by the same \cref{cofreecomonVMat}
and the bottom arrow is the cocontinuous internal hom of $\VGrph$ (\cref{VGrphclosed})
restricted between the cocomplete fibres. Finally, composing with the
companion and conjoint of $\varepsilon$ on either side always preserves colimits, since tensoring does (see \cref{propVMat}). Therefore,
\cref{Kellyadj} establishes an adjunction
\begin{displaymath}
\xymatrix @C=1in
{\ca{V}\textrm{-}\B{Cocat}_{Y^Z}\ar @<+.8ex>[r]^-{(\varepsilon_Z)_!\circ K(-,B_Y)^\op}\ar@{}[r]|-\bot
& \ca{V}\textrm{-}\B{Cat}_Z^\op\ar @<+.8ex>[l]^-{T_0(-,B_Y)}}
\end{displaymath}
between the fibre categories, enough by \cref{totaladjointthm}
to induce a right adjoint of $K(-,B_Y)^\op$ between the total categories such that 
\begin{displaymath}
\xymatrix @R=.5in @C=.8in 
{\ca{V}\textrm{-}\B{Cocat} \ar @<+.8ex>[r]^-{K(-,B_Y)^\mathrm{op}}
\ar@{}[r]|-\bot \ar[d]_-{W} &
\ca{V}\textrm{-}\B{Cat}^\mathrm{op}\ar @<+.8ex>[l]^-{T(-,B_Y)}
\ar[d]^-{P^\mathrm{op}} \\ 
\B{Set} \ar@<+.8ex>[r]^-{{Y^{(-)}}^\mathrm{op}} 
\ar@{}[r]|-\bot &
\B{Set}^\mathrm{op} \ar @<+.8ex>[l]^-{Y^{(-)}}}
\end{displaymath}
is a general lax opfibred adjunction. The assumptions of \cref{MndfibredenrichedCmnd} are now satisfied and the result follows.

\begin{thm}
Suppose $\ca{V}$ is a symmetric monoidal closed category, which is locally presentable.
The opfibration $\VCat^\op\to\Set^\op$ as well as the fibration $\VCat\to\Set$ are enriched in the symmetric monoidal opfibration
$\VCocat\to\Set$.
\end{thm}

Notice that the total parametrized adjoint $T\colon\VCat^\op\times\VCat\to\VCocat$ of $K$ obtained as above is isomorphic
to \cref{defT}, but the fibred approach provided with the extra information that the underlying set of objects of some $T(A_X,B_Y)$
is precisely $Y^X$ in a straightforward way.

\subsection*{Acknowledgements}
I would like to thank Martin Hyland for the insightful ideas that led to this work, as well
as Ignacio L\'opez Franco for helpful discussions that affected this development.
Major parts of this work were accomplished during my PhD \cite{PhDChristina}, for which period I gratefully acknowledge
the financial support by Trinity College and DPMMS, University of Cambridge, as well as Propondis and Leventis Foundations.
This paper was written within the framework the ARC 
Consolidator project ``Hopf algebras and the symmetries of non-commutative space'', sponsored by F\'ed\'eration Wallonie-Bruxelles,
during my postdoc at Universit{\'e} Libre de Bruxelles with Joost Vercruysse.

\bibliographystyle{alpha}
\bibliography{myreferences}

\end{document}